\documentclass{tac}
\usepackage{amsmath,amssymb, amsfonts}
\usepackage[all]{xy}
\usepackage{verbatim}
\usepackage{enumerate}
\usepackage{MnSymbol}
\newtheorem{thm}{Theorem}[section]
\newtheorem{cor}[thm]{Corollary}
\newtheorem{con}[thm]{Conjecture}
\newtheorem{example}[thm]{Example}
\newtheorem{lem}[thm]{Lemma}
\newtheorem{obs}[thm]{Observation}

\newtheorem{prop}[thm]{Proposition}

\theoremstyle{definition}
\newtheorem{defn}[thm]{Definition}
\newtheorem{rem}[thm]{Remark}

\DeclareFontFamily{U}{rsf}{} \DeclareFontShape{U}{rsf}{m}{n}{
  <5> <6> rsfs5 <7> <8> <9> rsfs7 <10->  rsfs10}{}
\DeclareMathAlphabet{\mathscr}{U}{rsf}{m}{n}

\renewcommand{\imath}{\sqrt{-1}}

\DeclareMathOperator{\IB}{IB} 
 
\DeclareMathOperator{\Hom}{Hom}

\DeclareMathOperator{\id}{id}

\DeclareMathOperator{\colim}{colim}

\DeclareMathOperator{\im}{im}
\DeclareMathOperator{\coim}{coim}
\DeclareMathOperator{\coker}{coker}
\DeclareMathOperator{\Tot}{Tot}

\DeclareMathOperator{\Lin}{Lin}
\DeclareMathOperator{\Fun}{Fun}

\DeclareMathOperator{\straightK}{K}
\DeclareMathOperator{\straightN}{N}
\DeclareMathOperator{\straightD}{D}

\DeclareMathOperator{\triv}{triv}
\newcommand{\uHom}{\underline{\Hom}}

\newcommand{\ootimes}{\overline{\otimes}}
\newcommand{\wotimes}{\widehat{\otimes}}

\DeclareMathOperator{\spec}{\text{spec}}
\newcommand{\limpro}{\mathop{\lim\limits_{\displaystyle\leftarrow}}}
\newcommand{\limind}{\mathop{\lim\limits_{\displaystyle\rightarrow}}}

\newcommand{\ttSet}{\text{\bfseries\sf{Set}}}

\newcommand{\tC}{\text{\bfseries\sf{C}}}
\newcommand{\tD}{\text{\bfseries\sf{D}}}

\newcommand{\tI}{\text{\bfseries\sf{I}}}

\newcommand{\ttMod}{{\text{\bfseries\sf{Mod}}}}

\newcommand{\ttCBorn}{{\text{\bfseries\sf{CBorn}}}}

\newcommand{\ttBan}{{\text{\bfseries\sf{Ban}}}}

\newcommand{\Mod}{{\text{\bfseries\sf{Mod}}}}

\newcommand{\ttComm}{{\text{\bfseries\sf{Comm}}}}

\newcommand{\ttC}{{\tC}}
\newcommand{\ttD}{{\tD}}

\newcommand{\ttE}{{\text{\bfseries\sf{E}}}}

\newcommand{\ttPr}{{\text{\bfseries\sf{Pr}}}}

\newcommand{\ttIndBan}{{\text{\bfseries\sf{Ind(Ban)}}}}
\newcommand{\ttTc}{{\text{\bfseries\sf{Tc}}}}
\newcommand{\ttFr}{{\text{\bfseries\sf{Fr}}}}

\newcommand{\ttInd}{{\text{\bfseries\sf{Ind}}}}

\author{Oren Ben-Bassat and Kobi Kremnizer}

\title[Fr\'{e}chet modules and descent]{Fr\'{e}chet modules and descent}
\copyrightyear{2023}
\address{Oren Ben-Bassat, Department of Mathematics, University of Haifa, Haifa 3498838, Israel}
\address{Mathematical Institute,
University of Oxford,
Andrew Wiles Building,
Radcliffe Observatory Quarter,
Woodstock Road,
Oxford,
OX2 6GG, England}

\address{Department of Mathematics, \\ University of Haifa, \\
 Haifa 3498838, Israel\\[5pt]
 Mathematical Institute, \\
University of Oxford, \\
Andrew Wiles Building,\\
Radcliffe Observatory Quarter, \\
Woodstock Road, \\
Oxford, 
OX2 6GG, England\\
}
\eaddress{ben-bassat@math.haifa.ac.il \CR Yakov.Kremnitzer@maths.ox.ac.uk}
\begin{document}
\maketitle 
\begin{abstract} Motivated by classical functional analysis results over the complex numbers and results in the bornological setting over the complex numbers of R. Meyer, we study several aspects of the study of Ind-Banach modules over Banach rings. This allows for a synthesis of some aspects of homological algebra and functional analysis. This includes a study of nuclear modules and of modules which are flat with respect to the projective tensor product. 
We also study metrizable and Fr\'{e}chet Ind-Banach modules. We give explicit descriptions of projective limits of Banach rings as ind-objects. We study exactness properties of the projective tensor product with respect to kernels and countable products. As applications, we describe a theory of quasi-coherent modules in Banach algebraic geometry. We prove descent theorems for quasi-coherent modules in various analytic and arithmetic contexts and relate them to well known complexes of modules coming from covers.

\noindent 
\end{abstract}

\tableofcontents
\section{Introduction} 
The use of categorical and homological techniques in functional analysis has a long and complicated history which we can not adequately summarize here. This includes work of  Helemskii \cite{Hel}, Meyer \cite{MeyerDerived} \cite{M}, Cigler, Losert and Michor \cite{CLM}, Paugam \cite{Pa}, Taylor \cite{Tay}, Wengenroth \cite{Wen} and others.  We follow the approach of using the homological algebra of quasi-abelian categories of Prosmans and Schneiders \cite{SchneidersQA}, \cite{PS} generalized from the functional analysis of Banach and Ind-Banach spaces over complex numbers to general Banach rings.

Grothendieck developed the theory of nuclearity for topological vector spaces over $\mathbb{C}$. In \cite{PS} these ideas are carried over to the closely related setting of ind-Banach spaces over $\mathbb{C}$. We were able to prove analogues of these results in the setting of Ind-Banach modules over arbitrary Banach rings $R$. The definition of nuclearity we use is in Definition \ref{defn:IndBanNuc} and an equivalent characterization in Remark \ref{rem:summary}. Not having Hilbert space techniques available when working over general Banach rings, we were unable to prove that subspaces and quotients of nuclear maps are nuclear.  However, we can prove many other standard ``permanence properties" of nuclearity. We discuss countable products and coproducts in Corollary \ref{cor:product_sum_nuc} and a two out of three rule for strict short exact sequences in Lemma \ref{lem:twothreenuc} and the projective tensor product of nuclear spaces in Lemma \ref{lem:tensnuc}.  A different approach to nuclearity which works in both the Archimedean and non-Archimedean settings could be inspired by Schneider's notion (see \cite{Schneider}) of compact morphisms between Banach spaces. Corollary \ref{cor:limsprods} proves that nuclearity also ensures an interesting interaction with products of dual spaces. Following work of Prosmans and Schneiders we prove that nuclear spaces can be written in certain canonical ways in Lemmas \ref{lem:NicePres} and \ref{lem:NicePres2}. We define metrizability in Definition \ref{defn:met}. Important examples of metrizable modules are Banach or Fr\'{e}chet modules. Notice that as nuclearity of an object implies it is flat for the projective tensor product (Lemma \ref{lem:NucImpliesFlat}), one may ask what condition on an object might ensure that the projective tensor product with it commutes with countable products. This turns out to be a complete characterization of metrizability as proven in Lemmas \ref{lem:tensprod} and \ref{lem:surprise}. Therefore, in combination the properties of nuclearty and metrizability for an object imply that the projective tensor product with it commutes with countable limits (Lemma \ref{lem:tensprod}). Banach algebraic geometry and its derived versions is an approach to analytic geometry which uses geometry relative to categories of Banach spaces (or modules) in the same way that usual algebraic geometry is based on categories of abelian groups.  In particular, this philosophy applies to rigid analytic geometry \cite{BK}, overconvergent rigid geometry \cite{BB} and Stein geometry (\cite{BBK}, \cite{Pir}, \cite{ArPir}) and in these articles it was shown that the homotopy monomorphism topology specializes to conventional ones in special cases. There are also projects on derived analytic geometry \cite{BBK2} and analytic $\mathbb{F}_1$-geometry \cite{BK2}. Most of the constructions in this article are based on an arbitrary Banach ring $R$. If $R$ is a non-archimedean Banach ring (see Definition \ref{defn:n-arch}), this entire article can be separately read in two different versions, depending on whether one considers the categories $\ttInd(\ttBan_R)$ of all Banach modules or $\ttInd(\ttBan^{na}_R)$ of non-archimedean Banach modules. Therefore, in this case, notation such as symbols for limits, colimits, products and coproducts, can sometimes take on two different meanings. We have chosen to write everything with the default version being of the archimedean version. This has the appealing aspect of being completely the same for any $R$, archimedean or not. In the case that $R$ is non-archimedean, the reader who wants to work in a non-archimedan context should replace all limits and colimits in the category $\ttInd(\ttBan_R)$ by those in the category $\ttInd(\ttBan^{na}_R)$. All the proofs go through in a similar way. Given a union of subsets one often wants to describe modules on the union in terms of modules on the components together with gluing data. The subsets themselves usually must cover the space and each subset individually should have nice properties. We therefore need to translate both of these features into algebra. Our main descent results can be found in Theorem \ref{thm:Big}. To formulate this we introduce a generalization of a coherent module called a quasi-coherent module. This notion was also needed in \cite{BK, BB, BBK} where some properties of quasi-coherent modules were studied, and in this article we extend that study. We relate our results to Tate's acyclicity theorem in Lemma \ref{lem:TateAcy} and modules on Stein covers.

This article was originally motivated by a desire to make ``more categorical" the results on descent for Stein algebras from \cite{BBK} (see also \cite{ArPir}). We believe that we have succeeded in a large aspect in terms of the issues surrounding infinite products and completed tensor products and their interaction. Unfortunately, we have not been able to make categorical the Mittag-Leffler aspects which involve dense maps of algebras in the projective system and $\lim$-acyclicity. In standard complex analysis one often exhausts a Stein open subset by an increasing union of compact, convex subsets with the Noether property. Then one would like to understand how certain quasi-abelian categories of quasi-coherent modules on the Stein open are constructed as categorical limits of the similar categories on the compact subsets. More precisely, can a nice enough module over the algebra of holomorphic functions on the Stein open be determined in terms of gluing data for modules on the compact subsets? These questions also have a rich history in the non-archimedean literature, for example see work of Ardakov and Wadsley \cite{AW}, where one uses affinoids in place of compact convex subsets. In our desire for a unified approach to the archimedean and non-archimedean case, we can restate Theorem \ref{thm:Big} in this case.
Let $A_i$ be Banach rings, flat over $R$ together with a sequence of dense, nuclear, homotopy epimorphisms $\cdots \to A_2 \to A_1.$ Any quasi-coherent, metrizable, ind-Banach module $M$ flat over $R$ over $A=\lim A_i$ can be expressed as a limit in $\Mod(A)$ of a sequence $\cdots \to M_2 \to M_1$ where each $M_i$ is a nuclear, metrizable ind-Banach object of $\Mod(A_i)$ flat over $R$ and the morphisms are consistent with this action in the sense that there are isomorphisms $A_{i+1}\wotimes^{\mathbb{L}}_{A_i}M_i \cong A_{i+1}\wotimes_{A_i}M_i \cong M_{i+1}$ compatible with one another and with the maps in the sequence. Any element of $\Hom(M,N)$ in the category of ind-Banach $A$-modules is a consistent limit of elements of $\Hom(M_i, N_i)$ in the category of ind-Banach $A_i$-modules. This should have applications in non-archimedean geometry for instance in the case of analytic differential operators as appear in work of Ardakov and Wadsley (see also \cite{AB}) which are Fr\'{e}chet (and as we have shown therefore metrizable) modules which are not coherent over the functions, but which can be shown to be quasi-coherent in our definition. A version of the results in this subsection was given in \cite{BBK} over a complete valuation field but there we needed to separately prove the theorem in the archimedean and non-archimedean cases whereas in this article we provide a single proof over a Banach ring that works in the archimedean or non-archimedean case.

In future work, to be based on this article, we will discuss a new notion of the analytic specrum of the integers and its covers by homotopy monomorpisms. We will also introduce anlaytic versions of the Weil-\'{e}tale topos.

\section{Notation}
We use the notation $\lim$ instead of $\limpro$ and $\colim$ for $\limind$. The letter $R$ denotes a general Banach ring, defined in Definition \ref{defn:BanRng}. We denote categorical products by $\prod$ and categorical coproducts by $\coprod$, it should be clear in what category these take place, usually it is sufficient to consider them in the category $\ttInd(\ttBan_{\mathbb{Z}})$. Given an object $A$ in $\ttComm(\ttInd(\ttBan_{\mathbb{R}}))$, we use $\spec(A)$ to just denote the same object in the opposite category. As usual, $\mathbb{Z}_p$ denotes the $p$-adic integers, unless we are scaling the norm on $\mathbb{Z}$ with a real number in the sense of Definition \ref{defn:scaling}, this should be clear from the context. For an abelian group $A$ we use $A^{\times}$ to denote $A-\{0\}$.
\section{Some Category Theory and Its Uses in Functional Analysis and Geometry}
\subsection{Relative Algebra and Homological Algebra}
\begin{defn}\label{def:verb}
In an additive category with kernels and cokernels, a morphism $f:E\to F$ is called strict if the induced morphism 
\[\coim(f)\to \im(f)
\] is an isomorphism. Here $\im(f)$ is the kernel of the canonical map
$F \to \coker(f)$, and $\coim(f)$ is the cokernel of the canonical map $\ker (f)\to E$. An object $P$ is {\it projective} if for all strict epimorphisms $E \to F$ the associated map $\Hom(P,E) \to \Hom(P,F)$ is onto. An object $I$ is {\it injective} if for all strict monomorphisms $E \to F$ the associated map $\Hom(F,I) \to \Hom(E,I)$ is onto.  If the category is equipped with a unital symmetric monoidal structure $\ootimes$ then an object $F$ is called {\it flat} if the functor $(-)\ootimes F$ preserves strict monomorphisms. 
\end{defn}

Consider a unital, closed symmetric monoidal category $(C, \ootimes, e=\text{id}_{C})$ with finite limits and colimits (more details in \cite{BK}). The internal Hom in $C$ will be denoted $\uHom$ as it should be clear from the context what category we are working in. We will always suppress the commutativity, unitality, and associativity natural transformations from the notation. It is easy to see the following lemma.

\begin{lem}\label{lem:FlatProps}
The unit of $C$ is flat in $C$. A coproduct of objects is flat if and only if each of them is flat. The monoidal product of two flat objects is flat.
\end{lem}
Recall \cite{TV} the category $\ttComm(C)$ of commutative unital rings with respect to $(C, \ootimes, e=\text{id}_{C})$ and for any $S\in \ttComm(C)$ the category $\ttMod(S)$ of $S$-modules internal to $(C, \ootimes, e=\text{id}_{C})$. If we further assume that $(C, \ootimes, e=\text{id}_{C})$ has countable coproducts then an important construction is the symmetric ring construction which is a left adjoint to the forgetful functor $\ttComm(C) \to C$ 
\begin{equation}\label{eqn:SymAlgConst}\text{Sym}(V) = \underset{n\geq 0}\coprod V^{\ootimes n}/S_n.
\end{equation}
\subsection{Quasi-abelian categories}

\begin{defn}\label{defn:CartCoCart}
Let $\mathcal{E}$ be an additive category with kernels and cokernels. We say that $\mathcal{E}$ is quasi-abelian if it satisfies the following two conditions:
\begin{itemize}
\item In a cartesian square 
\begin{equation*}
\xymatrix{  E' \ar[r]^{f'} \ar[d] & F' \ar[d] \\
E \ar[r]_{f}& F}
\end{equation*}
if $f$ is a strict epimorphism then $f'$ is a strict epimorphism.

\item In a co-cartesian square
\begin{equation*}
\xymatrix{  E \ar[r]^{f} \ar[d] & F \ar[d] \\
E' \ar[r]_{f'}& F'}
\end{equation*}
if $f$ is a strict monomorphism then $f'$ is a strict monomorphism.

\end{itemize}
\end{defn}
A quasi-abelian category is a category where the strict monomorphisms and strict epimorphisms satisfy the conditions of a Quillen exact category. It may be useful to allow for more general Quillen exact structures (see \cite{BK2}) for instance using short exact sequences that split over $R$ but in this work we avoid this. 
\begin{defn}
Let $\ttE$ be a quasi-abelian category. Let $\straightK(\ttE)$ be its category of complexes up to homotopy. The derived category of $\ttE$ is $\straightD(\ttE)=\straightK(\ttE)/\straightN(\ttE)$ where $\straightN(\ttE)$ is the full subcategory of strictly exact sequences.
\end{defn}
Here a sequence 
\[E' \stackrel{e'}\longrightarrow E \stackrel{e''}\longrightarrow E''
\]
in a quasi-abelian category is strictly exact when the image of the first map is isomorphic to the kernel of the second, and $e'$ is strict.
\begin{lem}\label{lem:StrictEpiPreserving}
Let $\ttC$ and $\ttD$ be quasi-abelian categories. Let $L:\ttC \to \ttD$ be any functor with a right adjoint $R:\ttD \to \ttC$.  Then $L$ preserves strict epimorphisms and $R$ preserves strict monomorphisms.
\end{lem}
\proof 
Let $f: V\to W$ be a strict epimorphism in $\ttC$. Then of course $L(f)$ is an epimorphism. Because $f$ is a strict epimorphism, we have $W= \coker(\ker(f) \to V)$. Therefore, since left adjoints preserve cokernels, $L(f)$ expresses $L(W)$ as the cokernel of the morphism $L(\ker(f)) \to L(V)$. The second statement is proven in a similar way.
\endproof
\begin{defn}
Let $\ttE$ be a quasi-abelian category. Let $\straightK(\ttE)$ be its homotopy category. A morphism in 
$\straightK(\ttE)$ is called a strict quasi-isomorphism if its mapping cone is strictly exact. 
\end{defn}

\begin{defn}
Let $\ttE$ be an additive category with kernels and cokernels. An object $I$ is called injective if the functor $E\mapsto \Hom(E,I)$ is exact, i.e., for any strict monomorphism $u:E\to F$, the induced map $\Hom(F,I)\to \Hom(E,I)$ is surjective. Dually, $P$ is called projective if the functor $E\mapsto \Hom(P,E)$ is exact, i.e., for any strict epimorphism $u:E\to F$, the associated map $\Hom(P,E)\to \Hom(P,F)$ is surjective.     
\end{defn}

\begin{defn}\label{defn:enough}
A quasi-abelian category $\ttE$ has enough projectives if for any object $E$ there is a strict epimorphism $P\to E$ where $P$ is projective. A quasi-abelian category $\ttE$ has enough injectives if for any object $E$ there is a strict monomorphism $E \to I$ where $I$ is injective.
\end{defn}

\begin{defn} 
Let $\ttE$ be an additive category. An object $E$ is called:
\begin{itemize}
\item small, if 
\begin{equation}
\Hom(E,\coprod_{i\in \tI} F_i)\cong \coprod_{i\in I}\Hom(E,F_i)
\end{equation}
for any small family $(F_i)_{i\in \tI}$ of $\ttE$ whenever the coproduct on the left exists
\item tiny, if 
\begin{equation}
\Hom(E,\underset{i\in \tI}\colim F_i)\cong 
\underset{i\in \tI}\colim \Hom(E,F_i)
\end{equation}
for any filtering inductive system $\tI \to\ttE$ whenever the colimit on the left exists.
\end{itemize}
\end{defn}

\begin{defn}\label{defn:StrGen}
Let $\ttE$ be a quasi-abelian category. A strict generating set of $\ttE$ is a subset $\mathcal{G}$ of $Ob(\ttE)$ such that for any monomorphism
\begin{equation*}
m : S\to E
\end{equation*}
of $\ttE$ which is not an isomorphism, there is a morphism
\begin{equation*}
G \to E
\end{equation*}
 with $G \in \mathcal{G}$ which does not factor through $m$.
\end{defn}

\begin{defn}\label{defn:elementary}
A quasi-abelian category is quasi-elementary (resp. elementary) if it is cocomplete and has a small strict generating set of small (resp. tiny) projective objects.
\end{defn}

For abelian categories quasi-elementary is equivalent to elementary.
We will freely use the following proposition which comes from Proposition 2.1.18 of \cite{SchneidersQA}

\begin{prop} Let $\ttC$ be a small, closed, symmetric monoidal, quasi-abelian category and $R\in \ttComm(\ttC)$. Then $\ttMod(R)$ is elementary if $\ttC$ is, $R\ootimes P$ is tiny in $\ttMod(R)$ whenever $P$ is tiny in $\ttC$. If $\mathcal{G}$ is a strict generating set of $\ttC$ then $\{R\ootimes G\ \ | \ \ G \in \mathcal{G}\}$ is a strict generating set of $\ttMod(R)$.
\end{prop}
\subsection{Ind-Categories and Ind-Categories of quasi-abelian categories}

Recall that for any category $\ttC$ we can define its ind-completion.
\begin{defn} 
Let $\ttC$ be a category. An ind-completion of $\ttC$ is a category $\ttD$ with a functor 
$i:\ttC \to \ttD$, such that $\ttD$ is closed under filtered colimits, and the functor $i$ is initial with respect to functors into categories closed under filtered colimits.
\end{defn}

\begin{lem}
Let $\ttC$ be a category. Its ind-completion exists and can be realized as the full subcategory of the category $\ttPr(\ttC)=\Fun(\ttC^{op},\ttSet)$ whose objects are filtered colimits of representable functors (note that the category of presheaves is cocomplete). 
\end{lem}

We will denote the ind-completion of $\ttC$ by $\ttInd(\ttC)$.  
Given two presentations of objects $E \cong ``\underset{i\in I}\colim"E_i$ and $F \cong ``\underset{j \in J}\colim" F_j$, we have a canonical isomorphism
\[\Hom(E, F) \cong
\lim_{i\in I}\underset{j \in J}\colim\Hom(E_i,F_j).
\]
\begin{rem}\label{rem:DescribeHoms}
Therefore, a morphism can be represented by a functor  $\alpha: I \to J$ and for each $i \in I$ an element of $\Hom(E_{i}, F_{\alpha(i)})$ giving a natural transformation $E \to F\circ \alpha$.
\end{rem}
One way of getting elementary quasi-abelian categories is by looking at ind-completions of quasi-abelian categories (2.1.17 in \cite{SchneidersQA}):

\begin{thm}
Let $\ttE$ be a small quasi-abelian category with enough projective objects. Then,
$\ttInd(\ttE)$
is an elementary quasi-abelian category.
\end{thm}

The following is 2.1.19 in \cite{SchneidersQA}:
\begin{prop}\label{prop:ScnMain}
Let $\ttE$ be a small, closed, symmetric monoidal, quasi-abelian category. The category $\ttInd(\ttE)$ has a canonical closed symmetric monoidal structure extending that on $\ttE$. Hence, if $\ttE$ has enough projectives, $\ttInd(\ttE)$ is a closed symmetric monoidal elementary quasi-abelian category.
\end{prop}
\proof 
The extension is given as follows: 
\begin{equation*}\label{eqn:IndRules}
``\underset{i\in I}\colim"E_i\ootimes ``\underset{j\in J}\colim" F_j=``\underset{(i,j) \in I\times J}\colim"E_i\ootimes F_j
\end{equation*}
\begin{equation*}
\underline{\Hom}(``\underset{i\in I}\colim"E_i, ``\underset{j \in J}\colim" F_j)=
\lim_{i\in I}``\underset{j \in J}\colim"\underline{\Hom}(E_i,F_j).
\end{equation*}
\endproof

The following is 2.1.7 in \cite{SchneidersQA}:
\begin{lem}\label{lem:cover}
Let $\ttE$ be a cocomplete quasi-abelian category. A small subset $\mathcal{G}$ of objects of $\ttE$ is a strictly generating set of $\ttE$ if and only if for any object $E$ of $\ttE$, there is a strict epimorphism of the form
\begin{equation*}
\coprod_{j\in J} G_j \to E 
\end{equation*}
where $(G_j)_{j\in J}$ is a small family of elements of $\mathcal{G}$.
\end{lem}

We assume the reader is familiar with the notions of a family of injective objects with respect to a functor between quasi-abelian categories \cite{SchneidersQA}.
In this section we recall how to derive the inverse limit functor in quasi-abelian categories.

Let $I$ be a small category. Given a functor $V:I \to \ttC$ the Roos complex of $V$ is of the form 
\[0\to \mathscr{R}^{0}(V)\to \mathscr{R}^{1}(V) \to  \mathscr{R}^{2}(V) \to \cdots
\]
$\mathscr{R}^{0}(V)=\underset{i\in I}\prod V_i $
and \[\mathscr{R}^{n}(V)=\underset{i_0\stackrel{\alpha_1}\to i_1 \stackrel{\alpha_2}\to i_2 \to \cdots \stackrel{\alpha_n}\to i_n }\prod V_{i_0} \] where the product is over all composable sequences of $n$ morphisms in $I$. The differential $\mathscr{R}^{n}(V)\to \mathscr{R}^{n+1}(V)$ is defined for $\alpha$ the composable sequence $i_0 \stackrel{\alpha_1}\to i_1 \to \cdots \stackrel{\alpha_{n+1}}\to i_{n+1}$
\[(d((v_{\beta})_{\beta}))_{\alpha} = V(\alpha_1)v_{i_1\stackrel{\alpha_2}\to i_2 \to \cdots \to i_{n+1} }+\sum_{l=1}^{n}(-1)^{l}v_{i_0\stackrel{\alpha_1}\to i_2 \to \cdots \to i_{l-1} \stackrel{\alpha_{l+1} \circ \alpha_{l}}\to i_{l+1} \to \cdots \to i_{n+1} } + (-1)^{n+1}v_{i_0\stackrel{\alpha_1}\to i_1 \to \cdots \to i_{n} }.
\]
Usual abstract nonsense arguments show the existence of a derived functor $\mathbb{R} \underset{i\in I}\lim$.
\begin{defn}The $\underset{i\in I}\lim$-acyclic objects are objects $V$ of $\ttC^{I}$ satisfying $\mathbb{R} \underset{i\in I}\lim V_{i} \cong \underset{i\in I}\lim V_{i}.$ An object $V$ of $\ttC^{I}$ will be called Roos-acyclic if the differentials in the Roos complex are strict and the cohomology of the Roos complex is concentrated in degree zero.
\end{defn}
Inverse limits have an explicit derived functor because of the following proposition of Prosmans \cite{Pr2}. 
\begin{prop}
Let $I$ be a small category and $\ttC$ a quasi-abelian category with exact products. Then the family of objects in $\ttC^{I}$ which are Roos-acyclic form a $\underset{i\in I}\lim$-acyclic family. As a result, the functor 
\[\underset{i\in I}\lim:\ttC^{I} \to \ttC \]
is right derivable and for any object $V \in \ttC^{I}$, we have an isomorphism 
\begin{equation}\label{eqn:Derived2Roos}
\mathbb{R} \underset{i\in I}\lim V_{i}\cong \mathscr{R}^{\bullet}(V),
\end{equation}
where the right hand side is the Roos complex of $V$. The family of $\underset{i\in I}\lim$-acyclic objects for the functor $\underset{i\in I}\lim:\ttC^{I} \to \ttC$ form a family of injectives relative to this functor (a concept appearing in \cite{SchneidersQA}).
\end{prop}

Because of the explicit formula of the Roos complex, notice that
\begin{cor}Let $I$ be a small category and $\ttC$ a quasi-abelian closed symmetric monoidal category with exact products. If $W$ is flat in  $\ttC$ and $W \ootimes(-)$ commutes with products in $\ttC$ then the natural morphism
\[W \ootimes^{\mathbb{L}}(\mathbb{R} \underset{i\in I}\lim V_{i}) \to \mathbb{R} \underset{i\in I}\lim (W \ootimes V_{i} )
\]
is an isomorphism. In particular, if $V \in \ttC^{I}$ is $\underset{i\in I}\lim$-acyclic then so is $W \ootimes V$ and the canonical morphism
\[W \ootimes (\underset{i\in I}\lim V_{i} ) \to  \underset{i\in I}\lim (W \ootimes V_{i} )
\]
is an isomorphism.
\end{cor}
We will use this material again in Lemma \ref{lem:UseML}.

\subsection{Relative Geometry}
Just as algebraic geometry is ``built" from the theory of commutative rings and their modules, much work on other kinds of geometry and topology is based on commutative monoids and their modules internal to general symmetric monoidal categories $(C, \ootimes, e)$, for instance see \cite{TV}. In our approach we also ask that they be equipped with compatible Quillen exact structures. The category most important for us is the quasi-abelian example of Ind-Banach modules over a Banach ring together with its projective tensor product and its applications to analytic and arithmetic geometry. An important class of morphisms between ``affine schemes" in relative geometry are opposite to those morphisms $A\to B$ in $\ttComm(C)$ such that the natural map $B\ootimes^{\mathbb{L}}_{A}B \to B$  is a quasi-isomorphism. Such a morphism will be called a \emph{homotopy epimorphism}. We use that terminology because this notion actually is equivalent to the general model or infinity-category notion of a homotopy epimorphism as found in work of To\"{e}n and Vezzosi (as used in homotopical or derived algebraic geometry). However, other sources call this a stably flat morphism \cite{NR}, an isocohomological morphism \cite{MeyerDerived} or a homological epimorphism \cite{GL}, \cite{CL}. It appears in homotopy theory \cite{EKMM}, representation theory, and algebra under different names. Practically, the only way we know to prove that a morphism is a homotopy epimorphism is to resolve $B$ by projective and flat $A$-modules in a clever way allowing for computation of the derived projective tensor product. In particular, we must prove that the resolution remains a resolution after applying the projective tensor product with $B$. A particular example of a homotopy epimorphism is a flat epimorphism, however there are other examples which will be shown in the next section.
\subsection{Banach Rings and Banach Modules}\label{BRBM}
\begin{defn}\label{defn:BanRng}
By a \emph{complete normed (or Banach) ring} we mean a commutative ring with identity $R$ equipped with a function, $|\cdot|: R \to \mathbb{R}_{\ge 0}$ such that
\begin{itemize}
\item $|a| = 0$ if and only if $a=0$;
\item $|a + b| \le |a| + |b|$ for all $a,b \in R$;
\item there is a $C>0$ such that $|a b| \leq C |a||b|$ for all $a,b \in R$;
\item $R$ is a complete metric space with respect to the metric $(a,b) \mapsto |a-b|$.
\end{itemize}
The category of Banach rings has as morphisms those ring homomorphisms $R \to S$ such that there exists a constant $C > 0$ such that $|\phi(a)|_S \le C |a|_R$ for all $a \in R$, in other words bounded ring homomorphisms.
\end{defn}
The initial Banach ring is the integers $\mathbb{Z}$ equipped with the standard absolute value as norm. 

\begin{defn}\label{defn:module}
Let $(R, |\cdot|_R)$ be a Banach ring. A \emph{Banach module} over $R$ is an $R$-module $M$ equipped with a function $\|\cdot \|_M: M \to \mathbb{R}_{\ge 0}$ such that for any $m,n \in M$ and $a \in R$:
\begin{itemize}
\item $\|0_M\|_M = 0$;
\item $\|m + n\|_M \le \|m\|_M + \|n\|_M$;
\item $\|a m\|_M \le C |a|_R\|m\|_M$ for some constant $C > 0$;
\item $\|m\|_M = 0$ implies that $m=0_M$;
\item $M$ is complete with respect to the metric $d(m,n)=\| m-n \|$.
\end{itemize}
\end{defn}
\begin{example}Any abelian group or ring can be considered a Banach ring by equipping it with the trivial norm which assigns $0$ to the zero element and $1$ for each non-zero element. We use notation such as for example $\mathbb{Z}_{\triv}$ for the integers equipped with the trivial norm. If $M$ is a module over a Banach ring $R$, we can make $M$ into a Banach module by equipping it with the trivial norm.
\end{example}
\begin{defn}\label{defn:n-arch}
A Banach ring or a Banach module over a Banach ring is called \emph{non-archimedean} if its semi-norm obeys the strong triangle inequality: for any two elements $v, w$ we have $ \|v +w \| \leq \max\{ \|v \|, \|w \| \}$.
\end{defn}
\begin{defn}\label{defn:scaling}If $(M, \| \ \ \|_{M})$ is a Banach module over a Banach ring $R$ and $r$ is a positive real number then $M_r$ is the Banach module over $R$ defined by the underlying module $M$ equipped with the Banach structure $r \| \ \ \|_{M}$.
\end{defn}
\begin{defn}Let $(R, |\cdot|_R)$ be a Banach ring. 
A $R$-linear map between Banach $R$-modules (Definition \ref{defn:module}), $f: (M, \|\cdot\|_M) \to (N, \|\cdot\|_N)$ is called \emph{bounded} if there exists a real constant $C > 0$ such that
\[ \|f(m)\|_N \le C \|m\|_M \]
for any $m \in M$. The homomorphism $f$ is called \emph{non-expanding} if this equation holds for $C=1$.
\end{defn}
The category of Banach modules with bounded morphisms is denoted by $\ttBan_{R}$. If $R$ is non-archimedean $\ttBan^{na}_{R}$ denotes the category of non-archimedean Banach modules with bounded morphisms.

\begin{lem}For any Banach ring $R$, $R$ is projective as a Banach $R$-module.
\end{lem}
\begin{lem}For any projective $R$-module, $P$, and any real number $r>0$, $P_{r}$ is also projective.
\end{lem}
\begin{defn}
Given $M, N \in \ttBan_{R}$ we define $M \wotimes_{R}N$ as the (separated) completion of $M \otimes_{R}N$ with respect to the semi-norm
\[||x||=\inf \{ \sum_{i=1}^{
n} ||m_i||||n_i|| \ \  | \ \ x= \sum_{i=1}^{
n} m_i  \wotimes_{R} n_i\}.
\]
Similarly, if $R$ is non-archimedean, given  $M, N \in \ttBan^{na}_{R}$ we define $M \wotimes^{na}_{R}N$ as the (separated) completion of $M \otimes_{R}N$ with respect to the semi-norm
\[||x||=\inf \{ \underset{i=1, \dots, n}\sup ||m_i||||n_i|| \ \  | \ \ x= \sum_{i=1}^{
n} m_i  \wotimes_{R} n_i\}.
\]
\end{defn}
The internal Hom in these categories
is denoted by $\uHom_{R}(V,W)$ and given by the Banach space whose underlying vector space is just the bounded $R$-linear maps
\[\{T \in \Lin_{R}(V,W)| \|T\| < \infty \}
\]
with norm given by  $\|T\|=\underset{v \in V, v \neq 0}\sup \frac{\|T(v)\|}{\|v\|}$. We write $V^{\vee}$ for $\uHom_{R}(V,R) \in \ttBan_{R}$.
The categories $\ttBan_{R}$ and $\ttBan^{na}_{R}$ are both closed symmetric monoidal when equipped with these projective tensor product with unit object given by $R$.

\begin{defn}The category $\ttBan^{\leq 1}_{R}$ 
is defined to have the same objects as $\ttBan_{R}$.
The morphisms are the linear maps with norm less than or equal to one (these are called non-expanding or sometimes just contracting). \end{defn}
This defines a closed symmetric monoidal category with the same internal hom and tensor product and as before it has two versions (one of which exists only when $R$ is non-archimedean). Infinite products and coproducts in $\ttBan^{\leq 1}_{R}$ exist even though they do not exist in $\ttBan_{R}$. In the archimedean case (see page 63 of \cite{Hel}) the product $\prod^{\leq 1}_{i\in I} V_{i}$ of a collection $\{V_{i}\}_{i \in I}$ in $\ttBan^{\leq 1}_{R}$ is given by 
\[\{
(v_i)_{i \in I} \in \bigtimes_{i \in I}V_{i} \ \ | \ \ \sup_{i \in I} \|v_{i}\| < \infty \}  
\] equipped with the norm 
\[\|(v_i)_{i \in I} \| =\sup_{i \in I} \|v_{i}\|
\]
while the coproduct $\coprod^{\leq 1}_{i\in I} V_{i}$ of a collection $\{V_{i}\}_{i \in I}$ in $\ttBan^{\leq 1}_{R}$ is given by \[\{
(v_i)_{i \in I} \in \bigtimes_{i \in I}V_{i} \ \ | \ \sum_{i \in I} \|v_{i}\| <\infty \}  
\] equipped with the norm 
\[\|(v_i)_{i \in I} \| = \sum_{i \in I} \|v_{i}\|.
\]
If $R$ is non-archimedean and we choose to work in the non-archimedean case, they can be computed as in \cite{Gruson}: the product $\prod^{\leq 1}_{i\in I} V_{i}$ of a collection $\{V_{i}\}_{i \in I}$ in $\ttBan^{\leq 1}_{R}$ is given by 
\[\{
(v_i)_{i \in I} \in \bigtimes_{i \in I}V_{i} \ \ | \ \ \sup_{i \in I} \|v_{i}\| < \infty \}  
\] equipped with the norm 
\[\|(v_i)_{i \in I} \| =\sup_{i \in I} \|v_{i}\|
\]
while the coproduct $\coprod^{\leq 1}_{i\in I} V_{i}$ of a collection $\{V_{i}\}_{i \in I}$ in $\ttBan^{\leq 1}_{R}$ is given by \[\{
(v_i)_{i \in I} \in \bigtimes_{i \in I}V_{i} \ \ | \ \lim_{i \in I} \|v_{i}\| =0 \}  
\] equipped with the norm 
\[\|(v_i)_{i \in I} \| = \sup_{i \in I} \|v_{i}\|.
\]
General limits (resp. colimits) are constructed out of kernels and products (resp. cokernels and coproducts) in the usual way. Finite limits (resp. finite colimits) in $\ttBan^{\leq 1}_{R}$ agree with those in $\ttBan_{R}$.
\begin{lem}\label{lem:ker_colim}
Suppose we are given a collection $\{f_{i}:V_{i} \to W_{i}\}_{i \in I}$ in $\ttBan^{\leq 1}_{R}.$ Then observe that the natural morphism
\[\coprod{}^{\leq 1}_{i\in I} \ker(f_i) \rightarrow \ker [\coprod{}^{\leq 1}_{i\in I} V_{i} \rightarrow  \coprod{}^{\leq 1}_{i\in I} W_{i}] 
\]
is an isomorphism. Similarly, if $V_{i} \subset V$ and $W_{i} \subset W$ are countable increasing unions of complete closed isometric submodules with union $V$ and $W,$ respectively, then the natural map 
\[\colim^{\leq 1}_{i\in I} \ker(f_i) \rightarrow \ker [V \rightarrow W] 
\]
is an isomorphism. 
\end{lem}
\begin{lem}\label{lem:Explicit} An object $M$ of $\ttBan_{R}$ is projective if and only if there exists a set $S$ and a function $f:S \to \mathbb{R}_{\geq 0}$ and another element $N$ along with an isomorphism $M \coprod N \cong \underset{s\in S}\coprod^{\leq 1} R_{f(s)}$.
\end{lem}
\proof  There is a canonical strict epimorphism $\underset{m\in M^{\times}}\coprod^{\leq 1} R_{||m||} \to M$ discussed in \cite{BK} and if $M$ is projective this splits. Conversely, if $M \coprod N \cong \underset{s\in S}\coprod^{\leq 1} R_{f(s)}$ and $F \to E$ is a strict epimorphism then $\Hom(\underset{s\in S}\coprod^{\leq 1} R_{f(s)}, F) \to \Hom(\underset{s\in S}\coprod^{\leq 1} R_{f(s)}, E)$ is surjective and this breaks up into a product of a map $\Hom(M, F) \to \Hom(M, E)$ and a map $\Hom(N, F) \to \Hom(N, E)$ and so these are both surjective. Therefore $M$ is projective.
\endproof
\begin{lem}\label{lem:TensProj} If $P$ and $Q$ are projective in $\ttBan_{R}$ then $P\wotimes_{R}Q$ is also projective in $\ttBan_{R}$.
\end{lem}
\proof 
Using Lemma \ref{lem:Explicit}, we can complement $P$ and $Q$ by modules $P'$ and $Q'$ in order to conclude that there is a module $S=(P\wotimes_{R}Q') \oplus (P' \wotimes_{R}Q) \oplus (P'\wotimes_{R}Q')$ so that 
\[(P\wotimes_{R}Q )\oplus S \cong \underset{(p,q)\in P^{\times} \times Q^{\times}}\coprod^{\leq 1} R_{||p||||q||}
\]
and so the lemma follows from another application of Lemma \ref{lem:Explicit}.
\endproof
\begin{lem}\label{lem:Proj2Flat}
Any projective in $\ttBan_{R}$ is flat in $\ttBan_{R}$.
\end{lem}
\proof 
Let $P$ be a projective in $\ttBan_{R}$. There is a canonical strict epimorphism $\underset{p\in P^{\times}}\coprod^{\leq 1} R_{||p||} \to P$ discussed in \cite{BK}. As usual, it splits and so $\underset{p\in P^{\times}}\coprod^{\leq 1} R_{||p||}$ is coproduct (in $\ttBan_{R}$) of the kernel and $P$. Hence $P$ is flat by Lemma \ref{lem:FlatProps}.
\endproof
 
The proof of the following is obvious from the definitions. 
\begin{lem}\label{lem:colimProps} Any finite coproduct of projective objects in $\ttBan_{R}$ is projective in $\ttBan_{R}$. 
\end{lem}

\begin{lem}\label{lem:commuteFiltSSES}\cite{Gruson} A filtered colimit of strict, short exact sequences in $\ttBan_{R}^{\leq 1}$ is a strict short exact sequence.
\end{lem}

\proof  See Proposition 1 on page 69 of \cite{Gruson}.
\endproof
\begin{lem}\label{lem:IntHomProj} If $V \to W$ is a strict epimorphism and $P$ is projective then the corresponding morphism $\uHom(P,V) \to \uHom(P,W)$ is a strict epimorphism.
\end{lem}

\proof 
By Proposition 1.3.23 of \cite{SchneidersQA} it is enough to show that for any projective $Q$, that \[\Hom(Q, \uHom(P,V)) \to \Hom(Q, \uHom(P,W))\] is surjective. This follows immediately from adjunction and from Lemma \ref{lem:TensProj}.
\endproof
\begin{rem}The projectives and strict epimorphisms determine one another in the sense that a morphism $V\to W$ is a strict epimorphism if and only if  $\Hom(P, V) \to \Hom(P,W)$ is surjective for every projective $P$ and a module $M$ is projective if and only if $\Hom(M, V)\to \Hom(M,W)$ is surjective for every strict epimorphism $V \to W$. We will use this often in what follows.
\end{rem}
\begin{lem} For any small set $S$ and projectives $P_{s} \in \ttBan_{R}$ for each $s\in S$ the object $P=\underset{s \in S}\coprod^{\leq 1}P_{s}$ is projective in $\ttBan_{R}$.
\end{lem}

\proof 
Let $\pi: V\to W$ be a strict epimorphism and let $f:P \to W$ be any morphism. By Lemma \ref{lem:IntHomProj}, for each projective $P_{s}$ we get a strict epimorphism \[\uHom(P_s, \pi): \uHom(P_s, V) \to \uHom(P_s, W).\] Let $f_s$ be the restriction of $f$ to $P_s$ so $||f_s||\leq ||f||$.  Fix $\epsilon>0$. Using the strict epimorphism property (see the characterizations in the appendix of \cite{BK}), choose for each $s$ a lift $\tilde{f_s}\in \uHom(P_s, V)$ of $f_s$ such that $||\tilde{f_s}|| \leq ||f_s||+\epsilon \leq ||f||+\epsilon$. As their norms are bounded independent of $s$, the $\tilde{f_s}$ assemble into a morphism $\tilde{f}:P\to V$ inducing $f$.
\endproof

\begin{lem}The category $\ttBan_{R}$ has enough projectives and all projectives in $\ttBan_{R}$ are flat.
\end{lem}
\proof 
The proof is exactly as in \cite{BK}.
\endproof

\begin{lem}Let $R$ be a Banach ring and $M$ a Banach $R$-module. Then for any positive real number $r$ we have $ (M_r)^{\vee}\cong (M^{\vee})_{r^{-1}}$.
\end{lem}
\begin{lem}\label{lem:WeightsProj}Let $R$ be a Banach ring and $M$ a Banach $R$-module. Then for any positive real number $r$, $M_{r}$ is projective if and only if $M$ is projective.
\end{lem}
\begin{lem}\label{lem:DualColimIsLim}Given an inductive system $V_i$ in $\ttBan^{\leq 1}_{R}$ the canonical morphism 
\[(\underset{i\in I}\colim^{\leq 1}V_i)^{\vee}\to \underset{i\in I}\lim^{\leq 1}(V^{\vee}_i)
\]
(induced by the duals of the collection of isometric immersions $V_i \to \underset{i\in I}\colim^{\leq 1}V_i$) is an isomorphism. 
\end{lem}
\proof 
It is enough to show that it induces an isomorphism of sets 
\[\left((\underset{i\in I}\colim^{\leq 1}V_i)^{\vee}\right)^{\leq r} \to \left(\underset{i\in I}\lim^{\leq 1}(V^{\vee}_i)\right)^{\leq r}
\]
for any real number $r\geq 1$. The canonical morphism identifies the left hand side with 
\begin{equation}
\begin{split}
 \Hom^{\leq 1}(R_r, (\underset{i\in I}\colim^{\leq 1}V_i)^{\vee}) =  & \Hom^{\leq 1}(R_r, \uHom(\underset{i\in I}\colim^{\leq 1}V_i, R))= \Hom^{\leq 1}(R_r\wotimes_{R}(\underset{i\in I}\colim^{\leq 1}V_i), R) \\ 
 = & \Hom^{\leq 1}(\underset{i\in I}\colim^{\leq 1}((V_i)_{r}), R) = \underset{i\in I}\lim \Hom^{\leq 1}((V_i)_{r}, R)  \\ = & \underset{i\in I}\lim \Hom^{\leq 1}(R, ((V_i)_{r})^{\vee}) 
= \underset{i\in I}\lim \Hom^{\leq 1}(R, ({V_i}^{\vee})_{r^{-1}}) \\  = & \underset{i\in I}\lim \Hom^{\leq 1}(R_r, {V_i}^{\vee})  =   \Hom^{\leq 1}(R_r, \underset{i\in I}\lim^{\leq 1}({V_i}^{\vee})),
\end{split}
\end{equation}
which agrees with the right hand side.
\endproof
\begin{cor}Given a morphism of inductive systems induced by morphisms $V_i \to W_i$ in $\ttBan^{\leq 1}_{R}$ the dual of the corresponding morphism 
\[\underset{i\in I}\colim^{\leq 1}V_i\to \underset{i\in I}\colim^{\leq 1}W_i
\]
is the morphism 
\[\underset{i\in I}\lim^{\leq 1}(V^{\vee}_i)\leftarrow\underset{i\in I}\lim^{\leq 1}(W^{\vee}_i)
\]
corresponding to the dual morphisms $V^{\vee}_i \leftarrow W^{\vee}_i$.
\end{cor}
\proof 
This is automatic from the definitions and Lemma \ref{lem:DualColimIsLim}.
\endproof

\begin{defn}\label{defn:Tate}For any Banach ring $R$ and $n$-tuple of positive real numbers $r=(r_1, \dots, r_n)$ the poly-disk algebra of poly-radius $r$ is defined by the sub-ring
\[R\{\frac{x_1}{r_1}, \dots, \frac{x_n}{r_n} \}=\{\underset{J}\sum a_J x^J \in R[[x_1, \dots, x_n]]\ \ | \ \ \underset{J}\sum |a_J| r^J< \infty\}
\]
equipped with the norm $|\underset{J}\sum a_J x^J|=\underset{J}\sum |a_J| r^J$. When $R$ is non-archimedian, one can still use the above if in the archimedean context, or instead if one wants to work in the non-archimedean context one can read this article using the Tate algebra
\[R\{\frac{x_1}{r_1}, \dots, \frac{x_n}{r_n} \}=\{\underset{J}\sum a_J x^J \in R[[x_1, \dots, x_n]]\ \ | \ \ \underset{J}\lim |a_J| r^J=0\}
\]
equipped with the norm $|\underset{J}\sum a_J x^J|=\underset{J}\sup |a_J| r^J$. These are symmetric ring constructions (see Equation \ref{eqn:SymAlgConst}) in the categories $\ttBan^{\leq 1}_{R}$ or $\ttBan^{\leq 1, na}_{R}$ applied to the object of $\ttBan_{R}$ given by $\underset{i=1, \dots, n}\coprod R_{r_i}.$ Similarly, we can define poly-disk versions of the Banach abelian groups $M\{\frac{x_1}{r_1}, \dots, \frac{x_n}{r_n} \}$ for any Banach abelian group $M$ and when $M$ is non-archimedean a Tate version which is non-archimedean. The formulas for them are 
\[M\{\frac{x_1}{r_1}, \dots, \frac{x_n}{r_n} \}=\{\underset{J}\sum m_J x^J \in M[[x_1, \dots, x_n]]\ \ | \ \ \underset{J}\sum |m_J| r^J< \infty\}
\]
equipped with the norm $|\underset{J}\sum m_J x^J|=\underset{J}\sum |m_J| r^J$. When $M$ is non-archimedian, one can still use the above if in the archimedean context, or instead if one wants to work in the non-archimedean context one define
\[M\{\frac{x_1}{r_1}, \dots, \frac{x_n}{r_n} \}=\{\underset{J}\sum m_J x^J \in M[[x_1, \dots, x_n]]\ \ | \ \ \underset{J}\lim |m_J| r^J=0\}
\]
equipped with the norm $|\underset{J}\sum m_J x^J|=\underset{J}\sup |m_J| r^J$. Notice that these are completions of the group $M[\frac{x_1}{r_1}, \dots, \frac{x_n}{r_n}]$ and subgroups of $M[[x_1, \dots, x_n]]$. It is easy to see that they satisfy $M\{\frac{x_1}{r_1}, \dots, \frac{x_n}{r_n} \}=M \wotimes_{R}R\{\frac{x_1}{r_1}, \dots, \frac{x_n}{r_n} \}$.
\end{defn}
\begin{rem}If $R$ is non-archimedean, all of this subsection goes through for $\ttBan^{na}_R$ in place of $\ttBan_R$. Just as finitely presentable rings play an important role in algebraic geometry, in Banach algebraic geometry over $R$, the nice objects of study are quotients of the above disk algebras by ideals, equipped with the quotient Banach structure. As the category of these affinoid algebras is not closed under filtered limits or colimits, it is natural to introduce also Stein and dagger algebras in Section \ref{Spaces} and perhaps even more general limits and colimits like quasi-Stein, Stein-dagger, and quasi-Stein-dagger, etc.
\end{rem}
\begin{rem}\label{rem:flatlocalization}
Let $R$ be a non-zero Banach ring with multiplicative norm. It is automatically an integral domain. Let $S$ be a multiplicative subset, and equip the localization $S^{-1}R$ of $R$ with the norm $|\frac{r}{s}|=\frac{|r|}{|s|}$. Then for any Banach ring $T$, the map $R \to  \widehat{S^{-1}R}$ identifies $\Hom(\widehat{S^{-1}R},T)$ with the bounded ring morphisms $R \to T$ sending $S$ to invertible elements and so categorically, the map $R \to  \widehat{S^{-1}R}$ is an epimorphism in the category of Banach rings, equivalently $\widehat{S^{-1}R}\wotimes_{R} \widehat{S^{-1}R} \cong \widehat{S^{-1}R}$. 
\end{rem}
\subsection{Geometric and Arithmetic Examples of Homotopy Epimorphisms}
In this subsection we discuss several examples of homotopy epimorphisms and derived projective tensor products. These all have a geometric meaning in terms of the Berkovich or Huber spectrum of $\mathbb{Z}$. As we are working over $\mathbb{Z}$ in this subsection, it takes place entirely in the archimedean context. We will return to looking at these in terms of covers of $\spec(\mathbb{Z})$ and descent in future work.  As a matter of notation, we consider $\mathbb{Z}$, $\mathbb{Z}_p$,  $\mathbb{Q}_p$, and $\mathbb{R}$ as Banach rings by using their standard norms.
\begin{obs}\label{obs:ortho} We have $\mathbb{Q}_{p}\wotimes_{\mathbb{Z}}\mathbb{R}=\{0\}= \mathbb{Z}_{p}\wotimes_{\mathbb{Z}}\mathbb{R}$ for any prime $p$ and for distinct primes $p$ and $q$, we have $\mathbb{Q}_{p}\wotimes_{\mathbb{Z}}\mathbb{Q}_{q}=\{0\}=\mathbb{Z}_{p}\wotimes_{\mathbb{Z}}\mathbb{Z}_{q}$. As a consequence, $\mathbb{Q}_p$ and $\mathbb{Z}_{p}$ are not flat with respect to the completed tensor product over $\mathbb{Z}$. 
\end{obs}
\proof  In $\mathbb{Z}_{p}\wotimes_{\mathbb{Z}}\mathbb{R}$ the element $1\wotimes 1$ can be written as $p^{n}\wotimes p^{-n}$ which has norm $p^{-2n}$ for any $n$. In $\mathbb{Z}_{p}\wotimes_{\mathbb{Z}}\mathbb{Z}_{q}$ choose for each $n$, integers $a_n$ and $b_n$ 
with $a_n p^{n}+b_n q^{n}=1$. Then $1\wotimes 1$ can be written as $a_n p^n \wotimes 1 + 1\wotimes b_n q^n$ which has norm less than or equal to $p^{-n}+q^{-n}$. Letting $n$ go to infinity we see that in both Banach rings, $1= 1\wotimes 1$ has norm zero and hence vanishes and so these rings are the zero ring. Applying the functor $\mathbb{Z}_{p}\wotimes_{\mathbb{Z}}(-)$ to the strict short exact sequence $\mathbb{Z}\to \mathbb{R}\to S^{1}$ gives $\mathbb{Z}_{p} \to \{0\} \to  \mathbb{Z}_{p}\wotimes_{\mathbb{Z}} S^{1}=\{0\}$ and so $\mathbb{Z}_{p}$ is not flat. The proofs for the fraction fields with their obvious Banach structures are similar.
\endproof

Using the resolutions we develop later, its easy to see that these rings are also orthogonal on the derived level. The lack of flatness with respect to the projective tensor product is similar to the known problem in analytic geometry that certain morphisms $A\to B$ of various Banach, Fr\'{e}chet, or bornological algebras corresponding to the restriction of spaces of functions over various ``open" sets do not exhibit $B$ as a flat module with respect to the completed tensor product over $A$ \cite{BK, BK2}.

\begin{example}Consider the usual Tate algebra $A=\mathbb{Q}_{p}\langle x\rangle$ of non-archimedean geometry. We can think of it as a Banach-algebraic version of the closed disk $\{t\in \mathbb{Q}_{p} \ \ | \ \  |t|\leq 1\}$. Meanwhile, if we let $W=\{t\in \mathbb{Q}_{p} \ \ | \ \  \frac{1}{2} \leq |t|\leq 1\}$ and  $V=\{t\in \mathbb{Q}_{p} \ \ | \ \  |t|\leq \frac{1}{3}\}$ so that the intersection of $W$ and $V$ is empty, we let $A_V=A\langle 3y \rangle /(y-x)\cong \mathbb{Q}_{p} \langle x, 3y \rangle /(y-x) \cong \mathbb{Q}_{p}\langle 3x\rangle $ and $A_W=A \langle \frac{z}{2}\rangle/(xz-1)$. Then $A_V \wotimes^{\mathbb{L}}_{A} A_W= \{0\}$. We have then a strict short exact sequence 
\[0 \to A \to A_W \to A_W/A \to 0
\]
and applying the functor $(-)\wotimes_{A}A_V$ we get $0 \to A_V \to 0$ and conclude that the categorical epimorphism $A\to A_V$ is not flat in our sense and in fact $(A_W/A)\wotimes^{\mathbb{L}}_{A}A_V \cong A_V[1]$. However, $A\to A_V$ is a homotopy epimorphism as proven in \cite{BK}. In this example we have used the non-archimedean completed projective tensor product.
\end{example}

This explains our preference for using homotopy epimorphisms instead of flat epimorphisms. The analogous issue does not arise in finite dimensional algebraic or differential geometry in the standard topologies. 
Let $R$ be a Banach ring. Let $V$ be a finite rank, free Banach module over $R$. Let $S_{R}(V)$ be the symmetric algebra of $V$ in the category $\ttInd(\ttBan_{R})$, a free object in $\ttComm(\ttInd(\ttBan_{R}))$. Note that as a bornological ring, $A$ is a polynomial algebra over $R$ with number of generators equal to the rank of $V$. Consider the algebra $S^{\leq 1}_{R}(V)$, a free object in $\ttComm(\ttBan^{\leq 1}_{R})$. This is a Banach ring which can be explicitly described as the subring $R\{x_1, \dots, x_n\}$ of $R[[x_1, \dots, x_n]]$ where $n$ is the rank of $V$ consisting of elements $ \underset{I \in \mathbb{Z}^{n}_{\geq 0}}\sum a_{I} x^{I}$ such that $ \underset{I \in \mathbb{Z}^{n}_{\geq 0}}\sum |a_I| < \infty$ and equipped with the norm $
|| \underset{I \in \mathbb{Z}^{n}_{\geq 0}}\sum a_{I} x^{I}|| = \underset{I \in \mathbb{Z}^{n}_{\geq 0}} \sum |a_I|$. Notice that even when $R$ is a non-archimedean ring or field, we can and will use this definition, because we are not restricting our attention to non-archimedean modules.
The idea of writing  $\mathbb{Z}_p$ in terms of disk algebras over $\mathbb{Z}$ goes back to F. Paugam \cite{Pa2}. We use his idea in the following lemma which uses the disk algebras (Definition \ref{defn:Tate}). We show here that one can think of $\mathbb{Z}_{p}$  as a sort of archimedean type rational localization of $\mathbb{Z}$.  The symmetric ring construction in the contracting category \cite{BK}  works equally well to define infinite dimensional disk algebras. 
\begin{defn}
Let ${\mathbb{Z}^\epsilon}_{p}$ be the completion of $\mathbb{Z}$ with respect to the norm $|ap^n|^{\epsilon}_{p}=p^{-n\epsilon}$ for $p$ not dividing $a \in \mathbb{Z}$, $n \geq 0$ and $0< \epsilon < \infty$ and $r=p^{-\epsilon}.$
\end{defn}
\begin{obs}
For distinct primes $p \neq q$ and $0< \epsilon, \delta < \infty$  we have ${\mathbb{Z}^\epsilon}_{p} \wotimes_{\mathbb{Z}} {\mathbb{Z}^\delta}_{q} =\{1\}$ similarly to Observation \ref{obs:ortho}. On the other hand the tensor seminorm of  $\mathbb{Z} = \mathbb{Z} \otimes_{\mathbb{Z}} \mathbb{Z} \subset {\mathbb{Z}^\epsilon}_{p} \otimes_{\mathbb{Z}}  {\mathbb{Z}^\delta}_{p}$ is just the norm  $|ap^n|=\min \{p^{-n\epsilon}, p^{-n\delta} \}$ for $p$ not dividing $a \in \mathbb{Z}$. Therefore, $ {\mathbb{Z}^\epsilon}_{p} \wotimes_{\mathbb{Z}}  {\mathbb{Z}^\delta}_{p}$ is the completion of $\mathbb{Z}$ with respect to the norm $|ap^n|^{\epsilon}_{p}=p^{-n\zeta}$ for $p$ not dividing $a \in \mathbb{Z}$ and $\zeta = \max \{\epsilon, \delta \}$.
\end{obs}
\begin{lem} \label{lem:ZpLem} For $0<r<1$, there is a strict short exact sequence

\[
0 \to \mathbb{Z}\{\frac{x}{r}\} \stackrel{(x-p)}\longrightarrow \mathbb{Z}\{\frac{x}{r}\}  \to  {\mathbb{Z}^\epsilon}_{p}\to 0.
\]
A $p$-adic number $\sum _{i=0}^{\infty} a_i p^i$ with $a_i$ between $0$ and $p-1$ has norm $\sum _{i=0}^{\infty} |a_i| r^i$. This is the usual $\mathbb{Z}_p$ as an abstract ring. In particular for $r=p^{-1}$ this gives the usual $p$-adic norm on $\mathbb{Z}^1_p=\mathbb{Z}_p$.
\end{lem}
\proof 

In order to see that the multiplication by $x-p$ map on $\mathbb{Z}\{\frac{x}{r}\}$ is a strict monomorphism suppose that $(x-p)\sum_{i=0}^{\infty}a_{i}x^{i}=\sum_{j=0}^{\infty}b_{j}x^{j}$. Then we see in particular $b_0$ is divisible by $p$, $a_0=-\frac{1}{p}b_0$, $b_1 +\frac{1}{p}b_0$ is as well, and $a_1=-\frac{1}{p}(b_1+\frac{1}{p}b_0)$ and in fact we can solve for
\[a_i = -\frac{1}{p}b_i -\frac{1}{p^2}b_{i-1} - \cdots -\frac{1}{p^{i+1}}b_0.
\]
Therefore, 
\[
|a_i|\leq p^{-1}|b_i| +p^{-2} |b_{i-1}|+ \cdots + p^{-(i+1)} |b_0|=\sum_{j=0}^{i}|b_j|p^{-1-i+j}
\]
and so
\[||\sum_{i=0}^{\infty}a_{i}x^{i}|| = \sum_{i=0}^{\infty}|a_{i}|r^{i}\leq  \sum_{i=0}^{\infty}\sum_{j=0}^{i}|b_j|p^{j-i-1}r^i=\sum_{0\leq j \leq i < \infty} |b_j|r^{j}p^{j-i-1}r^{i-j} =  p^{-1}(\sum_{k=0}^{\infty}(\frac{r}{p})^{k})( \sum_{j=0}^{\infty}|b_{j}|r^{j})
\]
and so
\[||\sum_{i=0}^{\infty}a_{i}x^{i}|| \leq p^{-1} \frac{1}{1-\frac{r}{p} }||\sum_{j=0}^{\infty}b_{j}x^{j}||=\frac{1}{p-r} ||\sum_{j=0}^{\infty}b_{j}x^{j}|| .
\]

For every prime $p$ there are isomorphisms of normed rings
\[\mathbb{Z}[x]/(x-p) \cong \mathbb{Z}
\]
where $|x|=r$ and the right hand side has the $|\ |^{\epsilon}_{p}$ norm. In order to explain this, given a polynomial $f(x)=\sum_{i=n}^{m}a_{i}x^{i}$ with $a_{n}\neq 0$, it is assigned to a number $\sum_{i=n}^{m}a_{i}p^{i}$ with norm bounded as follows: $|\sum_{i=n}^{m}a_{i}p^{i}|^{\epsilon}_{p} \leq \max_{i=n}^{m} \{r^{v_{p}(a_i)+i}\}\leq r^{n}\leq \sum_{i=n}^{m}|a_{i}|r^{i}=||\sum_{i=n}^{m}a_{i}x^{i} ||_{\mathbb{Z}\{\frac{x}{r}\}}$.  This gives a (bounded) morphism $\mathbb{Z}[x]\to \mathbb{Z}$. In fact, any integer $bp^s$ where $p$ does not divide $b$ and $s\geq 0$ has a $p$-adic expansion $\sum_{i=0}^{m}b_{i}p^{s+i}$ where $0\leq |b_{i}| \leq p-1$, in other words it is the evaluation of $\sum_{i=0}^{m}b_{i}x^{s+i}$ . Therefore, the infimum of the norms of any lift of $bp^s$ to $\mathbb{Z}\{\frac{x}{r}\}$ is bounded by \[\sum_{i=0}^{m}|b_{i}|r^{i+s}\leq r^{s}(p-1)\sum_{i=0}^{m}r^{i} \leq r^{s}(p-1)\sum_{i=0}^{\infty}r^{i}=\frac{p-1}{1-r}r^s=\frac{p-1}{1-r} |bp^s|^{\epsilon}_{p}.\] 
Therefore the morphism $\mathbb{Z}\{\frac{x}{r}\} \to  \mathbb{Z}^{\epsilon}_{p}$ is strict. By computing order by order modulo powers of $p$ with any polynomial $\sum_{i=n}^{m} a_i x^i $ where $\sum_{i=n}^{m} a_i p^{i}=0$, one finds that this element must be in the ideal $(x-p)$. If the polynomial $f$ maps to $bp^s$ where $p$ does not divide $b$ then $bp^s=\sum_{i=n}^{m}a_{i}p^{i}$  and so $|bp^s|^{\epsilon}_{p}\leq ||f||.$ We apply the completion functor to get the desired isomorphism $\mathbb{Z}\{\frac{x}{r}\}/(x-p) \cong \mathbb{Z}^{-\log_{p}(r)}_{p}$. 
\endproof
\begin{rem}
$\mathbb{Z}\{\frac{x}{r}\}/(x-p)$ represents the subset of Berkovich points which send $p$ to $[0,r]$. 
\end{rem}
\begin{example}Cokernels in the category of Banach abelian groups are simply the group quotient equipped with the norm given by the infimum of the norms of all lifts. Consider the cokernel $S^1 = \mathbb{R}/\mathbb{Z}$. We can compute $S^1 \wotimes^{\mathbb{L}}_{\mathbb{Z}} \mathbb{Z}_{p}$ using the above projective resolution of $\mathbb{Z}_p$. It is a non-zero Banach abelian group $K$ sitting in degree $-1$ including for example the element $\frac{1}{p}\sum_{i=0}^{\infty} (x/p)^{i}$ which has norm $\frac{1}{p-1}$ and is in the kernel $K$ of the strict epimorphism $S^{1} \{px\} \stackrel{x-p}\to  S^{1} \{px\}$.
\end{example}

\begin{rem}This geometric perspective can be useful, for instance, one could define the $p$-adic completion of a Banach ring $R$ as 
$R\{px\}/(x-p).$ There are also interesting new rings to define such as the following Fr\'{e}chet version of the $p$-adic integers. 
For instance we can consider the functions on the open disk of radius $1/p$ over 
$\spec(\mathbb{Z})$, $\underset{r<1/p} \lim \mathbb{Z}\{\frac{x}{r}\}$ in place of $ \mathbb{Z}\{px\}$ in the role it plays in 
Lemma \ref{lem:ZpLem}.
\end{rem}
\begin{defn}
\[\widetilde{\mathbb{Z}_{p}}= \underset{r<1/p} \lim ( \mathbb{Z}\{\frac{x}{r}\}/(x-p)) \cong  (\underset{r<1/p} \lim  \mathbb{Z}\{\frac{x}{r}\})/(x-p)  
\cong  \mathcal{O}(D^{1}_{<p^{-1}, \mathbb{Z}})/(x-p).
\]
\[\mathbb{Z}_{p}^{\dagger}= ``\underset{r>\frac{1}{p}}\colim" ( \mathbb{Z}\{\frac{x}{r}\}/(x-p))
\]
\end{defn}

There are bounded morphisms $\mathbb{Z}_{\triv}\to \mathbb{Z}_{p}$ and homotopy epimorphisms $\mathbb{Z} \to \mathbb{Z}_{p}^{\dagger}\to  \widetilde{\mathbb{Z}_{p}}$. As far as the Fr\'{e}chet version, this will appear in future work but for now we show:
\begin{lem}The natural map $\mathbb{Z} \to \mathbb{Z}_{p}^{\dagger}$ is a homotopy epimorphism.
\end{lem}
\proof 
Let $p$ be a prime and let $R$ be any Banach ring without $p$-torsion and with multiplicative norm and let $s$ be any real with $1<s<p$. Further assume that $|p|_{R}s>1$. Consider the bounded morphism $R\{sx\} \stackrel{x-p}\longrightarrow R\{sx\}$. Define $K^s \subset R\{sx\}$ as the kernel of the (bounded) evaluation at $p$ morphism $R\{sx\} \to R$. We define an inverse map $K^s \longrightarrow R\{sx\}$ sending $\sum_{j=0}^{\infty}b_{j}x^{j}$ to  $\sum_{i=0}^{\infty}a_{i}x^{i}$ where $(x-p)\sum_{i=0}^{\infty}a_{i}x^{i}=\sum_{j=0}^{\infty}b_{j}x^{j}$. Then we see in particular $b_0$ is divisible by $p$, $a_0=-\frac{1}{p}b_0$, $b_1 +\frac{1}{p}b_0$ is as well, and $a_1=-\frac{1}{p}(b_1+\frac{1}{p}b_0)$ and in fact we can solve for
\[a_i = -\frac{1}{p}b_i -\frac{1}{p^2}b_{i-1} - \cdots -\frac{1}{p^{i+1}}b_0.
\]
Therefore, 
\[
|a_i|_{R}\leq \sum_{j=0}^{i}|b_j|_{R}|p^{-1-i+j}|_{R}
\]
and so making the substitution $k=i-j$ we have
\begin{equation}
\begin{split}
||\sum_{i=0}^{\infty}a_{i}x^{i}|| = \sum_{i=0}^{\infty}|a_{i}|_{R}s^{-i}\leq  \sum_{i=0}^{\infty}\sum_{j=0}^{i}|b_jp^{j-i-1}|_{R}s^{-i} & \leq   \sum_{i=0}^{\infty}\sum_{j=0}^{\infty}|b_j|_{R}|p^{j-i-1}|_{R}s^{-i} \\ & \leq   |p^{-1}|_{R}(\sum_{k=0}^{\infty}(|p|_{R}s)^{-k})( \sum_{j=0}^{\infty}|b_{j}|_{R}s^{-j})
\end{split}
\end{equation}
and so
\[||\sum_{i=0}^{\infty}a_{i}x^{i}|| \leq |p^{-1}|_{R} \frac{1}{1-\frac{1}{s|p|_R} }||\sum_{j=0}^{\infty}b_{j}x^{j}||=  \frac{s}{s|p|_R-1} ||\sum_{j=0}^{\infty}b_{j}x^{j}||.
\]
Therefore, for $1<u<s<p$ we have the strict short exact sequence
\[0 \to \mathbb{Z}\{sx\} \to \mathbb{Z}\{sx\} \to \mathbb{Z}_p^{\log_{p} s} \to 0
\]
and it becomes a strict short exact sequence 
\[0 \to  \mathbb{Z}_p^{\log_{p} u}\{sx\} \to  \mathbb{Z}_p^{\log_{p} u}\{sx\} \to  \mathbb{Z}_p^{\log_{p} u} \wotimes_{\mathbb{Z}} \mathbb{Z}_p^{\log_{p} s} \to 0
\]
after applying the functor $(-) \wotimes_{\mathbb{Z}} \mathbb{Z}_p^{\log_{p} u}$.
Passing to colimits over $u<s<p$ we find that we have a strict short exact sequence
\[0 \to \mathbb{Z}\{px\}^{\dagger} \to \mathbb{Z}\{px\}^{\dagger} \to \mathbb{Z}_{p}^{\dagger} \to 0
\]
which is a projective resolution. Furthermore, after applying the  functor $(-) \wotimes_{\mathbb{Z}} \mathbb{Z}^{\dagger}_p$ the result is the strict short exact sequence given by the colimits over $u<s<p$ of  \[0 \to  \mathbb{Z}_p^{\log_{p} u}\{sx\} \to  \mathbb{Z}_p^{\log_{p} u}\{sx\} \to  \mathbb{Z}_p^{\log_{p} u} \wotimes_{\mathbb{Z}} \mathbb{Z}_p^{\log_{p} s} \to 0.
\]
Now since we in fact have $\mathbb{Z}_p^{\log_{p} u} \wotimes_{\mathbb{Z}} \mathbb{Z}_p^{\log_{p} s}=  \mathbb{Z}_p^{\log_{p} s}$ we find a strict short exact sequence 
\[
0 \to \mathbb{Z}_{p}^{\dagger}\{px\}^{\dagger} \to \mathbb{Z}_{p}^{\dagger}\{px\}^{\dagger} \to \mathbb{Z}_{p}^{\dagger} \to 0
\]
showing that $\mathbb{Z}_{p}^{\dagger} \wotimes^{\mathbb{L}}_{\mathbb{Z}} \mathbb{Z}_{p}^{\dagger}  \cong \mathbb{Z}_{p}^{\dagger}$.
\endproof

\subsection{Bornological and Ind-Banach Modules}
Let $R$ be a Banach ring as defined in Definition \ref{defn:BanRng}. A bornological module over $R$ is a pair consisting of an $R$-module $M$ together with a bornology on the set $M$ as in Definition 3.36 of \cite{BB} such that the structure morphisms for addition and action of $R$ are bounded. The morphisms of bornological modules are bounded $R$-linear maps. The homological properties of bornological spaces over $\mathbb{C}$ were discussed in \cite{Pr}. We define the full subcategory $\ttCBorn_{R}$ as those bornological modules for which there is an increasing union of subsets, each of which has the structure of an object of $\ttBan_{R}$ and the inclusion of the subsets in $M$ and in one another are all bounded morphisms. This category is equivalent to the subcategory of essentially monomorphic objects in $\ttInd(\ttBan_{R})$. For more on this category see \cite{BB} and \cite{BBK}. The importance of bornological spaces in complex geometry was studied by Houzel \cite{H}. Given an object $A\in \ttComm(\ttInd(\ttBan_{R}))$, the category $\ttMod(A)$ shares all the nice properties of the category $\ttInd(\ttBan_{R})$. We just remark here that the completed projective tensor product is defined by 
\begin{equation}\label{eqn:Atens} M \wotimes_{A}N = \colim [M \wotimes_{R}A \wotimes_{R} N  \rightrightarrows M \wotimes_{R}N] 
\end{equation}
\begin{lem}The category $\ttCBorn_{R}$ is closed, symmetric monoidal, quasi-abelian, complete and co-complete. It has enough flat projectives.
\end{lem}
\begin{lem}
Direct products in $\ttCBorn_{R}$ commute with cokernels.
\end{lem}
\proof 
Suppose we have $f_i:V_i\to W_i$ with cokernels $C_i$. Then we have a natural map $\coker(\prod f_i) \to \prod C_i$. Since $\ttCBorn_{R}$ has enough projectives, it has exact products by Proposition 1.4.5 of \cite{SchneidersQA}. Therefore $\prod  W_i \to \prod C_i$ is a strict epimorphism as it is the product of strict epimorphisms. The kernel is $\prod  V_i$ so we are done.
\endproof
By Proposition \ref{prop:ScnMain} we have:
\begin{lem}
If $R$ is a Banach ring the categories $\ttInd(\ttBan_R)$ (or $\ttInd(\ttBan^{na}_R)$ for $R$ non-archimedean) is a closed, symmetric monoidal, complete and co-complete elementary quasi-abelian category. 
\end{lem}
\begin{defn}
A Fr\'{e}chet module over $R$ is an object of $\ttInd(\ttBan_{R})$ which is a countable limit of a diagram in $\ttBan_{R}$. We consider Fr\'{e}chet modules as a full subcategory of $\ttCBorn_{R}$.
\end{defn}

Note that many function spaces in complex analytic geometry carry natural Fr\'{e}chet structures or more generally locally convex structures. We would like to relate these to the category $ \ttInd(\ttBan_{\mathbb{C}})$.  Let $\ttTc$ be the category of locally convex 
topological vector spaces over $\mathbb{C}$ and $\ttFr$ the sub-category of Fr\'{e}chet spaces. Note that both of these categories are quasi-abelian but they don't share all of the nice properties of $\ttInd(\ttBan_{\mathbb{C}})$ such as having enough projectives and having a closed symmetric monoidal structure. The following definition is  \cite{PS} definition 1.1:

\begin{defn}
For any object $E$ of $\ttTc$ let $\mathcal{B}_E$ be the set of 
absolutely convex bounded subsets of $E$. Given $B\in \mathcal{B}_E$, let $E_B$ be the linear span of $B$ with its gauge semi-norm $p_B$. Let 
\begin{equation*} 
\IB:\ttTc\to \ttInd(\ttBan_{\mathbb{C}})
\end{equation*} 
be defined as 
\begin{equation*}
\IB(E)=\colim \widehat{E_B}
\end{equation*}
where the colimit is taken over the directed system $\mathcal{B}_E$. Given $f:E\to F$ in $\ttTc$ and $B\in \mathcal{B}_E$, then $f(B)\in \mathcal{B}_F$. Hence we get 
a natural map $\colim \widehat{E_B} \to \colim \widehat{F_{f(B)}}$. 
Composing this with the canonical map $\underset{B\in \mathcal{B}_E}\colim \widehat{F_{f(B)}}\to \underset{B'\in \mathcal{B}_F}\colim  \widehat{F_{B'}}$ we get the functoriality of $\IB$.
\end{defn}

Note that if $E$ is a Banach space then $\IB(E)=E$. A subset of an object of $\ttTc$ is called bounded when it can be absorbed by scaling an open neighborhood of the origin. An object of $\ttTc$ is called bornological if every seminorm that is bounded on bounded sets is continuous. The functor $\IB$ is in many cases fully faithful (\cite{PS} proposition 1.5):

\begin{prop}
Let $E,F$ be objects of $\ttTc$. Assume that $E$ is bornological and that $F$ is complete. Then
\begin{equation*}
\Hom_\ttTc(E,F)=\Hom_\ttIndBan(\IB(E),\IB(F)).
\end{equation*}
\end{prop}
We also have that $\IB$ has a left adjoint when restricted to the category of complete locally convex vector spaces (\cite{PS} proposition 1.6).
\section{Nuclear Modules}
\subsection{Nuclear Banach Modules}\label{NucBanMod}
Let $\tC$ be a closed symmetric monoidal category with monoidal structure $\ootimes$ and unit $e$. This subsection will be applied only in the case that $\ttC$ is one of four categories. For a general Banach ring we consider $\ttBan_{R}$ and $\ttBan^{\leq 1}_{R}$. Also, if $R$ is non-archimedean, we could still consider those but also in this case we could consider $\ttBan^{na}_{R}$ and $\ttBan^{na, \leq 1}_{R}$. These definitions and results \emph{will not be applied} to the Ind-categories we consider, on the other hand, in sub-section \ref{ssec:nibm} we will separately define nulclear objects of Ind-categories and work with objects in Ind-categories in a way that extends the definitions given in this subsection. 
\begin{defn}  Let $V$ and $W$ be Banach modules. A element of $\Hom(V,W)$ is called \emph{nuclear} if it lies in the image of the composition
\[\Hom(e, V^{\vee} \ootimes W) \to \Hom(e, \uHom(V,W)) = \Hom(V,W).
\]
An object is called \emph{nuclear} if the identity morphism of this object is nuclear.
\end{defn}
From Higgs and Rowe \cite{HiRo}
\begin{lem}\label{lem:NuclearProps}If a morphism is nuclear then so is its dual. If a morphism is nuclear then so is any pre or post composition with it. The monoidal product of two nuclear morphisms is nuclear. \end{lem}
\proof  The first two statements can be found in Proposition 2.2 of \cite{HiRo}. The statement about the monoidal product can be found in Proposition 2.3 of \cite{HiRo}
\endproof
\begin{lem}\label{lem:NuclearProps2}
The following are equivalent  \cite{HiRo}:
\begin{enumerate}
\item The object $V$ is nuclear.
\item The natural morphism $V^{\vee}\ootimes V \to \uHom(V, V)$ is an isomorphism.
\item For every object $W$, $W\ootimes V^{\vee} \to \uHom(V,W)$ is an isomorphism.
\item For every object $W$, $V\ootimes W^{\vee} \to \uHom(W,V)$ is an isomorphism.
\end{enumerate}
\end{lem}
\begin{lem}\label{cor:DualIsNuc} \cite{HiRo} If an object $V$ is nuclear then its dual is also nuclear. Any nuclear object is reflexive. 
\end{lem}

\begin{lem}
Let $V$ be a nuclear object of $\ttC$.  
\begin{enumerate}
\item
The functor
\[\ttC\to \ttC
\]
given by 
\[W\mapsto \uHom(V, W)
\]
preserves strict epimorphisms. If in addition, $e$ is projective, then $V$ is projective. 
\item
The functor 
\[\ttC\to \ttC
\]
given by 
\[W\mapsto W\ootimes V
\]
preserves strict monomorphisms ($V$ is flat). 
\end{enumerate}
\end{lem}
\proof 
\begin{enumerate}
\item
Since $V$ is nuclear, we can consider the naturally isomorphic functor
 \[W\mapsto V^{\vee}\ootimes W
\]
By Lemma \ref{lem:StrictEpiPreserving}, this functor preserves strict epimorphisms so we are done.
\item
Because $V$ is reflexive, we have $W\ootimes V \cong W\ootimes (V^{\vee \vee}) \cong \uHom(V^{\vee}, W).$ Therefore, we can consider the naturally isomorphic functor
\[W \mapsto  \uHom(V^{\vee}, W).
\]
This preserves strict monomorphisms by Lemma \ref{lem:StrictEpiPreserving}.
\end{enumerate}

\endproof
The following Lemma was shown over $\mathbb{C}$ in \cite{PS}. The following is our version over a Banach ring $R$.
\begin{lem}\label{lem:DecompNuc}
Let $V$ and $W$ be two Banach modules over a Banach ring $R$ and let $f : V \to W$ be a nuclear
morphism in $\ttBan_{R}$. Then there exists a countable set $S$ and a map $m:S \to \mathbb{R}_{>0}$,  a nuclear morphism $p : V\to  \underset{s \in S}\coprod^{\leq 1} R_{m(s)}$ and a non-expanding morphism $c :  \underset{s \in S}\coprod^{\leq 1} R_{m(s)} \to W$
such that $f=c\circ p$.
\end{lem}
\proof 
Let $P$ be the element of $W \wotimes_{R}V^{\vee}$ corresponding to $f$. We have a countable set $S$ so that $P$ is a sum $\underset{s \in S}\sum w_{s} \wotimes \alpha_{s}$ where $L=\underset{s \in S}\sum||w_{s}||  ||\alpha_{s}||<\infty$. Let $m(s) = ||w_{s}||$. We define \[p: V \to \underset{s \in S}\coprod^{\leq 1} R_{m(s)}\]
by $p(v)=(\alpha_{s}(v))_{s\in S}$ where 
\[||p(v)|| \leq  \underset{s \in S}\sum m(s) ||\alpha_s||||v||=L||v||. 
\]  
The morphism $p$ is actually nuclear as it can be written as $\underset{s \in S}\sum\delta_s \wotimes \alpha_s$ where $\delta_s$ is the vector with $1\in R_{m(s)}$ in position $s$ and $0$ elsewhere.
Define $c$ by $c(\mu)=\underset{s \in S}\sum \mu_s w_{s}$ where $\mu = (\mu_{s})_{s\in S}$.
We have 
\begin{equation}
\begin{split}
||c(\mu)|| & \leq  \underset{s \in S}\sum |\mu_s|||w_s||  = \underset{s \in S}\sum |\mu_s|||w_s|| m(s) m(s)^{-1} \\ & \leq ( \underset{t \in S}\sum |\mu_t|m(t))( \underset{s \in S}\sup ||w_s|| m(s)^{-1})= \underset{t \in S}\sum |\mu_t|m(t)
=  ||\mu|| 
\end{split}
\end{equation}
which shows that $c$ is non-expanding. For any element $v\in V$ we have $p(v)=(\alpha_s (v))_{s\in S}$ so $c(p(v))=\underset{s \in S}\sum \alpha_{s}(v)w_{s}  =f(v)$.
\endproof
\begin{lem}\label{lem:DecompNuc2}
Let $V$ and $W$ be two Banach modules over a Banach ring $R$ and let $f : V \to W$ be a nuclear
morphism in $\ttBan_{R}$. Then there exists a countable set $S$ and a map $m:S \to \mathbb{R}_{>0}$,  a non-expanding morphism $c : V\to  \underset{s \in S}\prod^{\leq 1} R_{m(s)}$ and a nuclear morphism $p :  \underset{s \in S}\prod^{\leq 1} R_{m(s)} \to W$
such that $f=p\circ c$.
\end{lem}
\proof 
Let $P$ be the element of $W \wotimes_{R}V^{\vee}$ corresponding to $f$. We have a countable set $S$ so that $P$ is a sum $\underset{s \in S}\sum w_{s} \wotimes \alpha_{s}$ where $L=\underset{s \in S}\sum||w_{s}||  ||\alpha_{s}||<\infty$. Let $m(s) = L^{-1}||w_s||$. We define \[c: V \to \underset{s \in S}\prod^{\leq 1} R_{m(s)}\]
by $c(v)=(\alpha_{s}(v))_{s\in S}$ where 
\[||c(v)|| \leq  \underset{s \in S}\sup \ \ m(s) ||\alpha_s||||v|| \leq ||v||. 
\]  
Define $p$ by $p((\mu_{s})_{s\in S})=\underset{s \in S}\sum \mu_s w_{s}$.
We have 
\[||p(u)||\leq   \underset{s \in S}\sum |\mu_s|||w_s|| m(s) m(s)^{-1} \leq ( \underset{s \in S}\sum ||w_s|| m(s)^{-1})( \underset{t \in S}\sup |\mu_t|m(t)) =  ( \underset{s \in S}\sum ||w_s|| m(s)^{-1})||u||
\]
and so $p$ is bounded. 
The morphism $p$ is actually nuclear as it can be written as $\underset{s \in S}\sum\delta_s \wotimes w_s$ where $\delta_s$ is the vector with $1\in R_{m(s)}$ in position $s$ and $0$ elsewhere. For any element $v\in V$ we have $c(v)=(\alpha_s (v))_{s\in S}$ so $p(c(v))=\underset{s \in S}\sum \alpha_{s}(v)w_{s}  =f(v)$.
\endproof
\begin{rem}If $R$ is non-archimedean, all of this subsection goes through for $\ttBan^{na}_R$ in place of $\ttBan_R$.
\end{rem}
\subsection{Nuclear Ind-Banach Modules}\label{ssec:nibm} This subsection is about nuclear objects in $\ttInd(\ttBan_{R})$  or if $R$ is non-archimedean, about nuclear objects in  $\ttInd(\ttBan^{na}_{R})$. For readability, we suppress the non-archimedean versions, all the statements and proofs in the non-archimedean case are the same, up to the obvious substitutions. Of course in the non-archimedean version, all categorical constructions in $\ttBan^{\leq 1}_{R}$ are replaced by those in $\ttBan^{\leq 1, na}_{R}$ and $\wotimes_{R}$ is replaced by  $\wotimes^{na}_{R}$. As the beginning is more general, we work with a closed symmetric monoidal category $C$ with monoidal structure $\ootimes$ and unit $e$ with finite limits and colimits, but the reader is invited to take $C=\ttBan_{R}$, $\ootimes=\wotimes_{R}$ and $e=R$ or if $R$ is non-archimedean there is also the option $C=\ttBan^{na}_{R}$, $\ootimes=\wotimes^{na}_{R}$ and $e=R$.
\begin{defn}\label{defn:IndBanNuc} An object $V$ in $\ttInd(C)$ is called nuclear if for every object $W$ of $C$ the natural morphism
\[W^{\vee} \ootimes V \to \uHom(W,V)\]
is an isomorphism.
\end{defn}
\begin{rem}
In general, an object which is nuclear under this definition is not nuclear in the sense of subsection \ref{NucBanMod} applied to the category $\ttInd(C)$. If we consider an object $V$ of $\ttInd(C)$ which happens to be in $C\subset \ttInd(C)$ itself then the definitions agree by Lemma \ref{lem:NuclearProps2}. Later, we find another situation when the two definitions agree in Lemma \ref{lem:MayerBig}.
\end{rem}
\begin{lem}\label{lem:twothreenuc} If $0\to V_1 \to V_2 \to V_3 \to 0$ is a strict short exact sequence in $\ttInd(C)$ then if $V_1$ and $V_3$ are nuclear then $V_2$ is as well.
\end{lem}
\proof 
This follows immediately from the fact that the assumption allows us to identify the strict sequences
\[ W^{\vee} \ootimes V_1 \to W^{\vee} \ootimes V_2   \to W^{\vee} \ootimes V_3 \to 0
\]
and
\[0 \to \uHom(W,V_1) \to \uHom(W,V_2)  \to \uHom(W,V_3).
\]
Hence both can be completed to strict short exact sequences. Now because the outer terms are identified by assumption, the needed isomorphism also holds for $V_2$ and hence it is nuclear.
\endproof

\begin{lem}\label{lem:NucMap} For any nuclear object $V$ of $\ttInd(C)$ and an arbitrary object $W$ of $\ttInd(C)$ represented as $V= ``\underset{i \in I}\colim" V_i$ and $W= ``\underset{j \in J}\colim" W_j$ any morphism in $\Hom(W, V)$ can be represented in terms of a system of nuclear maps in $C$, $W_j \to V_i$.
\end{lem}
\proof 
Notice that 
\begin{equation}\begin{split}\Hom(W, V) & = \Hom(e, \uHom(W, V))=  \Hom(e, \underset{j \in J}\lim \uHom(W_j, V)) \cong \Hom(e, \underset{j \in J}\lim (V \ootimes W^{\vee}_j))
\\
&= \underset{j \in J}\lim \Hom(e,  V \ootimes W^{\vee}_j)
=  \underset{j \in J}\lim \Hom(e,  (``\underset{i \in I}\colim" V_i)\ootimes W^{\vee}_j) \\
& =  \underset{j \in J}\lim \Hom(e,  ``\underset{i \in I}\colim" (V_i\ootimes W^{\vee}_j)) =  \underset{j \in J}\lim \ \ \underset{i \in I}   \colim\Hom(e, V_i \ootimes W^{\vee}_j).
\end{split}  
\end{equation}
On the other hand, by definition 
\[\Hom(W, V) =  \underset{j \in J}\lim \ \ \underset{i \in I}   \colim\Hom(W_j, V_i) = \underset{j \in J}\lim \ \ \underset{i \in I}   \colim\Hom(e, \uHom(W_j, V_i)).
\]
Therefore, the canonical map 
\begin{equation}
\begin{split}
\Hom(W,V) \cong  &  \ \ \underset{j \in J}\lim \ \ \underset{i \in I}   \colim\Hom(e, V_i\ootimes W^{\vee}_j) \to  \\& \ \  \underset{j \in J}\lim \ \ \underset{i \in I}   \colim\Hom(e, \uHom(W_j, V_i)) =  \underset{j \in J}\lim \ \ \underset{i \in I}   \colim\Hom(W_j, V_i) 
\end{split}
\end{equation}
is an isomorphism and so for any element $\phi$ in $\Hom(W, V)$ and for any $j \in J$ there exists $i \in I$ and a nuclear map $\phi_{i,j}:W_j \to V_i$ assigned to it. The collection of the $\phi_{i,j}$ define a morphism of inductive systems representing $\phi$. 
\endproof
\begin{lem}\label{lem:SameSame}Given a filtered inductive system $W_{i}$ of $C$ where all the maps in the system are nuclear, then the object of $\ttInd(C)$ given by $``\underset{i\in I}\colim"W_i$ is nuclear in the sense of Definition \ref{defn:IndBanNuc}.
\end{lem}

\proof 
Let $V$ be any object of $C$.  Consider the canonical morphism
\[f:(``\underset{i\in I}\colim"W_i)\ootimes V^{\vee} \to \uHom(V,``\underset{i\in I}\colim"W_i).\] Since $V$ is in $C$ (hence a compact object in $\ttInd(C)$) we can equivalently write this as the colimit of the morphisms 
\[W_{i} \ootimes V^{\vee} \to \uHom(V,W_i).
\]
Consider any $W_{i} \to W_{j}$ in the system corresponding to a non-identity arrow $i\to j$. Since they are nuclear, the precomposition $V \to W_{i} \to W_{j}$ is also nuclear by Lemma \ref{lem:NuclearProps}. Therefore it lies in the image of $\Hom(e, V^{\vee}\ootimes W_{j})$. This constructs a two-sided inverse 
\[(``\underset{i\in I}\colim"W_i)\ootimes V^{\vee} \cong ``\underset{i\in I}\colim"(W_i\ootimes V^{\vee}) \leftarrow ``\underset{i\in I}\colim"\uHom(V,W_i) \cong \uHom(V,``\underset{i\in I}\colim"W_i)
\]
to $f$.
\endproof

\begin{lem}\label{lem:NucTranMaps}
If an object $W$ is nuclear in $\ttInd(C)$ and presented as $``\underset{i \in I}\colim" W_i$ for $I$ a filtering ordered set with transition maps $\phi_{ij}:W_{i}\to W_{j}$, then for each $i \in I$ there exists a $j\in I$ with $j>i$ such that morphism $\phi_{ij}$ is nuclear.
\end{lem}
\proof 
Consider Lemma \ref{lem:NucMap} in the case $V=W= ``\underset{i \in I}\colim" W_i$ in the case of identical inductive systems applied to the element $\text{id} \in \Hom(W,W)$. A representative of the identity is given by a cofinal choice of transition maps. The lemma provides the nuclear maps $\phi_{ij}$ representing the identity, which are therefore transition maps in the given presentation of $W$. 
\endproof
\begin{rem}\label{rem:summary}Because of Lemmas \ref{lem:SameSame} and \ref{lem:NucTranMaps} we can conclude that nuclear objects are just those representable by an ind-sytem with nuclear transition maps.
\end{rem}
\begin{lem}\label{lem:tensnuc}If $X$ and $Y$ are nuclear in $\ttInd(C)$ then so is $X\ootimes Y$.
\end{lem}
\proof 
If we consider Remark \ref{rem:summary} and choose presentations of $X$ and $Y$ with nuclear transition maps, then the induced monoidal structure presentation of $X\ootimes Y$ also has nuclear transition maps by Lemma \ref{lem:NuclearProps}.
\endproof
\begin{lem}\label{lem:NicePres}Given $I$, an infinite filtering ordered set, and a functor $I \to \ttBan_{R}$ such that the corresponding object $W= ``\underset{i \in I}\colim" W_i$ is nuclear in $\ttInd(\ttBan_{R})$, there is a filtered category $K$ with the same cardinality of objects and morphisms as $I$, a functor $K\to  \ttBan_{R}$ with corresponding object $P = ``\underset{k \in K}\colim" P_k$ and an isomorphism $W\cong P$ such that each Banach space $P_k$ is a countable coproduct in $\ttBan^{\leq 1}_{R}$ of weighted copies of $R$ and the transition functions are nuclear.
\end{lem}
\proof 
Using Lemma \ref{lem:NucTranMaps} we may assume that for each $i$ there is a $j>i$ such that $\phi_{ij}:W_i \to W_j$ 
is nuclear. Define 
\[K = \{(i,j) \in I \times I \ \ | \ \ j \geq i, \phi_{ij}:W_i \to W_j \text{ is nuclear} \}.
\]
Using Lemma \ref{lem:DecompNuc} we can decompose each such morphism $\phi_{ij}$ into $c_{ij} \circ p_{ij} :W_{i} \to P_{ij} \to W_{j}$ where $P_{ij}$ is a countable non-expanding coproduct of weighted copies of $R$ and $p_{ij}$ is nuclear. Given two pairs $k=(i,j)$ with $j\geq i$ and $k'=(i',j')$  with $j'\geq i'$ of $K$ we define the nuclear (see Lemma \ref{lem:NuclearProps}) morphism $n_{kk'}: P_k \to P_{k'}$ by $p_{k'} \circ c_{k}$ for any pair $k,k'$ such that $j=i'$. As in the proof of Lemma 2.3 of \cite{PS} we have a filtering inductive system $(K, \{P_k\}, \{n_{kk'}\})$ which defines an object $P$ of $\ttInd(\ttBan_{R})$ isomorphic to $W$ which has the desired properties.
\endproof
\begin{lem}\label{lem:NicePres2}Given $I$, an infinite filtering ordered set and a functor $I \to \ttBan_{R}$ such that the corresponding object $W= ``\underset{i \in I}\colim" W_i$ is nuclear in $\ttInd(\ttBan_{R})$, there is a filtered category $K$ with the same cardinality of objects and morphisms as $I$, a functor $K\to  \ttBan_{R}$ with corresponding object $L = ``\underset{k \in K}\colim" L_k$ and an isomorphism $W\cong L$ such that each Banach space $L_k$ is a countable product in $\ttBan^{\leq 1}_{R}$ of weighted copies of $R$ and the transition functions are nuclear.
\end{lem}
\proof Using Lemma \ref{lem:NucTranMaps} we may assume that for each $i$ there is a $j>i$ such that $\phi_{ij}:W_i \to W_j$ 
is nuclear. Define 
\[K = \{(i,j) \in I \times I | j \geq i, \phi_{ij}:W_i \to W_j \text{ is nuclear} \}.
\]
Using Lemma \ref{lem:DecompNuc2} we can decompose each such morphism $\phi_{ij}$ into $p_{ij} \circ c_{ij} :W_{i} \to L_{ij} \to W_{j}$ where $L_{ij}$ is a countable non-expanding product of weighted copies of $R$ and $p_{ij}$ is nuclear. Given two pairs $k=(i,j)$ with $j\geq i$ and $k'=(i',j')$  with $j'\geq i'$ of $K$ we define the nuclear (see Lemma \ref{lem:NuclearProps}) morphism $n_{kk'}: L_k \to L_{k'}$ by $c_{k'} \circ p_{k}$ for any pair $k,k'$ such that $j=i'$. As in the proof of Lemma 2.3 of \cite{PS} we have a filtering inductive system $(K, \{L_k\}, \{n_{kk'}\})$ which defines an object $L$ of $\ttInd(\ttBan_{R})$ isomorphic to $W$ which has the desired properties.
\endproof

\begin{lem}\label{lem:NucImpliesFlat}Any nuclear object of $\ttInd(\ttBan_{R})$ is flat in $\ttInd(\ttBan_{R})$.
\end{lem}
\proof 
Using Lemma \ref{lem:NicePres} we can write a nuclear object in a certain nice form as the formal filtered colimit of countable coproducts in $\ttBan^{\leq 1}_{R}$ of weighted copies of $R$. Each weighted copy of $R$ is projective by Lemma \ref{lem:WeightsProj}. Therefore their coproduct in $\ttBan^{\leq 1}_{R}$ is projective by Lemma \ref{lem:colimProps} and hence flat by Lemma \ref{lem:Proj2Flat}.  By Lemma \ref{lem:colimProps} this colimit of flat objects is flat.
\endproof
\begin{rem}If $R$ is non-archimedean, all of this subsection goes through for $\ttBan^{na}_R$ in place of $\ttBan_R$.
\end{rem}
\section{The interaction of products and tensor products}
Just as flatness in our context  is about commuting kernels and completed tensor products, we need to investigate the interaction of the other type of limit (products) with the completed tensor product.  Many of the results in this section are either taken from or inspired by the book \cite{M} by R. Meyer and discussions with him. In this section we define metrizable modules and examine how the tensor product with them interacts with products. We work in a general context and as usual, for simplicity, we look at the archimedean setting, even if $R$ is non-archimedean. In that setting, all the proofs in this section go through with the obvious modifications for $\ttBan^{na}_R$ in place of $\ttBan_R$.
\begin{defn}Let $\lambda$ be a cardinal. A poset $J$ is called $\lambda$-filtered if any subset $S$ of $J$ with $|S|<\lambda$ has an upper bound. 
\end{defn}
\begin{rem}\label{rem:ProdFilt}Let $\lambda$ be a cardinal. A finite product of $\lambda$-filtered posets is $\lambda$-filtered.
\end{rem}
\begin{lem}\label{lem:SETlimcolim} Let $\lambda$ be a cardinal. Suppose that $I$ is a poset and we are given a functor $F:I \times J \to Set$ where $J$ has cardinality less than $\lambda$ and $I$ is $\lambda$-filtered. Then the natural morphism 
\[\underset{i\in I}{\colim} \lim_{j\in J} F(i,j) \to \lim_{j\in J} \underset{i\in I}{\colim} F(i,j)
\]
is an isomorphism. 
\end{lem}
\proof 
This is well known in set theory. For example when $\lambda=\aleph_{0}$, one can consider sets $X_{1} \subset X_{2} \subset X_{3} \subset \cdots $, $Y_{1} \subset Y_{2} \subset Y_{3} \subset \cdots $, and $Z_{1} \subset Z_{2} \subset Z_{3} \subset \cdots $, along with maps $X_{i} \rightarrow Z_{i} \leftarrow Y_{i}$ compatible with inclusions and then the claim is that $(\bigcup_{i} X_{i})\times_{\bigcup_{i} Z_{i}}(\bigcup_{i} Y_{i})=\bigcup_{i}(X_{i}\times_{Z_i} Y_{i})$ as can be shown by hand.
\endproof
By considering objects in $\ttInd(C)$ as functors from $C^{op}$ to sets, Lemma \ref{lem:SETlimcolim} immediately implies the following.
\begin{lem}\label{lem:limcolim} Let $\lambda$ be a cardinal. Suppose that $I$ is a poset and we are given a functor $F:I \times J \to \ttInd(C)$ where $J$ has cardinality less than $\lambda$ and $I$ is $\lambda$-filtered. Then the natural morphism 
\[\underset{i\in I}{\colim} \lim_{j\in J} F(i,j) \to \lim_{j\in J} \underset{i\in I}{\colim} F(i,j)
\]
is an isomorphism. 
\end{lem}
Suppose now that $C$ is a closed symmetric monoidal category.
\begin{defn}\label{defn:met}An object $V$ of $\ttInd(C)$ will be called metrizable if the category whose objects consist of objects of $C$ along with morphisms to $V$ and whose morphisms are commuting triangles is $\aleph_{1}$-filtered. 
\end{defn}
\begin{lem}An object $V$ of $\ttInd(C)$ is metrizable if and only if there is an $\aleph_{1}$-filtered category $I$, a functor $F:I \to C$ and an isomorphism $V \cong \underset{I}\colim F$.
\end{lem}
\proof If there exists a functor as in the statement of the lemma, any morphism $W \to V$ would factor via $F(i)$ for some object $i \in I$.
\endproof
\begin{cor}\label{cor:HomRewrite} Let $V =  ``\underset{i \in I}{\colim}" V_{i} \in \ttInd(C)$ where $V_i \in C$ and $I$ has cardinality less than $\lambda$. Let $W= ``\underset{j \in J}{\colim}" W_{j}$ where $J$ is $\lambda$-filtered. Then there is an isomorphism 
\[\uHom(V, W) \cong \underset{j \in J}{\colim} \uHom(V, W_{j}).
\]
\end{cor}
{\bf Proof.} By Lemma \ref{lem:limcolim} we have 
\[\uHom(V, W)  \cong \lim_{i\in I} ``\underset{j\in J}{\colim}"\uHom(V_i, W_j) \cong \underset{j\in J}{\colim} \lim_{i\in I} \uHom(V_i, W_j) \cong \underset{j\in J}{\colim}  \uHom(V, W_j).
\]
\begin{lem}\label{lem:MayerBig}Let $V =  ``\underset{i \in I}{\colim}" V_{i} \in \ttInd(C)$ where $V_i \in C$ and $I$ has cardinality less than $\lambda$. Let $W= ``\underset{j \in J}{\colim}" W_{j}$ where $J$ is $\lambda$-filtered. Assume that $W$ is nuclear in $\ttInd(C)$,
then 
\[\uHom(V, W) \cong V^{\vee} \ootimes W
\]
\end{lem}
\proof Using Lemma \ref{lem:NucTranMaps}, without loss of generality we can assume that $W$ is presented by a system where all the structure morphisms are nuclear. Consider the morphism
\[(``\underset{i \in I}\colim"W_i)\ootimes V^{\vee} \to \uHom(V,``\underset{i \in I}\colim"W_i).\] By Corollary \ref{cor:HomRewrite} we can equivalently write this as the colimit of the morphisms 
\[W_{i} \ootimes V^{\vee} \to \uHom(V,W_i).
\]
Consider any $W_{i} \to W_{j}$ in the system corresponding to a non-identity arrow $i\to j$. Since they are nuclear, the precomposition $V \to W_{i} \to W_{j}$ is also nuclear by Lemma \ref{lem:NuclearProps}. Therefore, it lies in the image of $\Hom(e, V^{\vee}\ootimes W_{j})$. This constructs an inverse 
\[(``\underset{i \in I}\colim"W_i)\ootimes V^{\vee} \leftarrow \uHom(V,``\underset{i \in I}\colim"W_i).
\]
\endproof
\begin{cor}\label{cor:limsprods}Let $V =   ``\underset{i \in I}{\colim}" V_{i} \in \ttInd(C)$ where $V_i \in C$ and $I$ has cardinality less than $\lambda$. Let $W= ``\underset{j \in J}{\colim} "W_{j}$ where $J$ is $\lambda$-filtered. Assume that $W$ is nuclear in $\ttInd(C)$. Then the natural morphism
\[(\underset{i \in I}\lim (V^{\vee}_{i}))\ootimes W \to \underset{i \in I}\lim (V^{\vee}_{i} \ootimes W)
\]
is an isomorphism. Also, if $\lambda=\aleph_1$ and $V_1, V_2, V_3, \dots$ is a countable list of objects in $\ttC,$ the natural morphism
\[(\underset{i \in I}\prod (V^{\vee}_{i}))\ootimes W \to \underset{i \in I}\prod (V^{\vee}_{i} \ootimes W)
\]
is an isomorphism. 
\end{cor}
\proof 
The left hand side is $V^{\vee}\ootimes W,$ which is isomorphic to 
\[\uHom(V, W)\cong \underset{i \in I}\lim \uHom(V_i, W) \cong \underset{i \in I}\lim (V_{i}^{\vee} \ootimes W)\]  
where we have used Lemma \ref{lem:MayerBig}. 
For the second statement let $I=\mathbb{Z}_{>0}$ and just consider the system \[V_1\to V_1\oplus V_2 \to  V_1\oplus V_2 \oplus V_3 \to \cdots\] and apply the statement already proven.
\endproof

\begin{defn}\label{defn:PsiUpsilon} Let $\Psi$ be the poset consisting of functions $\psi:I \to \mathbb{Z}_{\geq 1}$ with the order $\psi_{1}\leq \psi_{2}$ if $\psi_{1}(i)\leq \psi_{2}(i)$ for all $i \in I$. Let $\Upsilon$ be the poset consisting of functions $\psi:I \to \mathbb{Z}_{\geq 1}$ with the order $\psi_{1}< \psi_{2}$ if $\psi_{1}(i)< \psi_{2}(i)$ for all $i \in I-J$ where $J$ is a finite subset of $I$. The categories $\Psi$ and $\Upsilon$ with objects $\underset{i \in I}\prod\mathbb{Z}_{\geq 1}$ can be thought of as categories of maps $I\to \mathbb{Z}_{\geq 1}$.
\end{defn}
At this point in the subsection, we need to reduce the generality and take $C=\ttBan_{R}$ for a Banach ring $R$. Of course, as usual, if $R$ is non-archimedean, one can use $C=\ttBan^{na}_{R}$ instead with the obvious modifications, which we suppress to save space.
\begin{lem}\label{lem:towerRep} Suppose we are given a family $(V_i)_{i \in I}$ 
in $\ttBan_{R}$ (n.b. not in $\ttInd(\ttBan_{R})$) indexed by a set $I$. Then the natural morphism in $\ttInd(\ttBan_{R})$
\[``\underset{\psi \in \Psi}\colim" \prod_{i \in I}{}^{\leq 1} ((V_{i})_{\psi(i)^{-1}}) \to \prod_{i\in I}V_{i} 
\]
is an isomorphism in $\ttInd(\ttBan_{R})$ where the product on the right is taken in $\ttInd(\ttBan_{R})$, the product on the left is taken in  $\ttBan^{\leq 1}_{R}$ and the notation $(V_{i})_{\psi(i)^{-1}}$ uses Definition \ref{defn:scaling}. 
\end{lem}

{\bf Proof.}
 It is enough to show that the morphisms \[\Hom(M, ``\underset{\psi \in \Psi}\colim" \underset{i\in I}\prod{}^{\leq 1} ((V_{i})_{\psi(i)^{-1}})) \to \Hom(M, \underset{i\in I}\prod V_{i}) \] are  isomorphisms of abelian groups for any $M \in \ttBan_{R}$. We have 
\begin{equation}
\begin{split}\Hom(M, ``\underset{\psi \in \Psi}\colim" \prod_{i \in I}{}^{\leq 1} ((V_{i})_{\psi(i)^{-1}})) & \cong \underset{\psi \in \Psi}\colim \Hom(M,  \prod_{i \in I}{}^{\leq 1} ((V_{i})_{\psi(i)^{-1}}))  \\
 & \cong \underset{\psi \in \Psi} \colim   \ \ \underset{j \in \mathbb{Z}_{>0}}\colim\Hom^{\leq j}(M,  \prod_{i \in I}{}^{\leq 1} ((V_{i})_{\psi(i)^{-1}}))  \\ 
   & \cong \underset{\psi \in \Psi} \colim   \ \ \underset{j \in \mathbb{Z}_{>0}}\colim\Hom^{\leq 1}(M,  (\prod_{i \in I}{}^{\leq 1} ((V_{i})_{\psi(i)^{-1}}))_{j^{-1}})  \\
  & \cong \underset{\psi \in \Psi} \colim   \ \ \underset{j \in \mathbb{Z}_{>0}}\colim\Hom^{\leq 1}(M,  \prod_{i \in I}{}^{\leq 1} ((V_{i})_{(j\psi(i))^{-1}}))  \\
& \cong \underset{\psi \in \Psi} \colim   \Hom^{\leq 1}(M,  \prod_{i \in I}{}^{\leq 1} ((V_{i})_{\psi(i)^{-1}}))  \\
&  \cong \underset{\psi \in \Psi}\colim  \prod_{i \in I}  \Hom^{\leq 1}(M,   (V_{i})_{\psi(i)^{-1}}) \\
& \cong   \prod_{i \in I}\underset{j \in \mathbb{Z}_{>0}}\colim \Hom^{\leq 1}(M,   (V_{i})_{j^{-1}}) \\
& =   \prod_{i \in I} \underset{j \in \mathbb{Z}_{>0}}\colim \Hom^{\leq j}(M,   V_{i}) \\
& =   \prod_{i \in I}  \Hom(M,   V_{i}) \\
& \cong     \Hom(M,   \prod_{i\in I} V_{i}).
\end{split}
\end{equation}
Notice here that in order to pass from colimits over $\Psi$ to colimits over $ \mathbb{Z}_{>0}$, in the isomorphism \[ \underset{\psi \in \Psi}\colim  \prod_{i \in I}  \Hom^{\leq 1}(M,   (V_{i})_{\psi(i)^{-1}}) \cong \prod_{i \in I}\underset{j \in \mathbb{Z}_{>0}}\colim \Hom^{\leq 1}(M,   (V_{i})_{j^{-1}})\] we have used that in the category of sets, products and filtered colimits distribute (not commute!). 
This means that for a set indexed by $I$ of filtered sets $\{S_{i,j} \}_{j\in J}$ we have
\[ \prod_{i \in I}\underset{j \in \mathbb{Z}_{>0}}\colim S_{i,j}\cong \underset{\psi \in \underset{i \in I}\prod\mathbb{Z}_{>0}}\colim  \prod_{i \in I} S_{i, \psi(i)}.
\]
See for instance \cite{ALR} and \cite{AR}, where this is explained. 
\hfill $\Box$

\begin{lem}\label{lem:ProdFre} Suppose we are given a countable family $(V_i)_{i \in I}$ 
in $\ttBan_{R}$ (n.b. not in $\ttInd(\ttBan_{R})$) indexed by a set $I$. Then $\underset{i\in I}\prod V_i$ is metrizable in $\ttInd(\ttBan_{R})$. 
\end{lem}
\proof 
Notice that $\Psi$ is a non-full subcategory of $\Upsilon$ but they have the same objects. The functor $\Psi\to \ttBan_{R}$ which sends $\psi$ to $\underset{i\in I}\prod{}^{\leq 1} (V_{i})_{\psi(i)^{-1}}$ can be extended to a functor $\Upsilon \to \ttBan_{R}$.  Indeed suppose that $\psi_{1}(i)< \psi_{2}(i)$ for all $i \in I-J$ where $J$ is a finite subset of $I$. Then we have
\[\underset{i\in I}\sup ||v_i||_i \psi_2(i)^{-1}\leq \max\{ \underset{i\in I-J}\sup ||v_i||_i \psi_2(i)^{-1}, \underset{i\in J}\sup ||v_i||_i \psi_2(i)^{-1}\}\leq \max\{ \underset{i\in I-J}\sup ||v_i||_i \psi_1(i)^{-1}, \underset{i\in J}\sup ||v_i||_i \psi_2(i)^{-1}\}
\]
and therefore
\[\underset{i\in I}\sup ||v_i||_i \psi_2(i)^{-1}\leq C \underset{i\in I}\sup ||v_i||_i \psi_1(i)^{-1}
\]
where \[C=\max \huge\{1, c \huge\}\] and
\[c=\frac{\underset{i\in J}\sup ||v_i||_i \psi_2(i)^{-1}}{\underset{i\in J}\sup ||v_i||_i \psi_1(i)^{-1}}.
\]
The inclusion $\Psi\to \Upsilon$ is a final functor. This is because if $\psi_1(i) <\psi_2(i)$ for all $i$ in $I-J$ then by letting $\psi_3=\psi_1+\psi_2$ we have $\psi_1(i) <\psi_3(i)$ and $\psi_2(i) <\psi_3(i)$ for all $i\in I$. The category $\Upsilon$ is $\aleph_1$-filtered:  Given a countable collection $\psi_1, \psi_2, \dots$ of objects of $\Upsilon$, define $\alpha\in \Upsilon$ by $\alpha(i) = 1+ \sum_{k=0}^{i}\psi_k(i)$. Then clearly for all $i\geq j$ we have $\psi_j(i) < \alpha(i)$ and so in $\Upsilon$ we have that $\psi < \alpha$.
Hence the inclusion induces an isomorphism \[``\underset{\psi \in \Upsilon}\colim" \prod_{i \in I}{}^{\leq 1} ((V_{i})_{\psi(i)^{-1}})\to``\underset{\psi \in \Psi}\colim" \prod_{i \in I}{}^{\leq 1} ((V_{i})_{\psi(i)^{-1}}).\] Combining this with the isomorphism of Lemma \ref{lem:towerRep} we see that $\underset{i\in I}\prod V_i$ is metrizable. 
\endproof
\begin{cor}\label{cor:TowerLim}Suppose we are given a system 
\[\cdots \to V_4 \to V_3 \to V_2 \to V_1
\]
in
 $\ttBan_{R}$ (n.b. not in $\ttInd(\ttBan_{R})$) such that all morphisms in the system are in $\ttBan^{\leq 1}_{R}$. Let $\Psi$ be the poset consisting of non-decreasing functions $\psi:\mathbb{Z}_{\geq 1} \to \mathbb{Z}_{\geq 1}$ with the order $\psi_{1}\leq \psi_{2}$ if $\psi_{1}(i)\leq \psi_{2}(i)$ for all $i \in \mathbb{Z}_{\geq 1}$. 
Then the natural morphism in $\ttInd(\ttBan_{R})$
\[``\underset{\psi \in \Psi}\colim" \ker[\prod_{i \in \mathbb{Z}_{\geq 1}}{}^{\leq 1} ((V_{i})_{\psi(i)^{-1}}) \stackrel{\id-s}\longrightarrow \prod_{i \in \mathbb{Z}_{\geq 1}}{}^{\leq 1} ((V_{i})_{2^{-1}\psi(i+1)^{-1}})]\longrightarrow \lim_{i\in \mathbb{Z}_{\geq 1}}V_{i} 
\]
is an isomorphism in $\ttInd(\ttBan_{R})$. Furthermore, $\underset{i\in \mathbb{Z}_{\geq 1}}\lim V_{i} $ is metrizable and so any  Fr\'{e}chet module is metrizable.

\end{cor}
\proof 
First notice that 
\[\lim_{i\in \mathbb{Z}_{\geq 1}}V_{i} =\ker [\prod_{i \in \mathbb{Z}_{\geq 1}} V_{i} \stackrel{\id-s}\longrightarrow \prod_{i \in \mathbb{Z}_{\geq 1}} V_{i}].
\]
Using Lemma \ref{lem:towerRep} we can write $\underset{i \in \mathbb{Z}_{\geq 1}}\prod V_{i} \stackrel{\id-s}\to \underset{i \in \mathbb{Z}_{\geq 1}} \prod V_{i}$ as a map 
\begin{equation}\label{eqn:themap}``\underset{\psi \in \Psi}\colim" \prod_{i \in \mathbb{Z}_{\geq 1}}{}^{\leq 1} ((V_{i})_{\psi(i)^{-1}}) \to ``\underset{\psi \in \Psi}\colim" \prod_{i \in \mathbb{Z}_{\geq 1}}{}^{\leq 1} ((V_{i})_{\psi(i)^{-1}})
\end{equation}
however, the shift of an element $(v_i)_{i\in \mathbb{Z}_{\geq 1}}$ of $ \underset{i \in \mathbb{Z}_{\geq 1}}\prod{}^{\leq 1} ((V_{i})_{\psi(i)^{-1}}) $ lands in  $\underset{i \in \mathbb{Z}_{\geq 1}}\prod^{\leq 1} ((V_{i})_{\phi(i)^{-1}})$ whenever $\underset{i \in \mathbb{Z}_{\geq 1}}\sup ||v_i||_{i-1}\phi(i-1)^{-1}<\infty$. This happens as long as $\underset{i \in \mathbb{Z}_{\geq 1}}\sup ||v_i||_{i}\phi(i-1)^{-1}<\infty$ since the maps are non-expanding. For $\phi=\psi$, $\psi(i-1)^{-1}\geq \psi(i)^{-1}$ so there is no reason why this should be true. However, if we define $\phi=s\psi$ by $(s\psi)(i)=\psi(i+1)$ then we do have the map \[s:\prod_{i \in \mathbb{Z}_{\geq 1}}{}^{\leq 1} ((V_{i})_{\psi(i)^{-1}}) \to \prod_{i \in \mathbb{Z}_{\geq 1}}{}^{\leq 1} ((V_{i})_{(s\psi)(i)^{-1}}) \] 
induced by the obvious maps $s_i:(V_{i})_{\psi (i)^{-1}} \to (V_{i-1})_{\psi(i)^{-1}}= (V_{i-1})_{(s\psi)(i-1)^{-1}}$. Luckily, there is also the map \[\id:\prod_{i \in \mathbb{Z}_{\geq 1}}{}^{\leq 1} ((V_{i})_{\psi(i)^{-1}}) \to \prod_{i \in \mathbb{Z}_{\geq 1}}{}^{\leq 1} ((V_{i})_{(s\psi)(i)^{-1}})\] induced by $\id_i: (V_{i})_{\psi(i)^{-1}} \to (V_{i})_{(s\psi)(i)^{-1}}$ since $s\psi \geq \psi$. 
The morphisms 
\[(\id-s)_{j}:\prod_{i \in \mathbb{Z}_{\geq 1}}{}^{\leq 1} ((V_{i})_{\psi(i)^{-1}})\to (V_{j})_{(2s\psi)(j)^{-1}}
\]
defined by 
\[(\alpha_1, \alpha_2, \dots) \mapsto \alpha_{j} -\alpha_{j+1}.
\]
are non-expanding because of the inequalities
\begin{equation}
\begin{split}
||\alpha_{j} -\alpha_{j+1} ||_{(V_{j})_{(2s\psi)(j)^{-1}}} & = 2^{-1}\psi(j+1)^{-1}||\alpha_{j} -\alpha_{j+1} ||_{V_j}  \leq 2^{-1}\psi(j+1)^{-1} ||\alpha_{j+1} || + 2^{-1}\psi(j+1)^{-1} || \alpha_{j}||  \\ & \leq 2^{-1}\psi(j+1)^{-1} ||\alpha_{j+1} || + 2^{-1}\psi(j)^{-1} || \alpha_{j}||  \leq \underset{i \in \mathbb{Z}_{\geq 1}} \sup \psi(i)^{-1} ||\alpha_{i}||_{V_{i}}.
\end{split}
\end{equation}
Therefore, we can rewrite (\ref{eqn:themap}) as 
\[``\underset{\psi \in \Psi}\colim" \left(\prod_{i \in \mathbb{Z}_{\geq 1}}{}^{\leq 1} ((V_{i})_{\psi(i)^{-1}})\longrightarrow  \prod_{i \in \mathbb{Z}_{\geq 1}}{}^{\leq 1} ((V_{i})_{(2s\psi)(i)^{-1}})\right).
\]
Because the functor $``\underset{\psi \in \Psi}\colim"$ is exact, we are done. As in the proof of Lemma \ref{lem:ProdFre} we can replace $\Psi$ with the $\aleph_1$-filtered category $\Upsilon$ and conclude that  \[``\underset{\psi \in \Upsilon}\colim" \ker[\prod_{i \in \mathbb{Z}_{\geq 1}}{}^{\leq 1} ((V_{i})_{\psi(i)^{-1}}) \stackrel{\id-s}\longrightarrow \prod_{i \in \mathbb{Z}_{\geq 1}}{}^{\leq 1} ((V_{i})_{2^{-1}\psi(i+1)^{-1}})]\longrightarrow \lim_{i\in \mathbb{Z}_{\geq 1}}V_{i} 
\]
is an isomorphism in $\ttInd(\ttBan_{R})$. Therefore, $\underset{i\in \mathbb{Z}_{\geq 1}}\lim V_{i} $ is metrizable. 
\endproof
\begin{rem}\label{rem:Fine} It is completely fine to take some or all of the maps $V_{i+1}\to V_i$ to be the identity. In particular, taking them all to be the identity we see that any Banach module is metrizable. 
\end{rem}
\begin{lem}\label{lem:Wprod}For each $k\in K$ suppose we are given an inductive system $I_{k}\to \ttBan_{R}$ given by the system of Banach modules $W^{(k)}_{i}$. Let $W^{(k)}=``\underset{i\in I_k}\colim"W^{(k)}_{i}$. Assume that for $i_1 <i_2$ the morphisms $W^{(k)}_{i_1}\to W^{(k)}_{i_2}$ are non-expanding. Let $\Phi$ be the poset whose objects are pairs $(\phi_1,\phi_2)$ where $\phi_1:K\to \underset{k\in K}\coprod I_{k}$ such that $\phi_1(k)\in I_k$ for all $k$ and $\phi_2:K\to \mathbb{N}_{\geq 1}$.
This has a partial order defined by $(\phi_1, \phi_2)\leq (\phi'_1, \phi'_2)$ if and only if $\phi_1(k) \leq \phi'_1(k)$ and  $\phi_2(k) \leq \phi'_2(k)$ for all $k\in K$. Then the natural morphism 
\[``\underset{(\phi_1,\phi_2)\in \Phi}\colim" \prod_{k\in K}{{}^{\leq 1}} (W_{\phi_{1}(k)}^{(k)})_{\phi_{2}(k)^{-1}}\to \prod_{k\in K}W^{(k)}
\]
is an isomorphism. To explain the structure maps in the formal filtered colimit, for $(\phi_1, \phi_2)\leq (\phi'_1, \phi'_2)$ the morphism \[ \prod_{k\in K}{{}^{\leq 1}} (W_{\phi_{1}(k)}^{(k)})_{\phi_{2}(k)^{-1}}\to  \prod_{k\in K}{{}^{\leq 1}} (W_{\phi'_{1}(k)}^{(k)})_{\phi'_{2}(k)^{-1}}\] is the non-expanding product over $k\in K$ of the obvious morphisms \[(W_{\phi_{1}(k)}^{(k)})_{\phi_{2}(k)^{-1}}\to (W_{\phi'_{1}(k)}^{(k)})_{\phi'_{2}(k)^{-1}}.\] If each $I_k$ is $\aleph_1$-filtered then if $K$ is countable then $\underset{k\in K}\prod W^{(k)}$ is metrizable.
\end{lem}
\proof 
 It is enough to show that the morphisms \[\Hom(M, ``\underset{(\phi_1,\phi_2)\in \Phi}\colim" \prod_{k\in K}{{}^{\leq 1}} (W_{\phi_{1}(k)}^{(k)})_{\phi_{2}(k)^{-1}}) \to \Hom(M,  \prod_{k\in K}``\underset{i\in I_k}\colim"W^{(k)}_{i}) \] are  isomorphisms of abelian groups for any $M \in \ttBan_{R}$. We have 
\begin{equation}
\begin{split}\Hom(M, ``\underset{(\phi_1,\phi_2)\in \Phi}\colim" \prod_{k\in K}{{}^{\leq 1}} (W_{\phi_{1}(k)}^{(k)})_{\phi_{2}(k)^{-1}}) 
& =    \underset{(\phi_1,\phi_2)\in \Phi}\colim \Hom(M,\underset{k\in K}\prod{}^{\leq 1} ( W^{(k)}_{\phi_1(k)})_{\phi_2(k)^{-1}}  ) \\
& =    \underset{(\phi_1,\phi_2)\in \Phi}\colim \underset{j \in \mathbb{Z}_{>0}}\colim\Hom^{\leq j}(M,\underset{k\in K}\prod{}^{\leq 1} ( W^{(k)}_{\phi_1(k)})_{\phi_2(k)^{-1}}  ) \\
& =    \underset{(\phi_1,\phi_2)\in \Phi}\colim \underset{j \in \mathbb{Z}_{>0}}\colim\Hom^{\leq 1}(M,\underset{k\in K}\prod{}^{\leq 1} ( W^{(k)}_{\phi_1(k)})_{(j\phi_2(k))^{-1}}  ) \\
& =    \underset{(\phi_1,\phi_2)\in \Phi}\colim \Hom^{\leq 1}(M,\underset{k\in K}\prod{}^{\leq 1} ( W^{(k)}_{\phi_1(k)})_{\phi_2(k)^{-1}}  ) \\
& =    \underset{(\phi_1,\phi_2)\in \Phi}\colim\prod_{k\in K}  \Hom^{\leq 1}(M,( W^{(k)}_{\phi_1(k)})_{\phi_2(k)^{-1}}  ) \\
& =   \prod_{k\in K}  \underset{i\in I_k}\colim \ \ \underset{j \in \mathbb{Z}_{>0}} \colim \Hom^{\leq 1}(M,( W^{(k)}_{i})_{j^{-1}}  ) \\
& =   \prod_{k\in K}  \underset{i\in I_k}\colim \ \ \underset{j \in \mathbb{Z}_{>0}} \colim \Hom^{\leq j}(M, W^{(k)}_{i}  ) \\
& =   \prod_{k\in K}  \underset{i\in I_k}\colim\Hom(M, W^{(k)}_{i}  ) \\
& =   \prod_{k\in K}  \Hom(M, ``\underset{i\in I_k}\colim"W^{(k)}_{i}  ) \\
& \cong     \Hom(M,   \prod_{k\in K}``\underset{i\in I_k}\colim"W^{(k)}_{i}).
\end{split}
\end{equation}
As in the proof of Lemma \ref{lem:towerRep} have used in the isomorphism 
\[ \underset{(\phi_1,\phi_2)\in \Phi}\colim\prod_{k\in K}  \Hom^{\leq 1}(M,( W^{(k)}_{\phi_1(k)})_{\phi_2(k)^{-1}}  ) \cong \prod_{k\in K}  \underset{i\in I_k}\colim \ \ \underset{j \in \mathbb{Z}_{>0}} \colim \Hom^{\leq 1}(M,( W^{(k)}_{i})_{j^{-1}}  )
\]
that in the category of sets, products and filtered colimits distribute (not commmute!) as explained in \cite{ALR}. Let us now assume that each $I_k$ is $\aleph_1$-filtered and $K$ is countable (so we can assume that $K=\mathbb{N}$). Let $\Lambda$ be the set whose objects are pairs $\lambda=(\phi_1,\phi_2)$ where $\phi_1:K\to \underset{k\in K}\coprod I_{k}$ such that $\phi_1(k)\in I_k$ for all $k$ and $\phi_2:K\to \mathbb{N}_{\geq 1}$.
This has a partial order defined by $(\phi_1, \phi_2)<(\phi'_1, \phi'_2)$ when $\phi_1(k) < \phi'_1(k)$ and  $\phi_2(k) < \phi'_2(k)$ for all but a finite number of $k\in K$. Say that we are given a collection $\lambda^{(1)}=(\phi^{(1)}_1,\phi^{(1)}_2), \lambda^{(2)}=(\phi^{(2)}_1,\phi^{(2)}_2), \dots \in \Lambda$. Define $\beta=(\beta_1, \beta_2):K \to ( \underset{k\in K}\coprod I_{k}) \times \mathbb{N}_{\geq 1}$ by choosing for each $k$ an element $\beta_1(k) \in I_k$ such that $\beta_1(k)>\phi_{1}^{(m)} (k)$ for all $m$ and $\beta_2(k)=1 +\underset{l\leq k}\sum\phi^{(l)}_{2}(k)\in \mathbb{N}_{\geq 1}$. Then for any fixed $m$, $\beta_2(k) > \lambda_2^{(m)}(k)$ for all $k\geq m$ and $\beta_1(k) > \lambda_1^{(m)}(k)$ for all $m$. Therefore $\beta(k) > \lambda^{(m)}(k)$ for all $k\geq m$ and so $\beta >  \lambda^{(m)}$ for all $m$.
\endproof
By comparing Lemma \ref{lem:towerRep} and Lemma \ref{lem:Wprod} we get:
\begin{cor}\label{cor:compare}Consider the set of functions $\phi_1:K \to \underset{k\in K}\coprod I_{k}$ such that $\phi_1(k)\in I_k$ for all $k$. It has a partial order defined by  $\phi_1 <\phi'_1$ when $\phi_1(k) < \phi'_1(k)$ for all but a finite number of $k\in K$. Denote this poset by $\Phi_1$. The natural morphism
\[\underset{\phi_1\in \Phi_1}\colim \underset{k}\prod W_{\phi_{1}(i)}^{(k)} \to \underset{k}\prod  ``\underset{i\in I_k}\colim"W^{(k)}_{i}
\]
is an isomorphism.
\end{cor}
\begin{lem}\label{lem:prodtens} Let $K$ be a countable set. For each $k\in K$ suppose we are given an inductive system $I_{k}\to \ttBan_{R}$ given by the system of Banach modules $W^{(k)}_{i}$. Let $W^{(k)}= ``\underset{i \in I_k}\colim" W^{(k)}_{i}\in \ttInd(\ttBan_R)$ for each $k$. 
Then for any $U\in\ttBan_{R}$ the natural morphism
\[U\wotimes_{R} \left(\prod_{k\in K} W^{(k)}\right) \to \prod_{k\in K} \left(U\wotimes_{R} W^{(k)}\right)
\]
is an isomorphism. Suppose now that $K$ is a category with a countable set of objects and morphisms. If $U$ is flat over $R$ and $K\to \ttInd(\ttBan_R)$ is any functor written as $k \mapsto W^{(k)}$ then the natural morphism 
\[U\wotimes_{R} \left(\lim_{k\in K} W^{(k)}\right) \to \lim_{k\in K} \left(U\wotimes_{R} W^{(k)}\right)
\]
is an isomorphism.
\end{lem}
\proof 
Let $P=\underset{s\in S}\coprod^{\leq 1}R_{r_s}$ with $r_s>0$. Notice that 
\[P \wotimes_{R} \left(\prod_{k\in K} W^{(k)}\right) \cong P\wotimes_{R} ``\underset{(\phi_1,\phi_2)\in \Phi}\colim" \prod_{k\in K}{{}^{\leq 1}} (W_{\phi_{1}(k)}^{(k)})_{\phi_{2}(k)^{-1}}\cong ``\underset{(\phi_1,\phi_2)\in \Phi}\colim"\underset{s\in S}\coprod^{\leq 1}\prod_{k\in K}{{}^{\leq 1}} (W_{\phi_{1}(k)}^{(k)})_{r_s \phi_{2}(k)^{-1}}
\]
while we can rewrite $\prod_{k\in K} \left(P \wotimes_{R} W^{(k)}\right)$ as
\[\prod_{k\in K}\left( P\wotimes_{R}``\underset{i\in I_k}\colim"W^{(k)}_{i} \right) \cong \prod_{k\in K}``\underset{i\in I_k}\colim" \underset{s\in S}\coprod^{\leq 1} (W^{(k)}_{i})_{r_s} = ``\underset{(\phi_1,\phi_2)\in \Phi}\colim" \prod_{k\in K}{{}^{\leq 1}} \underset{s\in S}\coprod^{\leq 1}(W_{\phi_{1}(k)}^{(k)})_{r_{s}\phi_{2}(k)^{-1}}.
\]
Let $f_{(\phi_1,\phi_2)}^{(\phi'_1,\phi'_2)}$ denote the morphisms 
\[\underset{s\in S}\coprod^{\leq 1}\prod_{k\in K}{{}^{\leq 1}} (W_{\phi_{1}(k)}^{(k)})_{r_s \phi_{2}(k)^{-1}}\longrightarrow \underset{s\in S}\coprod^{\leq 1}\prod_{k\in K}{{}^{\leq 1}} (W_{\phi'_{1}(k)}^{(k)})_{r_s \phi'_{2}(k)^{-1}}
\]
for $(\phi_1, \phi_2)\leq (\phi'_1, \phi'_2)$ and similarly let 
$g_{(\phi_1,\phi_2)}^{(\phi'_1,\phi'_2)}$ denote the morphisms 
\[\prod_{k\in K}{{}^{\leq 1}} \underset{s\in S}\coprod^{\leq 1}(W_{\phi_{1}(k)}^{(k)})_{r_{s}\phi_{2}(k)^{-1}}\longrightarrow \prod_{k\in K}{{}^{\leq 1}} \underset{s\in S}\coprod^{\leq 1}(W_{\phi'_{1}(k)}^{(k)})_{r_{s}\phi'_{2}(k)^{-1}}.
\]
Now clearly for each $(\phi_1, \phi_2)$ we have that $\underset{s\in S}\coprod^{\leq 1}\underset{k\in K}\prod{{}^{\leq 1}} (W_{\phi_{1}(k)}^{(k)})_{r_s \phi_{2}(k)^{-1}}$ is a Banach submodule of $\underset{k\in K}\prod{{}^{\leq 1}} \underset{s\in S}\coprod^{\leq 1}(W_{\phi_{1}(k)}^{(k)})_{r_{s}\phi_{2}(k)^{-1}}$, denote the bounded inclusion by \[\iota_{(\phi_1, \phi_2)}:\underset{s\in S}\coprod^{\leq 1}\underset{k\in K}\prod{{}^{\leq 1}} (W_{\phi_{1}(k)}^{(k)})_{r_s \phi_{2}(k)^{-1}}\longrightarrow \underset{k\in K}\prod{{}^{\leq 1}} \underset{s\in S}\coprod^{\leq 1}(W_{\phi_{1}(k)}^{(k)})_{r_{s}\phi_{2}(k)^{-1}}.\] Notice that \begin{equation}
\label{eqn:iotag}g_{(\phi_1,\phi_2)}^{(\phi'_1,\phi'_2)}\circ \iota_{(\phi_1, \phi_2)}= \iota_{(\phi'_1, \phi'_2)} \circ f_{(\phi_1,\phi_2)}^{(\phi'_1,\phi'_2)}.\end{equation} We now want maps in the other direction but this will not work without increasing $(\phi_1,\phi_2)$. Suppose that we are given an element 
$(w_{k,s})_{k \in K, s\in S} \in \underset{k\in K}\prod{{}^{\leq 1}} \underset{s\in S}\coprod^{\leq 1}(W_{\phi_{1}(k)}^{(k)})_{r_{s}\phi_{2}(k)^{-1}}$. By definition this means that \[\underset{s\in S}\sum ||w_{k,s}||^{(k)}_{\phi_1(k)}r_{s}\phi_{2}(k)^{-1} < \infty\] and \[\underset{k\in K}\sup \underset{s\in S}\sum ||w_{k,s}||^{(k)}_{\phi_1(k)}r_{s}\phi_{2}(k)^{-1} < \infty.\]
For each $k$ this implies that $\{ s\in S| w_{k,s}\neq 0 \}$ is countable. Let $S_c$ be the subset of $S$ defined by $S_c=S-\{s\in S| w_{k,s}=0 \ \  \text{for all}\ \  k\}$. Notice that $S_c$ is countable since it is a countable union of countable subsets: $S_{c}=\underset{k \in K}\bigcup(S-\{ s\in S| w_{k,s}=0 \})$. Because of this countability, we can choose a collection of positive real numbers $p_s$ for $s\in S_c$ such that $p=\underset{s\in S_c}\sum p_s$ is finite. Choose $\phi'_2$ so that $\phi'_2(k) = 2^{k} \phi_2(k)$ for all $k$. Then for any $s\in S_c$ there exists a $k_s \in K$ such that we have 
\[\underset{k\in K}\sup  ||w_{k,s}||^{(k)}_{\phi_1(k)}r_{s}\phi'_{2}(k)^{-1} \leq p_s + ||w_{k_s,s}||^{(k_s)}_{\phi_1(k_s)}r_{s}\phi'_{2}(k_s)^{-1}
\]
Now 
\begin{equation}
\begin{split}
\underset{s\in S}\sum\underset{k\in K}\sup  ||w_{k,s}||^{(k)}_{\phi_1(k)}r_{s}\phi'_{2}(k)^{-1} & \leq p+ \underset{s\in S} \sum ||w_{k_s,s}||^{(k_s)}_{\phi_1(k_s)}r_{s}\phi'_{2}(k_s)^{-1} \\
& \leq  p+ 
\underset{k\in K}\sum \underset{s\in S} \sum ||w_{k,s}||^{(k)}_{\phi_1(k)}r_{s}\phi'_{2}(k)^{-1} \\
& \leq  p+ 
\underset{k\in K}\sum 2^{-k} \underset{s\in S} \sum ||w_{k,s}||^{(k)}_{\phi_1(k)}r_{s}\phi_{2}(k)^{-1} \\
& \leq  p+ 
\left(\underset{k\in K}\sum 2^{-k}\right) \left(\underset{k\in K}\sup  \underset{s\in S} \sum ||w_{k,s}||^{(k)}_{\phi_1(k)}r_{s}\phi_{2}(k)^{-1}\right)
\end{split}
\end{equation}
and therefore $\underset{s\in S}\sum\underset{k\in K}\sup  ||w_{k,s}||^{(k)}_{\phi_1(k)}r_{s}\phi'_{2}(k)^{-1}$ is finite. If $\phi'_2(k) = 2^{k} \phi_2(k)$ for all $k,$ we get bounded morphisms 
\[\pi^{(\phi_1, \phi'_2)}_{(\phi_1, \phi_2)}:\underset{k\in K}\prod{{}^{\leq 1}} \underset{s\in S}\coprod^{\leq 1}(W_{\phi_{1}(k)}^{(k)})_{r_{s}\phi_{2}(k)^{-1}} \longrightarrow \underset{s\in S}\coprod^{\leq 1}\underset{k\in K}\prod{{}^{\leq 1}} (W_{\phi_{1}(k)}^{(k)})_{r_s \phi'_{2}(k)^{-1}}.
\]
Clearly we also have \[\pi^{(\phi_1, \phi'_2)}_{(\phi_1, \phi_2)} \circ  g_{(\alpha_1,\alpha_2)}^{(\phi_1,\phi_2)} =f_{(\alpha_1,\alpha'_2)}^{(\phi_1,\phi'_2)}\circ  \pi^{(\alpha_1, \alpha'_2)}_{(\alpha_1, \alpha_2)}.\] Finally, notice also that \[\iota_{(\phi_1, \phi'_2)} \circ  \pi^{(\phi_1, \phi'_2)}_{(\phi_1, \phi_2)} =g^{(\phi_1, \phi'_2)}_{(\phi_1, \phi_2)}\] and \[\pi^{(\phi_1, \phi'_2)}_{(\phi_1, \phi_2)} \circ \iota_{(\phi_1, \phi_2)}=f^{(\phi_1, \phi'_2)}_{(\phi_1, \phi_2)} .\] These three identities combined with Equation (\ref{eqn:iotag}) imply that 
\begin{equation}\label{eqn:ClimCcolim} ``\underset{(\phi_1,\phi_2)\in \Phi}\colim"\underset{s\in S}\coprod^{\leq 1}\prod_{k\in K}{{}^{\leq 1}} (W_{\phi_{1}(k)}^{(k)})_{r_s \phi_{2}(k)^{-1}} \cong ``\underset{(\phi_1,\phi_2)\in \Phi}\colim" \prod_{k\in K}{{}^{\leq 1}} \underset{s\in S}\coprod^{\leq 1}(W_{\phi_{1}(k)}^{(k)})_{r_{s}\phi_{2}(k)^{-1}}
\end{equation}
and therefore 
\[P\wotimes_{R} \left(\underset{k\in K}\prod W^{(k)}\right) \cong \underset{k\in K}\prod \left(P\wotimes_{R} W^{(k)}\right).\] Now let $U\in \ttBan_{R}$ be arbitrary. Similarly to Lemma A.39 of \cite{BK} we can find a projective resolution \[K\to P \to U \to 0\] where all morphisms are strict epimorphisms, $K=\underset{t\in T}\coprod^{\leq 1}R_{r_t}$ and $P=\underset{s\in S}\coprod^{\leq 1}R_{r_s}$. The fact that products are right exact immediately implies that  $U\wotimes_{R} \left(\underset{k\in K}\prod W^{(k)}\right) \cong \underset{k\in K}\prod \left(U\wotimes_{R} W^{(k)}\right)$. The second statement of the Lemma follows from writing the limit in terms of countable products and kernels. 
\endproof
\begin{lem}\label{lem:tensprod} Let $K$ be a countable set, For each $k\in K$ suppose we are given an inductive system $I_{k}\to \ttBan_{R}$ given by the system of Banach modules $W^{(k)}_{i}$. Let $W^{(k)}= ``\underset{i \in I_k}\colim" W^{(k)}_{i}$. 
Let $V\in \ttInd(\ttBan_{R})$ be metrizable. Then the natural morphism
\[V\wotimes_{R} \left(\underset{k\in K}\prod W^{(k)}\right) \to \underset{k\in K}\prod \left(V\wotimes_{R} W^{(k)}\right)
\]
is an isomorphism. Now let $K$ be a category with a countable set of objects and morphisms. If $V$ is metrizable and flat over $R$ (or metrizable and nuclear) and $K\to \ttInd(\ttBan_R)$ is any functor then the natural morphism 
\[V\wotimes_{R} \left(\underset{k\in K}\lim W^{(k)}\right) \longrightarrow \underset{k\in K}\lim \left(V\wotimes_{R} W^{(k)}\right)
\]
is an isomorphism.
\end{lem}
\proof Let $V=``\underset{j \in J}\colim" V_{j}$ where $J$ is $\aleph_1$-filtered. Using Lemma \ref{lem:limcolim} we have 
\[\underset{k\in K}\prod \left(V\wotimes_{R} W^{(k)}\right)\cong\prod_{k\in K} \underset{j\in J }\colim \left(V_{j} \wotimes_{R} W^{(k)}\right) \cong  \underset{j\in J }\colim \underset{k\in K}\prod \left(V_{j} \wotimes_{R} W^{(k)}\right)
\]
because colimits over an $\aleph_{1}$-filtered category commute with countable products by Lemma \ref{lem:limcolim}.
Also \begin{equation}
V\wotimes_{R} \left(\underset{k\in K}\prod W^{(k)}\right)  \cong \underset{j \in J}\colim \left( V_j\wotimes_{R} \left(\underset{k\in K}\prod W^{(k)}\right) \right) \cong  \underset{j\in J }\colim \prod_{k\in K}\left(V_{j} \wotimes_{R} W^{(k)}\right)
\end{equation}
by Lemma \ref{lem:prodtens}. The second statement of the Lemma follows from writing the limit in terms of countable products and kernels.
\endproof
\begin{lem}\label{lem:surprise}The converse to the first part of Lemma \ref{lem:tensprod} holds in the sense that we can conclude that an object $V \in \ttInd(\ttBan_{R})$ is metrizable if and only the functor $V \wotimes_{R} (-)$ commutes with countable products.
\end{lem}
\proof 
We will prove that if for some $V\in \ttInd(\ttBan_{R})$, that the natural morphism \[V \wotimes_{R} (\underset{\mathbb{Z}}\prod R) \longrightarrow \underset{\mathbb{Z}}\prod V\] is an isomorphism, then $V$ is metrizable. Consider the category $J$ of all objects of $\ttBan_{R}$ mapping to $V$.  Then of course $V \cong \underset{j\in J}\colim V_{j}$. The above isomorphism combined with Lemma \ref{lem:prodtens} which tells us \[V_{j} \wotimes_{R} (\underset{\mathbb{Z}}\prod R ) \cong \underset{\mathbb{Z}}\prod V_{j}\] immediately implies that the natural morphism 
\[\underset{j\in J}\colim (\underset{\mathbb{Z}}\prod V_{j}) \longrightarrow \underset{\mathbb{Z}}\prod \underset{j\in J}\colim V_{j} 
\]
is an isomorphism. Suppose we are given a chain $V_{j_1}\to V_{j_2}\to V_{j_3}\to V_{j_4}\to \cdots$ in $J$. We can lift the natural morphism $\underset{k \in \mathbb{Z}}\prod  V_{j_k}\to \underset{k \in \mathbb{Z}}\prod \underset{j\in J}\colim V_{j} $ to a morphism  $\underset{k \in \mathbb{Z}}\prod  V_{j_k}\longrightarrow \underset{j\in J}\colim  \underset{k \in \mathbb{Z}}\prod V_{j}$. Therefore there exists some $j\in J$ such that all morphisms $V_{j_k} \to V$ factor through some $V_{j} \to V$. Therefore $J$ is $\aleph_{1}$-filtered and so $V$ is metrizable.
\endproof
\begin{rem}This result is surprising since the analogous result in the purely algebraic case is not true. In fact, in the algebraic case the tensor product of a module will commute with all products of other modules if and only if the first module is finitely presented \cite{Lenzing}. However, there is no contradiction here because if we take a ring, endow it with the discrete Banach structure, then we can consider the category of discrete modules over the ring but this category is not closed under the operation $\wotimes_{R}$. 

\end{rem}
\begin{lem}Say we fix $A \in \ttComm(\ttInd(\ttBan_R))$. Let $K$ be a countable set, For each $k\in K$ suppose we are given an inductive system $I_{k}\to \ttBan_{R}$ given by the system of Banach modules $W^{(k)}_{i}$. Suppose we are given objects  $W^{(k)} \in \ttMod(A)$ with underlying object $ ``\underset{i \in I_k}\colim" W^{(k)}_{i} \in \ttInd(\ttBan_R)$. 
Let $V\in \ttMod(A)$. Suppose the objects underlying $A$ and $V$ in $\ttInd(\ttBan_R)$ are metrizable. Then the natural morphism
\[V\wotimes_{A} \left(\prod_{k\in K} W^{(k)}\right) \to \prod_{k\in K} \left(V\wotimes_{A} W^{(k)}\right)
\]
is an isomorphism. 
\end{lem}
\proof  This follows from the case of $A=R$ which was proven in Lemma \ref{lem:tensprod} and the description of $\wotimes_A$ as a coequalizer in Equation \ref{eqn:Atens}  together with the fact that  $\wotimes_R$ is right exact in each variable as discussed in subsection \ref{BRBM}.
\endproof
\begin{cor}\label{cor:product_sum_nuc} As a corollary of Remark \ref{rem:Fine} (or of Lemma \ref{lem:prodtens}) and Lemma \ref{lem:surprise}, we see that if we have a countable collection $V_i$ of nuclear objects of $\ttInd(\ttBan_{R})$, their product is nuclear. A coproduct of any collection of nuclear objects is also nuclear.
\end{cor}
\proof  Suppose that we have a countable collection of nuclear objects $V_i$ indexed by a countable set $I$.  For any Banach module $W$ the map
\[(\underset{i\in I}\prod V_i) \wotimes_R W^\vee \longrightarrow \uHom(W,\underset{i\in I}\prod V_i )
\]
breaks up as a product of maps $V_i \wotimes_R W^\vee \to \uHom(W, V_i )$. Write the coproduct of nuclear objects $V_i$ over a countable set $I$ as a filtered colimit of coproducts over finite subsets. The finite coproducts of $V_i$ are clearly nuclear. Then notice that both sides of the needed equation $V\wotimes_{R}W^{\vee}=\uHom(W, V)$ are filtered colimits of true equations.
\endproof
\section{Spaces of Functions}\label{Spaces}

In this section, we use the previous results to study rings of analytic functions and their modules on Stein spaces over Banach rings $R$. As usual, ``affine" spaces are considered as the opposite category of commutative, associative, unital ring objects over $\ttInd(\ttBan_{R})$. As these form a huge category, one often wants to do geometry with a more manageable class of objects. One can define define Stein algebras over $R$ as limits of a sequence $\cdots \to A_3 \to A_2\to A_1$ where the $A_i$ are quotients of Banach disk algebras by finitely generated closed ideals such that $A_i$ are flat over $R$, and the morphisms are ring homomorphisms over $R$ in $\ttInd(\ttBan_{R})$ which are injective, homotopy epimorphisms, nuclear, non-expanding, and dense. These requirements were chosen based on the properties of the natural maps on poly-disk or Tate algebras thought of as restrictions of functions from bigger radius to smaller. This definition is motivated by local models in complex analytic geometry together with properties of functions on open polydisks along with their usual structure as Fr\'{e}chet algebras. Similarly, dagger algebras over $R$ are defined as colimits of systems made up of the same type of morphisms $A_1 \to A_2 \to A_3 \to \cdots$. These limits and colimits take place in $\ttInd(\ttBan_{R})$. In this section we focus on limits of colimits of canonical maps of disk algebras, without quotienting by any ideals. It may be interesting to extend these results to more general contexts to develop a complete theory.

\begin{defn}
The analytic functions on $n$-dimensional affine space over $R$ are defined by  
\[\mathcal{O}(\mathbb{A}^{n}_{R}) = \underset{r\in \mathbb{Z}_{>0}}\lim R\{\frac{x_1}{r}, \dots, \frac{x_n}{r}\}.\] 
Similarly, given an $n$-tuple of positive real numbers $r=(r_1, \dots, r_n)$ the $n$-dimensional open disk with multi-radius $r$ is defined by 
\[\mathcal{O}(D^{n}_{<r, R})=\underset{\rho<r}\lim R\{\frac{x_1}{\rho_1}, \dots, \frac{x_n}{\rho_n}\}.
\]
\end{defn}
The ring 
$\mathcal{O}(\mathbb{A}^{n}_{R})$ is actually a bornological (in fact Fr\'{e}chet) ring over $R$ isomorphic to 
\[\{ \underset{I \in \mathbb{Z}^{n}_{\geq0}}\sum a_{I}x^{I} \in R[[x_1, \dots, x_n]] \ \ | \text{  for each}\ \  r \in  \mathbb{Z}_{>0}  , \underset{I \in \mathbb{Z}^{n}_{\geq0}}\sum |a_{I}|r^{I}
<\infty\}.
\] The bornology is induced by the family of semi-norms 
$||f||_{r} = \underset{I \in \mathbb{Z}^{n}_{\geq0}}\sum |a_{I}|r^{I}$ in the sense that a subset is bounded if it is simultaneously bounded for all the metrics induced by this collection of semi-norms.  It is easy to see that the limit of algebras gives the standard algebra of global analytic functions in both the $\mathbb{R}$ case and the $\mathbb{Q}_{p}$ case (with the standard non-archimedean adaptations of using the Tate algebras instead of the $\ell^1$ disk algebras).
The ring $\mathcal{O}(D^{n}_{<r, R})$ thought of as functions on the open disk of multi-radius $r$ is isomorphic to 
\[\{\underset{I \in \mathbb{Z}^{n}_{\geq0}}\sum a_{I}x^{I} \in R[[x_1, \dots, x_n]] \ \ | \ \  \text{ for each} \ \ \rho<r, \underset{I \in \mathbb{Z}^{n}_{\geq0}}\sum |a_{I}|\rho^{I}
<\infty\}.\] Here, the bornology is induced by the family of seminorms $||f||_{m} = \underset{I \in \mathbb{Z}^{n}_{\geq0}}\sum |a_{I}|r_{(m)}^{I}$ where $r_{(m)} = (r_1 - \frac{1}{m}, \dots, r_n-\frac{1}{m})$. We could even consider $\mathbb{A}^{n}_{R}$ as the open disk of multi-radius $(\infty, \dots, \infty)$.
\begin{obs}Recall the definition of the poset $\Upsilon=\underset{i \in I}\prod\mathbb{Z}_{>0}$ 
from Definition \ref{defn:PsiUpsilon}. We can give an explicit description of the algebras of functions on the affine line or on an open disk. For example, in the one dimensional cases Lemma \ref{cor:TowerLim} applied in the case $V_n=R\{\frac{x}{n}\} \subset R[[x]]$ we have 
\[\mathcal{O}(\mathbb{A}^{1}_{R})=``\underset{\psi \in \Upsilon}\colim"\{f=\sum_{i=0}^{\infty} a_i x^{i} \in R[[x]] \ \ | \ \ \underset{n\in \mathbb{N}}\sup \left(\psi(n)^{-1} \sum_{i=0}^{\infty} |a_i|n^{i}\right)<\infty \}
\]
and if we apply Lemma \ref{cor:TowerLim} in the case $V_n=R\{\frac{x}{r-n^{-1}}\}\subset R[[x]]$ we get
\[\mathcal{O}(D^{1}_{<r, R})=``\underset{\psi \in \Upsilon}\colim"\{f=\sum_{i=0}^{\infty} a_i x^{i} \in R[[x]] \ \ | \ \ \underset{n\in \mathbb{N}}\sup \left(\psi(n)^{-1} \sum_{i=0}^{\infty} |a_i|(r-n^{-1})^{i}\right)<\infty \}
.\]
These Stein rings are written here as formal $\aleph_1$-filtered colimits of Banach modules. But the Banach modules themselves indexed by $\Upsilon$ and appearing within this colimit are not Banach rings because they are not closed under the multiplication of $R[[x]]$. In fact the same description can be done for any type of categorical limit $\lim V_{i}$ of Banach modules \[\cdots \subset V_3 \subset V_2 \subset V_1\] as submodules of a given algebraic module where $V_n$ are the elements $f$ in $V_1$ with $||f||_{V_n}<\infty$:
\[\underset{i \in \mathbb{N}}\lim V_i = ``\underset{\psi \in \Upsilon}\colim" \{f \in V_1 \ \ | \ \ \underset{n\in \mathbb{N}}\sup \left(\psi(n)^{-1} ||f||_{V_n}\right) < \infty \}.
\]
\end{obs}

\begin{lem}For each $\rho<\tau$ the restriction map
\[R\{\frac{x_1}{\tau_1}, \dots, \frac{x_n}{\tau_n}\} \longrightarrow R\{\frac{x_1}{\rho_1}, \dots, \frac{x_n}{\rho_n}\}
\]
is nuclear.
\end{lem}
\proof  This map clearly decomposes into the completed tensor product of its factors $R\{\frac{x_i}{\tau_i}\} \to R\{\frac{x_i}{\rho_i}\}$. By the compatibility of nuclearity with the completed tensor product discussed in Lemma \ref{lem:NuclearProps}, it is enough to treat the one dimensional case and to show that for $\rho<\tau$ that the map $R\{\frac{x}{\tau}\}\to R\{\frac{x}{\rho}\}$ is nuclear. For any $\tau$ and $\rho$ the Banach module $R\{\frac{x}{\tau}\}^{\vee}\wotimes_{R}R\{\frac{x}{\rho}\}$ can be described as 
\[\left(\underset{j\in \mathbb{Z}_{\geq 0}}\prod{}^{\leq 1} R_{\tau^{-j}}\right)\wotimes_{R}\left(\underset{i\in \mathbb{Z}_{\geq 0}}\coprod{}^{\leq 1} R_{\rho^{i}}\right) \cong \underset{i\in \mathbb{Z}_{\geq 0}}\coprod{}^{\leq 1} \left(\underset{j\in \mathbb{Z}_{\geq 0}}\prod{}^{\leq 1} R_{\rho^{i}\tau^{-j}} \right)
\]
The right hand side consists of elements $(a_{ij})_{i,j\in \mathbb{Z}_{\geq 0}}$ such that first of all \[\sup_{j \in \mathbb{Z}_{\geq 0}}|a_{ij}|\rho^{i}\tau^{-j}<\infty\] for any $i\in \mathbb{Z}_{\geq 0}$ and that furthermore, \[\sum_{i \in \mathbb{Z}_{\geq 0}}\sup_{j \in \mathbb{Z}_{\geq 0}}|a_{ij}|\rho^{i}\tau^{-j}<\infty.\] In the case that $a_{ij}=\delta_{ij}$ which gives the restriction map, the first condition is vacuous because it just says that $(\frac{\rho}{\tau})^{i}<\infty$ for any $i\in \mathbb{Z}_{\geq 0}$. The second condition gives $\sum_{i \in \mathbb{Z}_{\geq 0}}(\frac{\rho}{\tau})^{i}<\infty$ which holds precisely when $\rho<\tau$.
\endproof
Let $\psi$ be a non-decreasing sequence $\mathbb{Z}_{\geq 0} \to \mathbb{Z}_{\geq 1}$. Define a Banach ring over $R$ by 
\[R\{\frac{x}{r}\}^{\psi}= \{\underset{j\in \mathbb{Z}_{\geq 0}}\sum a_j x^{j} \in R[[x]] \ \ | \ \ \underset{j\in \mathbb{Z}_{\geq 0}}\sum |a_j| r^{j} \psi(j) <\infty \}.
\]
equipped with the norm 
\[||\underset{j\in \mathbb{Z}_{\geq 0}}\sum a_j x^{j} ||=  \underset{j\in \mathbb{Z}_{\geq 0}}\sum |a_j| r^{j} \psi(j) .
\]
 
Consider the morphisms
\begin{equation}\label{eqn:PsiSeq} \underset{i\in \mathbb{Z}_{\geq 1}}\coprod^{\leq 1}R\{\frac{x}{r+i^{-1}}\}_{2\psi(i+1)}\stackrel{\id -s}\longrightarrow\underset{i\in \mathbb{Z}_{\geq 1}}\coprod^{\leq 1}R\{\frac{x}{r+i^{-1}}\}_{\psi(i)}  \stackrel{\sigma}\longrightarrow R\{\frac{x}{r}\}^{\psi}
\end{equation}
given by 
\[(\id -s)(f_1, f_2, f_3, \dots) = (f_1, f_2-f_1, f_3-f_2, \dots)
\]
and where the map $\sigma$ is defined by summation.
Here,  $s(f)_{i}= f_{i-1}$ for $i\geq 1$ and $s(f)_{0}=0$. The map $s$ is contracting because it comes from the contracting maps $R\{\frac{x}{r+(i-1)^{-1}}\}_{\psi(i)} \to R\{\frac{x}{r+i^{-1}}\}_{\psi(i)} $ and similarly the map $id$ is contracting because it comes from the contracting maps $R\{\frac{x}{r+i^{-1}}\}_{\psi(i+1)} \to R\{\frac{x}{r+i^{-1}}\}_{\psi(i)}$.
For $f \in R\{\frac{x}{r+i^{-1}}\}_{2\psi(i+1)}$ 
\[||(\text{id} - s)(f)|| = ||f||_{R\{\frac{x}{r+i^{-1}}\}_{\psi(i)}} + ||f||_{R\{\frac{x}{r+(i+1)^{-1}}\}_{\psi(i+1)} }\leq 2||f||_{R\{\frac{x}{r+i^{-1}}\}_{\psi(i+1)}}=||f||_{R\{\frac{x}{r+i^{-1}}\}_{2\psi(i+1)}}
\]

The map $\sigma$ is non-expanding as it is induced by the obvious non-expanding morphisms \[R\{\frac{x}{r+i^{-1}}\}_{\psi(i)} \to  R\{\frac{x}{r}\}^{\psi}.\]

Define $\delta:R\{\frac{x}{r}\}^{\psi}\to \underset{i\in \mathbb{Z}_{\geq 1}}\coprod^{\leq 1}R\{\frac{x}{r+i^{-1}}\}_{\psi(i)}$ by 
\[\delta(\underset{j}\sum a_j x^{j} )=(a_i x^{i} )_{i}.
\]
The fact that $\delta$ is bounded follows from the estimate:
\[\underset{i\in \mathbb{Z}_{\geq 1}}\sum |a_i|(r+i^{-1})^{i}\psi(i)\leq (\sup_{j\in \mathbb{Z}_{\geq 1}}(1+\frac{1}{jr})^{j}) \underset{i\in \mathbb{Z}_{\geq 1}}\sum |a_i| r^{i} \psi(i) =e^{(r^{-1})} \underset{i\in \mathbb{Z}_{\geq 1}}\sum |a_i| r^{i} \psi(i) .
\]
Then we have $\sigma \circ (id-s)=0$. Indeed the norm of $\sum_1^{N} (id-s)f$ is smaller than or equal to $e^{(r^{-1})}||\id - s|| ||f_{N}||,$ where $f_{N}$ is the component of $f$ in $R\{\frac{x}{r+N^{-1}}\}_{\psi(N+1)}$. Because the terms of $f$ are summable, $||f_{N}||\to 0$ as $N$ goes to infinity and so the norm of  $\sigma \circ (\id-s)f$ is zero for any $f$.
Also notice that $\sigma \circ \delta=\text{id}_{R\{\frac{x}{r}\}^{\psi}}$. We conclude that (\ref{eqn:PsiSeq}) splits on the right. By dualizing it and using Lemma \ref{lem:DualColimIsLim} we get the identification
\begin{equation}\label{eqn:PsiKer} (R\{\frac{x}{r}\}^{\psi})^{\vee} =\ker[ \underset{i\in \mathbb{Z}_{\geq 1}}\prod^{\leq 1}(R\{\frac{x}{r+i^{-1}}\}^{\vee})_{\psi(i)^{-1}} \longrightarrow \underset{i\in \mathbb{Z}_{\geq 1}}\prod^{\leq 1}(R\{\frac{x}{r+i^{-1}}\}^{\vee})_{2^{-1}\psi(i+1)^{-1}} ]
\end{equation}
where the morphism on the right is given by 
\[(f_1, f_2, \dots) \mapsto (f_{1} -f_{2}, f_2 - f_3, \dots).
\]
\begin{lem}\label{lem:PsySysNuc} For two sequences $\phi$, $\psi$ such that $\psi(i)>2^{i}\phi(i)$ for all $i$ the natural morphism 
\[R\{\frac{x}{r}\}^{\psi}\to R\{\frac{x}{r}\}^{\phi}
\]
is nuclear.
\end{lem}
\proof 
As a module, we can identify $R\{\frac{x}{r}\}^{\phi}=\underset{j\in \mathbb{Z}_{\geq 0}}\coprod^{\leq 1}R_{(r+j^{-1})^{j} \phi(j)}$ and so 
\[(R\{\frac{x}{r}\}^{\psi})^{\vee}=\underset{j\in \mathbb{Z}_{\geq 0}}\prod^{\leq 1}R_{(r+j^{-1})^{-j} \psi(j)^{-1}}.\]
We have 
\[R\{\frac{x}{r}\}^{\phi}\wotimes_{R}(R\{\frac{x}{r}\}^{\psi})^{\vee}=\underset{j\in \mathbb{Z}_{\geq 0}}\coprod^{\leq 1}\underset{l\in \mathbb{Z}_{\geq 0}}\prod^{\leq 1}R_{(r+l^{-1})^{-l} \psi(l)^{-1}(r+j^{-1})^{j} \phi(j)}.
\]
 An element of this space is just a collection $(a_{j,l})_{j,l}$ such that \[\underset{l\in \mathbb{Z}_{\geq 0}}\sup |a_{j,l}|(r+l^{-1})^{-l} \psi(l)^{-1}(r+j^{-1})^{j} \phi(j)<\infty\] for each $j$ and \[\underset{j\in \mathbb{Z}_{\geq 0}}\sum \underset{l\in \mathbb{Z}_{\geq 0}}\sup |a_{j,l}|(r+l^{-1})^{-l} \psi(l)^{-1}(r+j^{-1})^{j} \phi(j)<\infty.\] The morphism we care about is nuclear if and only if $a_{j,l}=\delta_{j,l}$ satisfies these conditions. The first condition is obvious and the second reduces to checking that $\underset{j\in \mathbb{Z}_{\geq 0}}\sum \psi(j)^{-1}\phi(j) $ is finite, but this finiteness follows from our assumptions.
\endproof 

\begin{defn} The dagger algebra \cite{Bam} of overconvergent functions on the polydisk of polyradius $(r_1, \dots, r_n)\in \mathbb{R}_{\geq 0}^{n}$ is the colimit of the monomorphic restrictions of the functions on closed polydisks in $\ttComm(\ttInd(\ttBan_{R}))$:
\[R\{\frac{x_1}{r_1}, \dots, \frac{x_n}{r_n}\} ^\dagger= ``\underset{\rho>r}\colim" R\{\frac{x_1}{\rho_1}, \dots, \frac{x_n}{\rho_n}\}.
\]
This also makes sense when $r=(0, \dots, 0)$. This bornological ring can be realized as the subring of elements $f= \underset{I \in \mathbb{Z}^{n}_{\geq0}}\sum a_{I}x^{I}$ of $R[[x_1, \dots, x_n]]$ such that for some $\rho>r$ we have that  $\underset{I \in \mathbb{Z}^{n}_{\geq0}}\sum |a_{I}|\rho^{I}
<\infty$. A subset is bounded precisely when it is bounded in one of the Banach rings $R\{\frac{x_1}{\rho_1}, \dots, \frac{x_n}{\rho_n}\}$.
\end{defn}
\begin{lem}There are canonical injective morphisms 
\[\left(R\{\frac{x_1}{r_1}, \dots, \frac{x_n}{r_n}\} ^\dagger\right)^{\vee} \longrightarrow \mathcal{O}(D^{n}_{<r^{-1}, R})
\]
\[  \mathcal{O}(D^{n}_{<r^{-1}, R})^{\vee} \longrightarrow  R\{\frac{x_1}{r_1}, \dots, \frac{x_n}{r_n}\} ^\dagger
\]
where $r^{-1}= (r^{-1}_{1}, \dots, r^{-1}_{n})$.

\end{lem}
\proof  For the second, we first define a bounded pairing between $ \mathcal{O}(D^{n}_{<r^{-1}, R})$  and $ R\{\frac{x_1}{r_1}, \dots, \frac{x_n}{r_n}\} ^\dagger$ which is $R$-linear and non-degenerate in each variable.
Consider the partially defined map 
\[f: R[[x_1, \dots, x_n]] \times R[[x_1, \dots, x_n]] \mathrel{-\,}\rightarrow R
\]
given by 
\[(\sum_{I \in \mathbb{Z}^{n}_{\geq 0}}a_{I}x^{I}, \sum_{J\in \mathbb{Z}^{n}_{\geq 0}}b_{J}x^{J})\mapsto \sum_{I \in \mathbb{Z}^{n}_{\geq 0}}a_{I}b_{I}.
\]
It suffices to show that it restricts to a well defined, bounded non-degenerate morphism on the product 
\[R\{\frac{x_1}{r_1}, \dots, \frac{x_n}{r_n}\} ^\dagger \times \mathcal{O}(D^{n}_{<r^{-1}, R}).
\]
The non-degeneracy is obvious. The fact that it is well defined follows from the estimate
\begin{equation}\label{eqn:CSinequality}|\sum_{I \in \mathbb{Z}^{n}_{\geq 0}}a_{I}b_{I}| \leq \sum_{I \in \mathbb{Z}^{n}_{\geq 0}}|a_{I}b_{I}|= \sum_{I \in \mathbb{Z}^{n}_{\geq 0}}|a_{I}\rho^{-I}||b_{I}\rho^{I}| \leq \left(\sum_{I \in \mathbb{Z}^{n}_{\geq 0}}|a_{I}\rho^{-I}|\right)\left( \sum_{J \in \mathbb{Z}^{n}_{\geq 0}}|b_{J}\rho^{J}|\right)
\end{equation}
assuming that $\underset{J\in \mathbb{Z}^{n}_{\geq 0}}\sum b_{J}x^{J} \in  R\{\frac{x_1}{\rho_1}, \dots, \frac{x_n}{\rho_n}\}$ and $\rho>r$. In order to show that it is bounded we need to take a bounded subset $B_1 \subset R\{\frac{x_1}{\rho_1}, \dots, \frac{x_n}{\rho_n}\}$ for $\rho>r$ and another bounded subset $B_2\subset \mathcal{O}(D^{n}_{<r^{-1}, R})$ and show that $f(B_1 \times B_2)$ is bounded. Let $ \mathcal{O}(D^{n}_{<r^{-1}, R})_{\rho^{-1}}$ denote the space $ \mathcal{O}(D^{n}_{<r^{-1}, R})$ equipped with the norm coming from $R\{\rho_1 x_1, \dots, \rho_n x_n\}$. The map $f:R\{\frac{x_1}{\rho_1}, \dots, \frac{x_n}{\rho_n}\}\times \mathcal{O}(D^{n}_{<r^{-1}, R})\to R$ factorizes as 
\[R\{\frac{x_1}{\rho_1}, \dots, \frac{x_n}{\rho_n}\}\times \mathcal{O}(D^{n}_{<r^{-1}, R}) \to R\{\frac{x_1}{\rho_1}, \dots, \frac{x_n}{\rho_n}\}\times \mathcal{O}(D^{n}_{<r^{-1}, R})_{\rho^{-1}} \to R.
\] 
The first map is clearly bounded and so $B_1\times B_2$ is still bounded in $R\{\frac{x_1}{\rho_1}, \dots, \frac{x_n}{\rho_n}\}\times \mathcal{O}(D^{n}_{<r^{-1}, R})_\rho$ and so lands inside  the product of disks $D_1 \times D_2$ where $D_i$ consists of elements of norm less than $d_i$. The estimate  (\ref{eqn:CSinequality}) again shows the image in $R$ is bounded.
\endproof

\begin{lem}
\[\left(R\{\frac{x_1}{r}\} ^\dagger\right)^{\vee} \cong \mathcal{O}(D^{1}_{<r^{-1}, R})
\]
\end{lem}
\proof 
We have 
\begin{equation}\begin{split}
\left(R\{\frac{x}{r}\} ^\dagger\right)^{\vee}  \cong \underset{\rho>r}\lim \left(R\{\frac{x}{\rho}\}^{\vee}\right) & = \underset{\rho>r}\lim \left((\underset{j\in \mathbb{Z}_{\geq 0}}\coprod^{\leq 1} (R_{\rho^j}))^{\vee}\right)\\ 
& = \underset{\rho>r}\lim \left((\underset{j\in \mathbb{Z}_{\geq 0}}\prod^{\leq 1} \left((R_{\rho^j}\right)^{\vee}))\right) \\
& \cong \underset{\rho>r}\lim \left(\underset{j\in \mathbb{Z}_{\geq 0}}\prod^{\leq 1} \left(R_{\rho^{-j}}\right)\right).
\end{split}
\end{equation}
On the other hand,  \[\mathcal{O}(D^{1}_{<r^{-1}, R})=\underset{\tau<r^{-1}}\lim R\{\frac{x}{\tau}\}\cong \underset{\rho>r}\lim R\{\rho x\}.\]
Now notice that $\underset{j\in \mathbb{Z}_{\geq 0}}\prod^{\leq 1} \left(R_{\rho^{-j}}\right)=\{\sum_{j\in \mathbb{Z}_{\geq 0}}a_jx^j | \sup_{j\in \mathbb{Z}_{\geq 0}}|a_j|\rho^{-j}<\infty\}$ and 
 \[R\{\rho x\}=\{\sum_{j\in \mathbb{Z}_{\geq 0}}a_jx^j | \sum_{j\in \mathbb{Z}_{\geq 0}}|a_j|\rho^{-j}<\infty\}\] and \[R\{\frac{ x}{\tau}\}= \{\sum_{j\in \mathbb{Z}_{\geq 0}}a_jx^j | \sum_{j\in \mathbb{Z}_{\geq 0}}|a_j|\tau^{j}<\infty\}\]
 and that the inclusion map for $\tau=\rho^{-1}$
 \[R\{\frac{ x}{\tau}\} \longrightarrow \underset{j\in \mathbb{Z}_{\geq 0}}\prod^{\leq 1} \left(R_{\rho^{-j}}\right) 
 \]
 or in other words
\[R\{\rho x\} \longrightarrow \underset{j\in \mathbb{Z}_{\geq 0}}\prod^{\leq 1} \left(R_{\rho^{-j}}\right) 
\]
is bounded. Also if $\eta$ satisfies $\eta<\rho^{-1}<r^{-1}$ we have a bounded inclusion $ \underset{j\in \mathbb{Z}_{\geq 0}}\prod^{\leq 1} \left(R_{\rho^{-j}}\right) \longrightarrow R\{\frac{x}{\eta}\}$ because of the inequality
\[\sum_{j\in \mathbb{Z}_{\geq 0}}|a_j|\eta^{j}= \sum_{j\in \mathbb{Z}_{\geq 0}}|a_j|\rho^{-j}(\eta \rho)^{j}\leq (\sup_{i\in \mathbb{Z}_{\geq 0}}|a_i|\rho^{-i})(\sum_{k\in \mathbb{Z}_{\geq 0}}(\eta\rho)^{k}).
\]
Together these inclusions give the desired isomorphisms.
\endproof
\begin{lem}
\[\left(R\{\frac{x_1}{r_1}, \dots, \frac{x_n}{r_n}\} ^\dagger\right)^{\vee} \cong \mathcal{O}(D^{n}_{<r^{-1}, R})
\]
\end{lem}
\proof 
Let $V_\rho= R_{\rho_1} \oplus \cdots \oplus R_{\rho_n}$.
We have 
\begin{equation}\begin{split}
\left(R\{\frac{x_1}{r_1}, \dots, \frac{x_n}{r_n}\} ^\dagger\right)^{\vee}  \cong \underset{\rho>r}\lim \left(R\{\frac{x_1}{\rho_1}, \dots, \frac{x_n}{\rho_n}\}^{\vee}\right) & = \underset{\rho>r}\lim \left((\underset{j\in \mathbb{Z}_{\geq 0}}\coprod^{\leq 1} (V_\rho^{\otimes j}/\Sigma_j))^{\vee}\right)\\ 
& = \underset{\rho>r}\lim \left(\underset{j\in \mathbb{Z}_{\geq 0}}\prod^{\leq 1} \left((V_\rho^{\otimes j}/\Sigma_j\right)^{\vee})\right)\\ 
& = \underset{\rho>r}\lim \left(\underset{j\in \mathbb{Z}_{\geq 0}}\prod^{\leq 1} \left(({V_{\rho^{-1}}^{\otimes j}})^{\Sigma_j}\right)\right)
\end{split}
\end{equation}
On the other hand,  \[\mathcal{O}(D^{n}_{<r^{-1}, R})=\underset{\tau<r^{-1}}\lim R\{\frac{x_1}{\tau_1}, \dots, \frac{x_n}{\tau_n}\}\cong \underset{\rho>r}\lim R\{\rho_1 x_1, \dots, \rho_n x_n\}= \underset{\rho>r}\lim\left( \underset{j\in \mathbb{Z}_{\geq 0}}\coprod^{\leq 1} (V_{\rho^{-1}}^{\otimes j}/\Sigma_j) \right)\]

Notice that for each $\rho>r$ we have injective bounded maps
\[f_{\rho}: R\{\rho_1 x_1, \dots, \rho_n x_n\} \longrightarrow R\{\frac{x_1}{\rho_1}, \dots, \frac{x_n}{\rho_n}\}^{\vee}
\]
sending $\underset{I}\sum a_{I}x^I$ to the map sending $\underset{I}\sum b_{I}x^I$ to $\underset{I}\sum a_{I}b_{I}$. The later is bounded by $\underset{J}\sum ||a_{J}||\rho^{-J}$
because 
\[||\underset{I}\sum a_{I} b_{I}||\leq \underset{I}\sum ||a_{I}|| ||b_{I}||= \underset{I}\sum ||a_{I}||\rho^{I} ||b_{I}||\rho^{-I}\leq (\underset{J}\sum ||a_{J}||\rho^{-J})(\underset{K}\sum ||b_{K}||\rho^{K}).
\]
This then shows that $f_\rho$ has norm less than or equal to one.
For any $\eta>\rho>r$ we have an injective bounded map
\[ g_{\rho,\eta}: R\{\frac{x_1}{\rho_1}, \dots, \frac{x_n}{\rho_n}\}^{\vee} \longrightarrow R\{\eta_1 x_1, \dots, \eta_n x_n\} 
\]
given by sending any $\alpha$ to $\underset{I}\sum \alpha(x^{I}) x^{I}$ which can be seen to be well defined and bounded by the estimate
\[||\underset{I}\sum \alpha(x^{I}) x^{I}||= \underset{I}\sum |\alpha(x^{I})|\eta^{-I}\leq \underset{I}\sum ||\alpha||||x^{I}||\eta^{-I}= \left(\underset{I}\sum (\rho/\eta)^{I}\right)||\alpha||.
\]
Notice that the composition $g_{\rho,\eta}\circ f_{\rho}$ is simply restriction from the disk of radius $\rho^{-1}$ to $\eta^{-1}$. The composition $f_{\eta}\circ g_{\rho,\eta}$ is the identity. Therefore these maps give maps of systems which give the required isomorphisms:
 \begin{equation}\label{eqn:PresOpenDisk}\mathcal{O}(D^{n}_{<r^{-1}, R})= \underset{\rho>r}\lim R\{\rho_1 x_1, \dots, \rho_n x_n\} \cong \underset{\rho>r}\lim \left(R\{\frac{x_1}{\rho_1}, \dots, \frac{x_n}{\rho_n}\}^{\vee}\right)\cong \left(R\{\frac{x_1}{r_1}, \dots, \frac{x_n}{r_n}\} ^\dagger\right)^{\vee}.
 \end{equation}
 \endproof
 We can now finish showing that open disks and affine spaces over $R$ are flat over $R$ in any dimension.
 \begin{cor}\label{cor:NuclearityOfFunctions}$\mathcal{O}(D^{n}_{<r^{-1}, R})$ is nuclear (and hence flat over $R$) for any $r=(r_1, \dots, r_n)$ with $r_i\geq 0$.
 \end{cor}
 \proof 
 We have by Equation \ref{eqn:PresOpenDisk} and the description of limits from Corollary \ref{cor:TowerLim} that $\mathcal{O}(D^{n}_{<r^{-1}, R})$ is isomorphic to  
 \tiny
 \[
``\underset{\psi \in \Psi}\colim"\ker[\underset{i}\prod^{\leq 1}(R\{\frac{x_1}{r_1^{-1}+i^{-1}},\dots, \frac{x_n}{r_n^{-1}+i^{-1}}\}^{\vee})_{\psi(i)^{-1}}\to \underset{i}\prod^{\leq 1}(R\{\frac{x_1}{r_1^{-1}+i^{-1}},\dots, \frac{x_n}{r_n^{-1}+i^{-1}}\}^{\vee})_{2^{-1}\psi(i+1)^{-1}}].
 \]
 \normalsize
 which is isomorphic to 
 $``\underset{\psi\in \Psi}\colim" (R\{r_1x_1, \dots, r_nx_n\}^{\psi})^{\vee}$ by Equation \ref{eqn:PsiKer}. Lemma \ref{lem:PsySysNuc} implies that we can find a final indexing set so that all the morphisms in the system are nuclear.  Therefore, by Lemma \ref{lem:SameSame}, $\mathcal{O}(D^{n}_{<r^{-1}, R})$ is nuclear and hence flat by Lemma \ref{lem:NucImpliesFlat}.
 \endproof
\begin{lem}\label{lem:TensorOfFunctions}
For any $\tau_1, \dots, \tau_n \in (0,\infty]$ and $\rho_1, \dots, \rho_m \in (0,\infty]$ , we have 
\[\mathcal{O}(D^{n}_{<\tau}) \wotimes_{R} \mathcal{O}(D^{m}_{<\rho}) \cong \mathcal{O}(D^{n+m}_{<(\tau,\rho)}).
\]
In particular $\mathcal{O}(\mathbb{A}^{n}_{R}) \wotimes_{R}\mathcal{O}(\mathbb{A}^{m}_{R}) \cong \mathcal{O}(\mathbb{A}^{n+m}_{R})$.
\end{lem}
\proof First notice that both $R\{\rho_1 x_1, \dots, \rho_n x_n\}$ and $\mathcal{O}(D^{n}_{<\tau})$ are metrizable by Corollary \ref{cor:TowerLim}. Also, $\mathcal{O}(D^{n}_{<\tau})$ is nuclear by Corollary \ref{cor:NuclearityOfFunctions} and hence flat by Lemma \ref{lem:NucImpliesFlat}. Now the second statement in Lemma \ref{lem:prodtens} shows that all the limits can be pulled outside.
Explicitly 
\begin{equation}
\begin{split}
\mathcal{O}(D^{n}_{<\tau}) \wotimes_{R} \mathcal{O}(D^{m}_{<\rho})  & \cong \mathcal{O}(D^{n}_{<\tau}) \wotimes_{R} \underset{s<\rho}\lim R\{\frac{y_1}{s_1}, \dots, \frac{y_m}{s_m}\} \cong \underset{s<\rho}\lim (\mathcal{O}(D^{n}_{<\tau}) \wotimes_{R}  R\{\frac{y_1}{s_1}, \dots, \frac{y_m}{s_m}\}) \\
& \cong \underset{s<\rho}\lim (( \underset{r<\tau}\lim R\{\frac{x_1}{r_1}, \dots, \frac{x_n}{r_n}\}) \wotimes_{R}  R\{\frac{y_1}{s_1}, \dots, \frac{y_m}{s_m}\})  \\ & \cong \underset{s<\rho}\lim \underset{r<\tau}\lim (R\{\frac{x_1}{r_1}, \dots, \frac{x_n}{r_n}\} \wotimes_{R}  R\{\frac{y_1}{s_1}, \dots, \frac{y_m}{s_m}\}) \\ & \cong  \underset{(s,r)<(\tau,\rho)}\lim R\{\frac{x_1}{r_1}, \dots, \frac{x_n}{r_n}, \frac{y_1}{s_1}, \dots, \frac{y_m}{s_m}\} \\ 
& \cong \mathcal{O}(D^{n+m}_{<(\tau,\rho)}).
\end{split}
\end{equation}
\endproof
\section{Topologies and Descent}\label{sec:Des}
Different considerations of abstract topologies and descent that we know of have appeared for example in works of Orlov \cite{O} and Kontsevich/Rosenberg \cite{KR}.
We consider descent in the infinity-category version of the homotopy monomorphism, flat, and other topologies in our project on derived analytic geometry \cite{BK2}. On the other hand, in this article we try to focus on non-derived categories of modules (i.e. homotopicaly discrete modules or complexes in degree $0$) and the underived pullback functors of restriction. As this would not work for arbitrary modules concentrated in degree zero, we need to specialize to quasi-coherent modules. The Grothendieck pre-topology that we work with is not quasi-compact, it has covers consisting of a countable collection of homotopy monomorphisms $\spec(A_i)\to \spec(A)$ such that given a morphism $f: M\to N$ in $\ttMod(A)$ with $M\wotimes_{A}A_{i} \to N\wotimes_{A}A_{i}$ an isomorphism for all $i$, then $f$ is an isomorphism. This property is called being conservative. It is expected to correspond to surjectivity of the cover for the topos-theoretic notion of points. For instance, it is known that the Huber points correspond with the topos-theoretic notion of points in the rigid-analytic context with the $G$-topology (see \cite{BK} where more explanation and citations are given). In most geometric settings \cite{BK} \cite{BK2} \cite{BBK} \cite{BBK2}, conservativity on the level of quasi-coherent modules agrees with surjectivity of the map $\underset{i}\cup \max(A_i) \to \max(A)$  on the maximal ideal spectrum, which then agrees with the pullback map from quasi-coherent modules to the descent category for quasi-coherent modules being fully faithful.

\subsection{Quasi-coherent modules}
\begin{defn}Let $A$ be an object of $\ttComm(\ttInd(\ttBan_R))$. Objects $M$ and $N$ of $\ttMod(A)$ are called transverse over $A$ if $M\wotimes^{\mathbb{L}}_{A}N \cong M\wotimes_{A}N$.
\end{defn}
\begin{defn}Let $A$ be an object of $\ttComm(\ttInd(\ttBan_R))$ flat over $R$. An object $M$ of $\ttMod(A)$ is called quasi-coherent if it is flat over $R$ and for all homotopy epimorphisms $A\to B$ where $B$ is metrizable that $M$ is transverse to $B$ over $A$. The full subcategory of quasi-coherent modules is denoted by $\ttMod^{RR}(A)$.
\end{defn}
\begin{defn}Let $A$ be an object of $\ttComm(\ttInd(\ttBan_R))$ which is metrizable. An object $M$ of $\ttMod_{F}(A)$ is called {\it quasi-coherent} if for all homotopy epimorphisms $A\to B$ in $\ttComm(\ttInd(\ttBan_R))$ where $B$ is metrizable that $M$ is transverse to $B$ over $A$. The full subcategory of quasi-coherent modules is denoted by $\ttMod^{RR}_{F}(A)$.
\end{defn}
Our notation ``RR" is in credit to Ramis and Ruget who introduced a similar notion in the context of complex analysis in \cite{RR}.
\begin{example}Let $A\in \ttComm(\ttInd(\ttBan_R))$ and $V\in \ttInd(\ttBan_R)$. Assume that both $A$ and $V$ are flat over $R$ and metrizable. Then $A\wotimes_{R}V\in \ttMod^{RR}_{F}(A)$.
\end{example}
\begin{lem}Say that $C$ is a category with countably many objects and morphisms. Then given any functor $F:C\to \ttMod^{RR}_{F}(A)$ the limit computed in $\ttMod(A)$ lives in $\ttMod^{RR}_{F}(A)$.
\end{lem}
\subsection{General Results on Descent}
\begin{lem}\label{lem:EF}Let $A\in \ttComm(\ttInd(\ttBan_R))$, and let  $\{E_i\}_{i\in I}$ be a projective system in $\ttMod(A)$ indexed by the countable poset $I$. Let $F$ be an object in $\ttMod(A)$. Suppose that the underlying objects of $A$ and $F$ in $\ttInd(\ttBan_R)$ are metrizable and  flat over $R$.  
Suppose in addition $F$ is transverse to $E_i$ over $A$ for each $i$ and the system $\{E_i\}_{i\in I}$ is $\underset{i\in I}\lim$-acyclic. Then 
 $\{F\wotimes_{A}E_i\}_{i\in I}$ is a $\underset{i\in I}\lim$-acyclic projective system, $F$ is transverse to $\underset{i\in I}\lim E_i$ over $A$ and the natural morphism \[F\wotimes_{A}(\underset{i\in I}\lim E_{i}) \to  \underset{i\in I}\lim (F\wotimes_{A} E_{i} )\] is an isomorphism. If instead of the condition that $F$ is flat over $R$ we have that both the $E_i$ and $\underset{i\in I}\lim E_{i}$ are flat over $R$ then the same conclusion holds.
\end{lem}
\proof  Recall that the Bar complex $\mathscr{L}^{\bullet}_{A}(F)$ is strictly quasi-isomorphic to $F$ and that $F\wotimes^{\mathbb{L}}_{A}(-)$ is computed by $\mathscr{L}^{\bullet}_{A}(F)\wotimes_{R}(-)$. Using the explicit form of this complex together with the fact that $F\wotimes_{R}(-)$ and $A\wotimes_{R}(-)$ commute with products by Lemma \ref{lem:tensprod}, we can see that it interacts well with the Roos complex of $\{E_i\}_{i\in I}$ in the sense that there is a strict quasi-isomorphism
\[\Tot(\mathscr{L}^{\bullet}_{A}(F)\wotimes_{R}\mathscr{R}^{\bullet}(\{E_i\}_{i\in I})) \cong \Tot( \mathscr{R}^{\bullet} (\{\mathscr{L}^{\bullet}_{A}(F)\wotimes_{R}E_i\}_{i\in I})).
\]
The left hand side computes $F\wotimes^{\mathbb{L}}_{A} (\mathbb{R}\underset{i\in I}\lim E_i)$ and the right hand side computes $\mathbb{R}\underset{i\in I}\lim (F\wotimes^{\mathbb{L}}_{A} E_i)$. Using that $F$ is transverse to $E_i$ over $A$ for each $i$ and the system $\{E_i\}_{i\in I}$ is $\underset{i\in I}\lim$-acyclic the above equation simplifies to a quasi-isomorphism 
\[F\wotimes^{\mathbb{L}}_{A} (\underset{i\in I}\lim E_i) \cong \mathbb{R}\underset{i\in I}\lim (F\wotimes_{A} E_i).
\] From which the rest of the claims follow immediately as one side is in non-negative degrees and the other is in non-positive degrees and so both are concentrated in degree zero.
\endproof

\begin{cor}\label{cor:AlimCons}Let $A\in \ttComm(\ttInd(\ttBan_R))$ presented by a system $A_i$ indexed by $i$ in a countable poset $I$ in $\ttComm(\ttInd(\ttBan_R))_{/A}$. Suppose that the  system $\{A_i\}_{i\in I}$ is $\underset{i\in I}\lim$-acyclic and that the objects of $\ttInd(\ttBan_R)$ underlying  $A_i$ are metrizable and flat over $R$ and transverse to one another over $A$ and in $\ttMod^{RR}_{F}(A)$. The natural functor $ \ttMod(A) \to \underset{i\in I}\lim \ttMod(A_i)$ induces a functor $ \ttMod^{RR}_{F}(A) \to \underset{i\in I}\lim \ttMod^{RR}_{F}(A_i)$. Then $A=\underset{i\in I}\lim A_i$ if and only if the collection of functors $(-) \wotimes_{A} A_i$ is conservative. When this holds the natural functor \[\underset{i\in I}\lim \ttMod^{RR}_{F}(A_i) \to \ttMod^{RR}_{F}(A)
\]
\[\{N_i\}_{i\in I}\mapsto \underset{i\in I}\lim N_i
\]
is essentially surjective.
\end{cor}
\proof Given an object $M \in  \ttMod^{RR}_{F}(A)$ and $A_i\to B$ is a homotopy epimorphism we have 
\[(M \wotimes_{A} A_i) \wotimes^{\mathbb{L}}_{A_i} B \cong (M \wotimes^{\mathbb{L}}_{A} A_i) \wotimes^{\mathbb{L}}_{A_i} B \cong  M \wotimes^{\mathbb{L}}_{A} (A_i \wotimes^{\mathbb{L}}_{A_i} B ) \cong  M \wotimes^{\mathbb{L}}_{A} B \cong M \wotimes_{A} B \cong  (M \wotimes_{A} A_i )\wotimes_{A_i} B
\] so $M \wotimes_{A} A_i \in \ttMod^{RR}(A_i)$ and each $M \wotimes_{A} A_i$ is metrizable. 

If $A\cong \underset{i\in I}\lim A_i$, give a morphism $f:M\to N$, we can rewrite $f$ using Lemma \ref{lem:EF} as $\underset{i\in I}\lim f_i$ where $f_i:M\wotimes_{A} A_i\to N\wotimes_{A} A_i$. Therefore, the collection is conservative. Conversely, if the collection is conservative let $\pi:A \to \underset{i\in I}\lim A_i$ be the canonical morphism. To show it is an isomorphism it is enough to know that it becomes so after applying the $A_j \wotimes_{A}(-)$. But after doing this we get using Lemma \ref{lem:EF}
\[A_j \to A_{j}\wotimes_{A}(\underset{i\in I}\lim A_i)\cong \underset{i\in I}\lim (A_{j}\wotimes_{A}A_i) \cong A_j.
\]
which is an isomorphism. The essential surjectivity holds because again using Lemma \ref{lem:EF} we have \[M=M\wotimes_{A}A\cong M\wotimes_{A}(\underset{i\in I}\lim A_i) \cong \underset{i\in I}\lim ( M\wotimes_{A} A_i)
\]
for any $M \in  \ttMod^{RR}_{F}(A)$.
\endproof
\begin{lem}\label{lem:epis}Let $A\in \ttComm(\ttInd(\ttBan_R))$ and say we are given a countable poset $A\to A_i$ of epimorphisms of $\ttComm(\ttInd(\ttBan_R))_{/A}$. The functor 
\[\underset{i\in I}\lim \ttMod(A_i) \to \ttMod(A)
\]
is fully faithful.
\end{lem}
\proof 
The natural ``pushforward" functors $\ttMod(A_i) \to \ttMod(A)$ are fully faithful. The limit of these functors is therefore fully faithful. 

\endproof
By combining Corollary \ref{cor:AlimCons} and Lemma \ref{lem:epis} we have 
\begin{thm}\label{thm:Big}Let $A\in \ttComm(\ttInd(\ttBan_R))$ and $A\to A_i$ are homotopy epimorphisms indexed by $i$ in a countable poset $I$ in $\ttComm(\ttInd(\ttBan_R))_{/A}$ whose underlying modules are in $\ttMod^{RR}_{F}(A)$. Suppose that the  system $\{A_i\}_{i\in I}$ is $\underset{i\in I}\lim$-acyclic and that the objects of $\ttInd(\ttBan_R)$ underlying  $A_i$ are   metrizable and flat over $R$. Assume the collection of functors $(-) \wotimes_{A} A_i$ is conservative. When this holds the natural functor \[\underset{i\in I}\lim \ttMod^{RR}_{F}(A_i) \to \ttMod^{RR}_{F}(A)
\]
is an equivalence of categories.
\end{thm}
This theorem can be used in the case of hypercovers. We now give a  more explicit proof in the case of covers which can be easily adapted to general posets. 
\begin{thm}Let $A\in \ttComm(\ttInd(\ttBan_R))$ and say we are given a countable collection $A\to A_i$ of objects of $\ttComm(\ttInd(\ttBan_R))_{/A}$ indexed by $i\in S$. Suppose that $A$ and $A_i$ are metrizable objects which are flat over $R$, in $\ttMod^{RR}_{F}(A)$. Suppose that each morphism $A\to A_i$ is a homotopy epimorphism and the collection of functors 
 $\ttMod^{RR}_{F}(A)\to \ttMod^{RR}_{F}(A_i)$
is conservative. Suppose that the corresponding system $A_w=A_{i_1}\wotimes_{A}\cdots \wotimes_{A}A_{i_m}$ for words $w$ in $S$ is a $\underset{w\in P}\lim$-acyclic projective system as above.  Then the canonical functor 
\[D:\ttMod^{RR}_{F}(A)\longrightarrow \underset{w\in P}\lim \ttMod^{RR}_{F}(A_w)
\]
is an equivalence of categories. 
\end{thm} 
\proof 
Given $N_w\in \ttMod^{RR}_{F}(A_w)$ and suppose that $A\to B$ is a homotopy epimorphism. 
\begin{equation}
\begin{split}
B\wotimes_{A}^{\mathbb{L}}N_w & \cong B\wotimes_{A}^{\mathbb{L}}(A_w\wotimes_{A_w}^{\mathbb{L}}N_w)\cong (B\wotimes_{A}^{\mathbb{L}}A_w)\wotimes_{A_w}^{\mathbb{L}}N_w\cong  (B\wotimes_{A}A_w)\wotimes_{A_w}^{\mathbb{L}}N_w \cong  (B\wotimes_{A}A_w)\wotimes_{A_w}N_w \\ & \cong B\wotimes_{A}N_w  
\end{split}
\end{equation}
since $A_w$ is quasi-coherent and $A_w \to B\wotimes_{A}A_w$ is a homotopy epimorphism and $N_w$ is a quasi-coherent $A_w$-module.
Since $B$ is transverse to $N_w$ for each $w$ over $A$ we have that $\underset{w\in P}\lim N_w$ is transverse to $B$ over $A$ by Lemma \ref{lem:EF}. Hence using Lemma \ref{lem:Wprod}, $\underset{w\in P}\lim N_w \in \ttMod^{RR}_{F}(A)$.

Consider the functor 
\[\ttMod^{RR}_{F}(A)\longleftarrow \underset{w\in P}\lim \ttMod^{RR}_{F}(A_w):R
\]
in the other direction defined by taking the limit. We have by Lemma \ref{lem:EF}
\[A_{v}\wotimes_{A}(\underset{w\in P}\lim N_w)\cong \underset{w\in P}\lim (A_{v}\wotimes_{A}N_w) \cong N_{v}
\]
showing that $D\circ R$ is naturally equivalent to the identity.
Using Lemma \ref{lem:EF} and Corrollary \ref{cor:AlimCons} we have
\[ \underset{w\in P}\lim(A_{w}\wotimes_{A} M) \cong (\underset{w\in P}\lim A_{w})\wotimes_{A} M \cong A\wotimes_{A}M\cong M
\]
showing that $R\circ D$ is naturally equivalent to the identity.
\endproof
\subsection{Examples of Descent}
\begin{lem}\label{lem:TateAcy}Recall that for a countable collection $\{A \to A_{i}\}_{i\in I }$ and $M \in \ttMod(A)$ we can form the usual complex 
\[\mathcal{C}^{\bullet}(M, \{A_{i}\}) = [\underset{i \in I}\prod (M \wotimes_{A} A_{i}) \longrightarrow \underset{i,j \in I}\prod (M \wotimes_{A} A_{i}\wotimes_{A} A_{j}) \longrightarrow \cdots]
\]
Suppose that the underlying objects of $A$ and $M$ in $\ttInd(\ttBan_R)$ are metrizable and flat over $R$. Suppose $M$ is transverse to all $A_{i_1}\wotimes_{A}A_{i_2}\wotimes_{A} \cdots \wotimes_{A} A_{i_n}$ in $\ttMod(A)$ and the natural morphism $A\to \mathcal{C}^{\bullet}(A, \{A_{i}\})$ is a quasi-isomorphism, then the natural morphism $M\to \mathcal{C}^{\bullet}(M, \{A_{i}\})$ is a quasi-isomorphism. If instead of the condition that $M$ is flat over $R$ we have that all the $A_{i_1}\wotimes_{A}A_{i_2}\wotimes_{A} \cdots \wotimes_{A} A_{i_n}$ are flat over $R$ then the same conclusion holds.
\end{lem}
\proof $M$ is quasi-isomorphic to $M \wotimes^{\mathbb{L}}_{A}\mathcal{C}^{\bullet}(A, \{A_{i}\})$. Using Lemma \ref{lem:tensprod}, each term \[\underset{i_1, \dots, i_n}\prod A_{i_1} \wotimes_{A}A_{i_2}\wotimes_{A} \cdots \wotimes_{A} A_{i_n}\] is transverse to $M$ over $A$.
Therefore, there is a quasi-isomorphism 
\[\Tot(\mathscr{L}^{\bullet}_{A}(M)\wotimes_{R}\mathcal{C}^{\bullet}(A, \{A_{i}\}))\cong \Tot(\mathcal{C}^{\bullet}(\mathscr{L}^{\bullet}_{A}(M), \{A_{i}\}).
\]
As our conditions guarantee that $\mathcal{C}^{\bullet}(-, \{A_{i}\})$ is an exact functor the right hand side is quasi-isomorphic to $\mathcal{C}^{\bullet}(M, \{A_{i}\})$ and $M= M \wotimes^{\mathbb{L}}_{A}A= M \wotimes^{\mathbb{L}}_{A}\mathcal{C}^{\bullet}(A, \{A_{i}\})$ is computed by the left hand side so we are done.
\endproof
\begin{rem}The non-archimedean version of this (the proof is the same) can give new settings for Tate's acyclicty theorem. We expect that the hypothesis of Lemma \ref{lem:TateAcy} will be satisfied whenever $A\to A_i$ are homotopy epimorphisms and the topological spaces associated to the $A_i$ form a cover of the topological space associated to $A$.
\end{rem}

\begin{lem}\label{lem:SteinExp}
Suppose $A \in \ttComm(\ttInd(\ttBan_R))$, $M$ is a metrizable $A$-module in $\ttInd(\ttBan_R)$ and both are flat over $R$. Say that we have 
\[
A \to \cdots \to A_3 \to A_2 \to A_1
\]
for $A_i \in \ttComm(\ttBan_{R})$. Suppose that $\underset{i}\prod A_i \longrightarrow \underset{i}\prod A_i$ defined by \[(a_1, a_2, \dots) \mapsto (a_2 - a_1, a_3 - a_2, \dots)\] is a strict epimorphism with kernel (with the induced subspace structure) isomorphic to $A$ and $M$ is transverse to each $A_i$ over $A$. Then we can conclude that $\underset{i}\prod M\wotimes_{A} A_i \longrightarrow \underset{i}\prod M\wotimes_{A}A_i$ is a strict epimorphism with kernel (with the induced structure) isomorphic to $M$.
\end{lem}
\proof 
The derived limit $\mathbb{R}\underset{i} \lim A_i$ is represented by the two term complex $[\underset{i}\prod A_i \longrightarrow \underset{i}\prod A_i]$. Therefore $A \cong \mathbb{R}\underset{i} \lim A_i$ and so $M \cong M \wotimes^{\mathbb{L}}_{A} \mathbb{R}\underset{i} \lim A_i$ which is represented by \begin{equation}
\begin{split}
\Tot(\mathscr{L}^{j}_{A}(M)\wotimes_{R}\underset{i}\prod A_i \to \mathscr{L}^{j}_{A}(M)\wotimes_{R}\underset{i}\prod A_i) & \cong \Tot(\underset{i}\prod(\mathscr{L}^{j}_{A}(M)\wotimes_{R}A_i ) \to \underset{i}\prod(\mathscr{L}^{j}_{A}(M)\wotimes_{R}A_i ) )\\ & \cong [\underset{i}\prod M\wotimes_{A} A_i \longrightarrow \underset{i}\prod M\wotimes_{A}A_i] \end{split}
\end{equation}
As we have proven that the last complex (representing $\mathbb{R}\underset{i} \lim (M \wotimes_{A} A_i)$) is isomorphic to $M$ in the derived category, we are done.
\endproof
\begin{rem}A situation where Lemma \ref{lem:SteinExp} can be used is the definition of a Stein by its defining affinoid cover. In fact, the category of modules over a Stein with an exhaustive affinoid cover which we define includes fully faithfully the category of co-admissible modules of Schneider and Teitelbaum.
\end{rem}\section{The Fargues-Fontaine Curve}
A thorough treatment of the Fargues-Fontaine curve from the point of view of Banach algebraic geometry appears in \cite{BBK2}. Therefore, we only focus on the aspects here which are relevant to the current article. Let
\[\mathbb{Z}\{(\frac{x}{r})^{\frac{1}{n}}\}= \mathbb{Z}\{\frac{x}{r}, \frac{y}{r^{\frac{1}{n}}} \}/(y^{n}-x)\cong\mathbb{Z}\{\frac{x}{r}\}\oplus  \mathbb{Z}\{\frac{x}{r}\}_{r^{-\frac{1}{n}}}\oplus \cdots \oplus \mathbb{Z}\{\frac{x}{r}\}_{r^{-\frac{(n-1)}{n}}}.
\]
For $r_2<r_1<1$, the non-expanding morphism $\mathbb{Z}\{(\frac{x}{r_1})^{\frac{1}{n}}\}\longrightarrow \mathbb{Z}\{(\frac{x}{r_2})^{\frac{1}{n}}\}$ is nuclear, being a sum of nuclear morphisms. Using the non-expanding morphisms \[\alpha_{n,m}: \mathbb{Z}\{(\frac{x}{r})^{\frac{1}{n}}\}\longrightarrow \mathbb{Z}\{(\frac{x}{r})^{\frac{1}{nm}}\}\]
we have the Banach ring $\underset{ n}\colim^{\leq 1} \mathbb{Z}\{(\frac{x}{r})^{\frac{1}{n}}\}$.
In order to study dagger or Stein versions which we expect to have better properties, we have:
\begin{con}The induced morphisms 
\[\underset{ n}\colim^{\leq 1} \mathbb{Z}\{(\frac{x}{r_1})^{\frac{1}{n}}\}\longrightarrow \underset{ n}\colim^{\leq 1} \mathbb{Z}\{(\frac{x}{r_2})^{\frac{1}{n}}\}
\]
are nuclear for all $r_2<r_1<1$.
\end{con}
Let $E$ be the field \[E= \mathbb{F}_{p}((\mathbb{Q}))=\{\underset{\gamma \in \mathbb{Q}}\sum a_{\gamma} x^{\gamma} | a_{\gamma} \in \mathbb{F}_{p} , \ \ \text{support}(a_{\gamma} ) \\ \ \ \text{well ordered} \\ \}.\] It is equipped with the valuation given by
\[v\left(\underset{\gamma \in \mathbb{Q}}\sum a_{\gamma} x^{\gamma}\right) = \min\{ \gamma : a_\gamma \neq 0 \}. \]
The associated valuation ring is 
\[\mathcal{O}_E=\mathbb{F}_{p}((\mathbb{Q}_{\geq 0}))= \{\underset{\gamma \in \mathbb{Q}_{\geq 0}}\sum a_{\gamma} x^{\gamma} | a_{\gamma} \in \mathbb{F}_{p}  , \ \ \text{support}(a_{\gamma} ) \\ \ \ \text{well ordered} \}.
\]

In Fargues-Fontaine theory one encounters a scheme $Y_E$ whose set of closed points $|Y_E|$ parametrize un-tilts of E. An un-tilt of $E$ is an isomorphism class of pairs $(F, \iota)$ where $F$ is a perfectoid field of characteristic $0$, $\iota:E \to F^{\flat}$ is a embedding of topological fields and the quotient is a finite extension. 
Here $F^{\flat}= \text{Frac}(\underset{x\mapsto x^{p}}\lim\mathcal{O}_{F}/p)$ where $\mathcal{O}_{F}$ is the ring of integers of $F$. Let $W$ denote the Witt vectors construction. Let $\mathbb{Z}_{r}$ be the Banach $\mathbb{Z}$-module which is $\mathbb{Z}$ with norm $r| \cdot |$ where $|\cdot |$ is the usual absolute value. For any $M \in \ttBan_{\mathbb{Z}}$, let $S^{\leq 1}(M)$ be the symmetric ring construction in the category $\ttBan_{\mathbb{Z}}^{\leq 1}$ consisting of Banach modules with non-expanding morphisms, i.e. $S^{\leq 1}(M)=\underset{n=0, \dots, \infty}\coprod^{\leq 1}  (M^{\wotimes^{n}_{\mathbb{Z}}} / \Sigma_n)$ where the coproduct is taken in  $\ttBan_{\mathbb{Z}}^{\leq 1}$. Consider the colimit in $\ttBan^{\leq 1}_{\mathbb{Z}}$ of the $l$-th power morphisms $x\mapsto x^{l}$ in the ring of functions on the ``closed $1$-dimensional disk of radius $r$" given by the contracting coproduct $S^{\leq 1}(\mathbb{Z}_{r})$. One then has that for each prime $p$, 
\[\left( \ \ \underset{l\in\mathbb{N}}   \colim^{\leq 1}\ \\ \  \mathbb{Z}\{(\frac{x}{r})^{\frac{1}{l}}\}\right)\wotimes_{\mathbb{Z}} \widetilde{\mathbb{Z}_{p}} \cong \lim\left( \ \ \underset{l\in\mathbb{N}}   \colim^{\leq 1}\ \\ \  \mathbb{Z}_{p}\{(\frac{x}{r})^{\frac{1}{l}}\}\right) 
\]
and this question is addressed more carefully in \cite{BBK2} using results from this article. Consider the Fr\'{e}chet completion of $W(\mathcal{O}_E)$ with respect to the semi-norms
\[|\underset{n>>-\infty}\sum[f_n]p^{n}|_{r}= \underset{n>>-\infty}\sup |f_n|p^{-rn}
.\]
 The importance of this completion is that the closed maximal ideals of the localization at $p$ are in bijection with $|Y_E|/\mathbb{Z}$ where $n\in \mathbb{Z}$ acts by by $(F, \iota) \mapsto (F, \iota \circ \phi^{n})$ where $\phi$ the $p$-th power Frobenius automorphism of $E$. In \cite{CD}, Cuntz and Deninger found a nice description of the additive group structure on the ring of $p$-typical Witt vectors of a perfect $\mathbb{F}_{p}$-algebra with basis $\mathfrak{b}$. They found it to be simply the $p$-adic completion of the free $\mathbb{Z}$-algebra with basis $\mathfrak{b}$.

\begin{lem}\label{lem:FiltLift}The natural functor 
\[F: \underset{\aleph_{1}}\ttInd(\ttBan^{K}_{R}) \longrightarrow \left(\underset{\aleph_{1}}\ttInd(\ttBan_{R})\right)^{K}
\]
is fully-faithful for any poset $K$ with cardinality less than $\aleph_{1}$.
\end{lem}
\proof Given objects $X:k \mapsto ``\underset{t \in T}\colim" X^{k}_{t}$ and $Y:k \mapsto ``\underset{s \in S}\colim" Y^{k}_{s}$ of $\underset{\aleph_{1}}\ttInd(\ttBan^{K}_{R})$, where $T$ and more importantly $S$ is an $\aleph_{1}$-filtered poset, we have 
\[\Hom(X,Y) = \underset{t\in T}\lim \ \ \underset{s\in S}  \colim \int_{k\in K}\Hom(X_t^{k}, Y_{s}^{k})
\]
where $\int_{k\in K}$ is a limit over the usual diagram used to define morphisms in diagram categories. This is a limit over a diagram with cardinality less than $\aleph_{1}$ since $K$ itself is such a diagram. It is a limit in the category of sets of a diagram of sets whose vertices are of the form $\Hom(X_t^{k}, Y_{s}^{l})$. On the other hand,
\[\Hom(FX,FY) = \int_{k\in K}\underset{t\in T}\lim \ \  \underset{s\in S}\colim \Hom(X_t^{k}, Y_{s}^{k}) \cong  \underset{t\in T}\lim\int_{k\in K} \underset{s\in S}\colim \Hom(X_t^{k}, Y_{s}^{k}).
\]
The term $\int_{k\in K} \underset{s\in S}\colim \Hom(X_t^{k}, Y_{s}^{k})$ is a limit in the category of sets over the same diagram whose vertices are of the form $\underset{s\in S}\colim\Hom(X_t^{k}, Y_{s}^{l})$ where the functor $\underset{s\in S}\colim$ has been applied to the previous diagram.  By Lemma \ref{lem:SETlimcolim} we can interchange $\int_{k\in K}$ and $\underset{s\in S}\colim$ so these different Hom-sets agree, finishing the proof.
\endproof
\begin{defn}\label{defn:ExtendColim}
Let $F:\mathbb{N} \to \underset{\aleph_{1}}\ttInd(\ttBan_R)$ be a functor such that there exists an $\aleph_{1}$-filtered category $L$ and a functor $\tilde{F}: \mathbb{N} \times L\to \ttBan^{\leq 1}_R$ such that the composition $\mathbb{N} \to (\ttBan^{\leq 1}_R)^L\to \underset{\aleph_{1}}\ttInd(\ttBan_R)$ agrees with $F$. Define $\underset{\mathbb{N}}\colim^{\leq 1} \tilde{F}$ by the composition $L\to (\ttBan^{\leq 1}_R)^\mathbb{N}\to \ttBan^{\leq 1}_R$. Define
\[\underset{\mathbb{N}} \colim^{\leq 1} F= ``\underset{L}\colim" \underset{\mathbb{N}} \colim^{\leq 1} \tilde{F}.
 \]
 This is well defined because the full subcategory of $\underset{\aleph_{1}}\ttInd(\ttBan_R)^{\mathbb{N}}$ admitting such lifts is by Lemma \ref{lem:FiltLift} actually equivalent to $\underset{\aleph_{1}}\ttInd((\ttBan^{\leq 1}_R)^{\mathbb{N}})$. This equivalence can be realized by sending $F$ to the equivalence class $[\tilde{F}]$ in  $\underset{\aleph_{1}}\ttInd((\ttBan^{\leq 1}_R)^{\mathbb{N}})$ of a lift $\tilde{F}$ and then we have \[\underset{\mathbb{N}} \colim^{\leq 1} F=\ttInd(\underset{\mathbb{N}} \colim^{\leq 1} )[\tilde{F}].\]  Therefore, under this equivalence, we simply have
 \[ \underset{\mathbb{N}} \colim^{\leq 1} =  \underset{\aleph_{1}}\ttInd(\underset{\mathbb{N}} \colim^{\leq 1} ): \underset{\aleph_{1}}\ttInd((\ttBan^{\leq 1}_R)^{\mathbb{N}}) \to \underset{\aleph_{1}}\ttInd(\ttBan_R) .
 \]
 \end{defn}
 This functor $\underset{\mathbb{N}} \colim^{\leq 1} F$ is an exact functor from a full subcategory of $ \ttInd(\ttBan_R)^{\mathbb{N}}$ to $\ttInd(\ttBan_R)$ (takes kernels to kernels) because both the ordinary non-exapanding colimit and the formal filtered colimit are exact functors. The functor we have described  commutes with $V\wotimes_{R}(-)$ for any $V\in \ttBan_{R}$.


\begin{lem} \label{lem:limcolimAbsEx} Consider a functor $K \times \mathbb{N} \to \ttBan^{\leq 1}_{R}$ where $K$ is a countable category. There exists a chain of isomorphisms: \begin{equation}\begin{split}   &\underset{i\in \mathbb{N}} \colim^{\leq 1} \underset{k\in K}\lim V_{i}^{(k)} \longrightarrow \\ &\underset{i\in \mathbb{N}} \colim^{\leq 1}  ``\underset{\psi\in \Psi}\colim" \ker [\underset{k\in K} \prod{}^{\leq 1}(V_{i}^{(k)})_{\psi(k)^{-1}} \longrightarrow \underset{k\in K} \prod{}^{\leq 1}(V_{i}^{(k)})_{\psi(k+1)^{-1}}  ] \longrightarrow \\ & ``\underset{\psi\in \Upsilon}\colim" \underset{i\in \mathbb{N}} \colim^{\leq 1} \ker [\underset{k\in K} \prod{}^{\leq 1}(V_{i}^{(k)})_{\psi(k)^{-1}} \longrightarrow \underset{k\in K} \prod{}^{\leq 1}(V_{i}^{(k)})_{\psi(k+1)^{-1}}  ] \longrightarrow \\ &``\underset{\psi\in \Upsilon}\colim" \ker [\underset{k\in K} \prod{}^{\leq 1}\underset{i\in \mathbb{N}}\colim^{\leq 1} (V_{i}^{(k)})_{\psi(k)^{-1}} \longrightarrow \underset{k\in K} \prod{}^{\leq 1}\underset{i\in \mathbb{N}}\colim^{\leq 1} (V_{i}^{(k)})_{\psi(k+1)^{-1}}] \\ & \longrightarrow \underset{k\in K}\lim\ \ \underset{i\in \mathbb{N}}\colim^{\leq 1} V_{i}^{(k)}.
\end{split}
\end{equation}
\end{lem}
\proof  The first and last morphisms are determined by the description of limits found in Corollary \ref{cor:TowerLim} in which they are shown to be isomorphisms. The second morphisms is an isomorphism as a consequence of Definition \ref{defn:ExtendColim}. The natural third morphism is an isomorphism because the non-expanding colimit functor is exact by Lemma \ref{lem:ker_colim} (see also Lemma \ref{lem:commuteFiltSSES}).
\endproof

\begin{lem}The natural morphism 
\begin{equation}\label{eqn:twoPres} \underset{l\in\mathbb{N}}   \colim^{\leq 1}  \ \ \underset{r<1} \lim\ \  \mathbb{Z}\{(\frac{x}{r})^{\frac{1}{l}}\}  \longrightarrow  \underset{r<1} \lim \ \ \underset{l\in\mathbb{N}}   \colim^{\leq 1}\ \\ \  \mathbb{Z}\{(\frac{x}{r})^{\frac{1}{l}}\} 
\end{equation}
is an isomorphism and this object of $\ttComm(\ttInd(\ttBan_{\mathbb{Z}}))$ is flat over $\mathbb{Z}$.
\end{lem}
\proof  The first statement follows immediately from Lemma \ref{lem:limcolimAbsEx} because taking  a cofinal system with $r$ within a countable set, $\underset{r<1} \lim\ \  \mathbb{Z}\{(\frac{x}{r})^{\frac{1}{l}}\}$ is $\aleph_{1}$-filtered by Corollary \ref{cor:TowerLim}. Given any $V \in \ttBan_{\mathbb{Z}}$ and any $F$ as in Definition \ref{defn:ExtendColim} admitting a suitable lift $\tilde{F}$, then $V\wotimes_{\mathbb{Z}}F$ admits $V\wotimes_{\mathbb{Z}}\tilde{F}$ as a suitable lift and therefore, the exact functor $\underset{l\in\mathbb{N}}   \colim^{\leq 1} $ commutes with $V\wotimes_{\mathbb{Z}}(-)$ and hence commutes with $V\wotimes^{\mathbb{L}}_{\mathbb{Z}}(-)$ as well. Hence it preserves flatness. $\underset{r<1} \lim\ \  \mathbb{Z}\{(\frac{x}{r})^{\frac{1}{l}}\} $ is flat because it is isomorphic to $(\underset{r<1} \lim\ \  \mathbb{Z}\{\frac{x}{r}\} )\wotimes_{\mathbb{Z}}\mathbb{Z}^{l}$ which is flat since $\underset{r<1} \lim\ \  \mathbb{Z}\{\frac{x}{r}\}$ is flat by Corollary \ref{cor:NuclearityOfFunctions}. So $\underset{l\in\mathbb{N}}   \colim^{\leq 1}  \ \ \underset{r<1} \lim\ \  \mathbb{Z}\{(\frac{x}{r})^{\frac{1}{l}}\}$ is  flat. Therefore, using the isomorphism of Equation (\ref{eqn:twoPres}) we get that $\underset{r<1} \lim \ \ \underset{l\in\mathbb{N}}   \colim^{\leq 1} \mathbb{Z}\{(\frac{x}{r})^{\frac{1}{l}}\} $ is flat as well.
\endproof
\begin{lem} The object 
\[\left( \underset{r<1} \lim \ \ \underset{l\in\mathbb{N}}   \colim^{\leq 1}\ \\ \  \mathbb{Z}\{(\frac{x}{r})^{\frac{1}{l}}\}\right)\wotimes^\mathbb{L}_{\mathbb{Z}} R
\]
is isomorphic to 
\[ \underset{r<1} \lim \ \ \underset{l\in\mathbb{N}}   \colim^{\leq 1}\ \\ \  R\{(\frac{x}{r})^{\frac{1}{l}}\}
\]
for any Banach ring $R$.
\end{lem}
\proof Since $\colim^{\leq 1}  R\{(\frac{x}{r})^{\frac{1}{l}}\}$ and $\underset{r<1} \lim  \ \ \underset{l\in\mathbb{N}}   \colim^{\leq 1}  \mathbb{Z}\{(\frac{x}{r})^{\frac{1}{l}}\}$ are flat over $\mathbb{Z}$, Lemma \ref{lem:EF} gives this result immediately.
\endproof
Notice that
\[\left( \underset{l\in\mathbb{N}}   \colim^{\leq 1}\ \\ \  \mathbb{Z}_{p}\{(\frac{x}{r})^{\frac{1}{l}}\}\right)/p\left(\underset{l\in\mathbb{N}}   \colim^{\leq 1}\ \\ \  \mathbb{Z}_{p}\{(\frac{x}{r})^{\frac{1}{l}}\}\right)\cong  \underset{l\in\mathbb{N}}   \colim^{\leq 1}\ \\ \  \mathbb{F}_{p}\{(\frac{x}{r})^{\frac{1}{l}}\}.
\]
and $  \underset{l\in\mathbb{N}}   \colim^{\leq 1}\ \  \mathbb{Z}_{p}\{(\frac{x}{r})^{\frac{1}{l}}\}$ is a strict $p$-ring. So we should show  that the natural morphism
\[  \underset{l\in\mathbb{N}}   \colim^{\leq 1} \mathbb{F}_{p}\{(\frac{x}{r})^{\frac{1}{l}}\} \longrightarrow \mathcal{O}_{E}
\]
is an isomorphism where $\mathbb{F}_{p}$ carries the residue norm from $\mathbb{Z}$.  This question is addressed in \cite{BBK2}.

\section{Appendix} As remarked above, most of this article has a non-archiedean version in the case that $R$ is non-archimedean, so in this appendix, let $R$ be a non-archimedean Banach ring. For $k$ a non-archimedean field, the standard Tate algebra representing an affinoid disk is $k\{\frac{x_1}{r_1}, \dots, \frac{x_n}{r_n}\}$. In order to compare this with the ``archimedean" disk algebra we used in this article which we denote the non-archimedean version by $R\{\frac{x_1}{r_1}, \dots, \frac{x_n}{r_n}\}_{na}$. Interestingly, Stein and Dagger algebras as defined in the introduction to Section \ref{Spaces} constructed from these two versions of disk algebras actually agree as we show in this informal Appendix.
For any $r>0$ there is an injective map $\ttComm(\ttCBorn_{R})$ 
\[R\{\frac{x_1}{r_1}, \dots, \frac{x_n}{r_n}\}\to R\{\frac{x_1}{r_1}, \dots, \frac{x_n}{r_n}\}_{na}
\]
which by density is an epimorphism. Let 
\[A_i=R\{\frac{x_1}{r_1-i^{-1}}, \dots, \frac{x_n}{r_n-i^{-1}}\}\] and \[C_i=R\{\frac{x_1}{r_1-i^{-1}}, \dots, \frac{x_n}{r_n-i^{-1}}\}_{na}.\] We have not only morphisms $A_i \subset C_i$ but also $C_{i} \subset A_{i-1}$ because for any $s<t$ we have 
\[\underset{I}\sum a_{I}s^{I}= \underset{I}\sum a_{I}(s/t)^{I}t^{I}  \leq(\underset{I}\sum (s/t)^{I})(\underset{J}\sup (a_{J}t^{J}) ).
\]
Therefore, we get isomorphisms $A=\lim A_i \cong \lim C_i=C$. Similarly, this holds for general Stein or dagger algebras as defined in the introduction to Section \ref{Spaces} described in the two different ways (using archimedian or non-archimedean disk algebras or their quotients) for a non-archimedean Banach ring $R$. 

Consider the category $D$ whose objects are pairs consisting of a sequence of objects $M_i\in \ttMod(A_i)$ and a collection of compatible isomorphisms $ M_{i}\wotimes_{A_i}A_{i-1} \to M_{i-1} $ where morphisms are the obvious thing. Similarly there is the category $D^{na}$  whose objects are pairs consisting of a sequence of objects $N_i\in \ttMod^{na}(C_i)$ and a collection of compatible isomorphisms $N_{i}\wotimes^{na}_{C_i}C_{i-1} \to N_{i-1} $ where morphisms are the obvious thing. These categories are isomorphic and if we specialize to the nuclear metrizable modules and algebras flat over $\mathbb{Z}$ we get by descent (Theorem \ref{thm:Big}) an equivalence of categories for these modules on $A$ and $C$.

\bibliographystyle{amsalpha}

\end{document}